\documentclass[a4paper,reqno,11pt]{amsart}

\usepackage[margin=1in]{geometry}
\usepackage{amsmath, amssymb,amsthm,url,mathrsfs,graphicx,amscd,amsfonts}
\usepackage[mathscr]{eucal}
\usepackage{dsfont}
\usepackage{mathtools}
\usepackage{hyperref}
\usepackage{bbm}
\usepackage{soul} 

\usepackage[utf8]{inputenc}
\usepackage{graphicx,todonotes,hyperref}
\usepackage{epstopdf}
\usepackage{inputenc}
\usepackage{amsmath}
\usepackage{xcolor}
\usepackage{xspace}
\usepackage{amssymb}
\usepackage{amsfonts}
\usepackage{graphicx}
\usepackage{mathtools}
\usepackage{amsthm}
\usepackage{tikz-cd}
\usepackage{hyperref}
\usepackage{multirow}
\usepackage{enumitem}
\usepackage[]{todonotes}
\input xy
\xyoption{all}

\newtheorem*{acknowledgements*}{Acknowledgements}

\newcommand{\s}{\mathbb{S}}
\newcommand{\z}{\mathbb{Z}}

\newcommand{\N}{\mathbb{N}}

\newcommand{\Z}{\mathbb{Z}}
\newcommand{\C}{\mathbb{C}}

\allowdisplaybreaks
\theoremstyle{definition}
\newtheorem{theorem}{Theorem}[section]
\newtheorem{prop}[theorem]{Proposition}
\newtheorem{lemma}[theorem]{Lemma}
\newtheorem{cor}[theorem]{Corollary}

\newtheorem{defn}[theorem]{Definition}

\newtheorem{rmk}[theorem]{Remark}
\newtheorem{fact}[theorem]{Fact}

\newtheorem{thm}[theorem]{Theorem}
\newtheorem*{theorem*}{Theorem}
\newtheorem*{corollary*}{Corollary}
\newtheorem*{prop*}{Proposition}
\newtheorem*{rmk*}{Remark}
\newtheorem{manualtheoreminner}{Theorem}
\newenvironment{manualtheorem}[1]{%
	\renewcommand\themanualtheoreminner{#1}%
	\manualtheoreminner
}{\endmanualtheoreminner}

\begin{document}
	
\title[Enumerating Smooth Structures on $\mathbb{C}P^3\times\mathbb{S}^k$]{Enumerating Smooth Structures on $\mathbb{C}P^3\times\mathbb{S}^k$}
	
	\author{Samik Basu}
	\address{Indian Statistical Institute, Theoretical Statistics and Mathematics Unit, Kolkata 700108, India}
	\email{samikbasu@isical.ac.in, samik.basu2@gmail.com}
	
	\author{Ramesh Kasilingam}
	\address{Indian Institute of Technology Madras, Department of Mathematics, Chennai 600036, India}
	\email{rameshkasilingam.iitb@gmail.com; rameshk@smail.iitm.ac.in}
	\author{Ankur Sarkar}
	\address{The Institute of Mathematical Sciences, A CI of Homi Bhabha National Institute, CIT Campus, Taramani, Chennai 600113, India.}
	\email{ankurimsc@gmail.com, ankurs@imsc.res.in}	

\begin{abstract}
 In this paper, we compute the concordance inertia group of the product $M \times \mathbb{S}^k$, where $M$ is a simply connected, closed, smooth 6-manifold, for $1 \leq k \leq 10$, using known low-dimensional computations of the stable homotopy groups of spheres. Specifically, for $M = \mathbb{C}P^3$, we determine the inertia group of $\mathbb{C}P^3 \times \mathbb{S}^k$ for $2 \leq k \leq 7, k \neq 6$, and establish a diffeomorphism classification of all smooth manifolds homeomorphic to $\mathbb{C}P^3 \times \mathbb{S}^k$ for $1 \leq k \leq 7$.
\end{abstract}

        \maketitle

	\section{Introduction} 
The study of smooth structures on manifolds is a central topic in differential topology, particularly in understanding when homeomorphic manifolds are diffeomorphic. This question has led to extensive research on exotic smooth structures and their classification. In this paper, we build upon the computations from \cite{SBRKAS} and examine the smooth structures on product manifolds of the form \( M \times \mathbb{S}^k \), where \( M \) is a simply connected, closed, smooth 6-manifold. A key motivation for this study is the classification of diffeomorphism types of smooth manifolds homeomorphic to \( \mathbb{C}P^3 \times \mathbb{S}^k \) for various values of \( k \).

A fundamental tool in diffeomorphism classification is the concordance structure set \( \mathcal{C}(M\times\mathbb{S}^k) \), which captures the different smooth structures on \( M\times\mathbb{S}^k \) up to concordance. Following the framework established by Kirby and Siebenmann \cite[Theorem 10.1, Essay IV]{kirby}, the set \( \mathcal{C}(N) \) is in bijection with the homotopy classes of maps \( [N, Top/O] \) for any closed, oriented, smooth manifold $N$ of dimension $n\geq 5$, where \(Top/O \) is the homotopy fiber of the map \( BO \to BTop \). Using this correspondence along with the fiber sequence 
\[
PL/O \to Top/O \to Top/PL,
\]
we derive the following computation of \( \mathcal{C}(M \times \mathbb{S}^k) \) for any \( k \).

\begin{manualtheorem}{A} \textit{Let \( M \) be a simply connected, closed, smooth 6-manifold. Then, for any \( k \geq 1 \),}
\[
\mathcal{C}(M \times \mathbb{S}^k) \cong [M \times \mathbb{S}^k, PL/O] \oplus \tilde{H}^3(M \times \mathbb{S}^k; \mathbb{Z}/2).
\]
\textit{(Theorem A follows from Corollary \ref{cor3.10} and Proposition \ref{TopmodO}.)}
\end{manualtheorem}
For the specific cases \( k = 1 \) and \( k = 2 \), this result can also be obtained from \cite[Theorem A]{AKupers23}. In addition to the above theorem, we provide explicit computations of \( C(\mathbb{C}P^3 \times \mathbb{S}^k) \) for \( 1 \leq k \leq 10 \) (see Proposition \ref{CP3computation}).

To further understand exotic smooth structures, we investigate the inertia group \( I(N) \), which consists of homotopy \( n \)-spheres \( \Sigma \) such that \( N \# \Sigma \) is diffeomorphic to \( N \) via an orientation-preserving diffeomorphism. Since determining \( I(M\times\mathbb{S}^k) \) in full generality is difficult, we instead compute a subgroup of \( I(M \times \mathbb{S}^k) \), namely the \textit{concordance inertia group} \( I_c(M \times \mathbb{S}^k) \), for \( 1 \leq k \leq 10 \), by analyzing the stable top cell attaching map of a simply connected, closed, smooth 6-manifold \( M \) (see Theorem \ref{prop2.9}).

Additionally, for the manifold \( M = \mathbb{C}P^3 \), by utilizing the calculations of \( \mathcal{C}(\mathbb{C}P^3 \times \mathbb{S}^k) \), we achieve the computation of the inertia group \( I(\mathbb{C}P^3 \times \mathbb{S}^k) \) for \( 2 \leq k \leq 7 \), \( k \neq 6 \). 
Note that the case $k=1$ was previously determined by Brumfiel, who showed that $I(\mathbb{C}P^3\times\mathbb{S}^1)=\mathbb{Z}/7$ \cite{gw}.
\begin{manualtheorem}{B}\noindent
\begin{itemize}
\item[(1)] \( I(\mathbb{C}P^3 \times \mathbb{S}^k) = 0 \) for \( k = 2 \) and \( k = 3 \).
\item[(2)] \( I(\mathbb{C}P^3 \times \mathbb{S}^k) \cong \mathbb{Z}/3 \) for \( k = 4 \) and \( k = 7 \).
\item[(3)] \( I(\mathbb{C}P^3 \times \mathbb{S}^5) \cong \mathbb{Z}/62 \).
\end{itemize}
\textit{(See Theorems \ref{inertiagroupCP3S3}, \ref{inertiacp3s4}, \ref{inertiacp3s5} and Corollaries \ref{InertiagroupMS7}, \ref{inertiagroupcp3s2}.)}
\end{manualtheorem}
We use \( \mathcal{S}(M) \) to denote the set of orientation-preserving diffeomorphism classes of smooth manifolds that are homeomorphic to a given manifold \( M \). Using the above calculations of the inertia group, we determine \( \mathcal{S}(\mathbb{C}P^3 \times \mathbb{S}^k) \) in the following.

\begin{manualtheorem}{C}
    \noindent
    \begin{itemize}
        \item[(a)] $\mathcal{S}(\mathbb{C}P^3\times\mathbb{S}^1)= \{[(\mathbb{C}P^3\times\mathbb{S}^1)\# \Sigma^7] : [\Sigma^7]\in \mathbb{Z}/4\subset \Theta_7\}.$
        \item[(b)] $\mathcal{S}(\mathbb{C}P^3\times\mathbb{S}^2)=\{[\mathbb{C}P^3\times\mathbb{S}^2], [(\mathbb{C}P^3\times\mathbb{S}^2)\#\Sigma]\},$ where $\Sigma$ is the exotic $8$-sphere.
        \item[(c)] $\mathcal{S}(\mathbb{C}P^3\times\mathbb{S}^3)=\{[(\mathbb{C}P^3\times\mathbb{S}^3)\#\Sigma^9]: [\Sigma^9]\in \Theta_9\}.$
        \item[(d)] $\mathcal{S}(\mathbb{C}P^3\times\mathbb{S}^4)= \{[(\mathbb{C}P^3\times\mathbb{S}^4)\# \Sigma^{10}] : [\Sigma^{10}] \in \mathbb{Z}/2\in \Theta_{10}\}.$  
        \item[(e)] $\mathcal{S}(\mathbb{C}P^3\times\mathbb{S}^5)=\{[(\mathbb{C}P^3\times\mathbb{S}^5)\#\Sigma^{11}]: [\Sigma^{11}] \in \mathbb{Z}/{2^4}\subset \Theta_{11}\}.$
        \item[(f)] $\mathcal{S}(\mathbb{C}P^3\times\mathbb{S}^6)=\{[\mathbb{C}P^3\times\mathbb{S}^6],~ [\widetilde{N}]\},$ where the manifold $\widetilde{N}$ is determined by the normal invariant. 
        \item[(g)] $\mathcal{S}(\mathbb{C}P^3\times\mathbb{S}^7)=\{[\mathbb{C}P^3\times\mathbb{S}^7]\}.$ 
    \end{itemize}
   \textit{(Theorem C is a combination of Theorem \ref{classificationCP3S1} {(ii)}, \ref{classificationcp3s2}, \ref{ClassificationCP3S3}, \ref{classificationCP3S4}, \ref{classificationcp3s5}, \ref{classificationCP3S6} and \ref{classificationCP3S7}.)}
\end{manualtheorem}

	\subsection{Notation}
	\begin{itemize}

		\item Let $O_n$ be the orthogonal group, $PL_n$ is the simplicial group of piece-wise linear homeomorphisms of $\mathbb{R}^n$ fixing origin, $Top_n$ is the group of self homeomorphisms of $\mathbb{R}^n$ preserving origin, and $G_n$ be the set of homotopy equivalences. Denote by $O=\varinjlim~O_n$, $PL= \varinjlim PL_n$, $Top=\varinjlim Top_n,$ and $G=\varinjlim~ G_n$ \cite{kuiperlashof66,lashofrothenberg65,milgram}.
		\item Let $G/O$ be the homotopy fiber of the canonical map $BO\to BG$ between the classifying spaces for stable vector bundles and stable spherical fibrations \cite[\S2 and \S3]{milgram}, and $G/PL$ be the homotopy fiber of the canonical map $BPL\to BG$ between the classifying spaces for stable piecewise linear microbundles and stable spherical fibrations \cite{gw2, rudyak15}.
        
		\item  There are standard fiber sequences $\cdots\to\Omega G/Top\xrightarrow{w} Top/O\xrightarrow{\psi} G/O\xrightarrow{\phi} G/Top$ and $O\xrightarrow{\Omega J}G\xrightarrow{j}G/O\xrightarrow{i}BO\xrightarrow{J}BG$ \cite{gw} where $J:\pi_n(SO)\to\pi_n^s$ denotes the classical $J$-homomorphism.
		\item We write $X_{(p)}$ for localization of the space $X$ at prime $p.$ We use the notation $f_{(p)}$ when a map $f$ is localized at $p.$ 
		\item For an infinite loop space $X$ we use the small letter $x$ to denote the connective spectrum such that $\Omega^\infty (x) \simeq X$. We use this notation to define the spectra $g$, $o$, $pl$, $pl/o$, $g/o$, $g/pl$. 
            \item $H G$ denotes the Eilenberg-Maclane spctrum for the abelian group $G.$
		\item The notation $\{ -, - \}$ is used to denote the stable homotopy classes of maps between spectra. 
		\item The notation $\tau_{\leq m}$ is reserved for the $m^{th}$ Postnikov section. It satisfies $\pi_i \tau_{\leq m}(E) = \pi_i(E)$ for $i\leq m$ and $0$ if $i>m$. The notation $\tau^{>m}$ refers to the $m$-connected cover, which is also the fiber of $X \to \tau_{\leq m} X$. 
		\item  $\Sigma^n M(A)$ denotes the Moore spectrum.
		\item  The notations $I$, $d$ and $p$ are reserved for the indexing set $I=\{1,2,\dots,k\}$, the degree one map $f_{M^n}:M^n\to \s^n$ and  the pinch map $p:\s^n \to \bigvee\limits_{i\in I}\s^n_i$ respectively, where $n\in \N$.
        \item $f:X\twoheadrightarrow Y$ denotes onto map.
        \item $\mathcal{E}(X)$ denotes the group of self-homotopy equivalence of $X.$
        \item $h_{\Sigma}: X\#\Sigma \to X$ denotes the canonical homeomorphism corresponding to the exotic sphere $\Sigma.$ 
	\end{itemize}
	
	\subsection{Acknowledgements} The third author wishes to express sincere gratitude to Wolfgang Lück for kindly clarifying various questions related to surgery theory, and to Parameswaran Sankaran for his invaluable assistance in computing of Pontrjagin classes of virtual bundles.


    	\section{Concordance smooth structure on the product of \texorpdfstring{$6$}{}-manifold \texorpdfstring{$M$}{} with standard \texorpdfstring{$k$}{}-sphere} 
        
\subsection{Preliminaries} In this section, we first recall some preliminaries and we study the concordance smooth structure set of $M\times\mathbb{S}^k.$
	
	Our goal is to enumerate smooth structures via the study of the \textit{concordance structure set} of the manifolds $M \times \s^k.$ 	
	Given $\Sigma\in \Theta_{n},$ the canonical homeomorphism $h_{\Sigma}:M\#\Sigma \to M $ induces a class $[M\#\Sigma]$ in $\mathcal{C}(M).$ This allows an extension of the action of $\Theta_n$ from diffeomorphism classes of $M$ to concordance classes $M$. The stabilizer of this action is known as the \textit{concordance inertia group}. More precisely,
	\begin{defn} 
		Let $N$ be a closed, oriented, smooth $n$-manifold. The \textit{concordance inertia group} $I_c(N)$ of $N$ is defined to be the subgroup of $I(N)$ containing all exotic $n$-spheres $\Sigma^n$ such that $N\#\Sigma^n$ is concordant to $N.$
	\end{defn} 
	
	Hence, based on the result of Kirby and Siebenmann result, for $n\geq 5$,  the \textit{concordance inertia group} $I_c(N)$ can be identified with $\mathrm{ker}\left(\Theta_n\xrightarrow{(f_{N})^*}\mathcal{C}(N)\right),$ 
 where $\mathcal{C}(\s^n)\cong \Theta_n\cong [\s^n,Top/O]$, and $f_{N}: N \to \s^n$ is the degree-one map associated with the manifold $N.$

 Since $\Sigma (M\times N)\simeq \Sigma M\vee \Sigma N\vee \Sigma (M\land N),$  then along any $H$ space $Y,$ there is a splitting of the short exact sequence
           \begin{equation}\label{note2.3}
			0\to[ M\land N,Y]\xrightarrow{p^*}[M\times N,Y]\xrightarrow{i^*}[M\vee N,Y]\to0,
		\end{equation}
		induced from 
		$$M\vee N\xhookrightarrow{i}M\times N\xrightarrow{p} M\times N/M\vee N\simeq M\land N.$$ 
	
As a consequence of Kirby and Siebenmann's result and the split short exact sequence (\ref{note2.3}), we obtain the following corollary.
   \begin{cor}\label{concorMtimesSk}
       For any closed oriented smooth manifold $M,$
       $$\mathcal{C}(M\times\mathbb{S}^k)\cong [M,Top/O]\oplus\pi_k(Top/O)\oplus[\Sigma^k M,Top/O],$$ where $\mathit{dim}(M)+k\geq 5.$
   \end{cor}
 
	We briefly recall \cite[Lemma 2.2]{SBRKAS}.

        \begin{lemma}\label{concordanceinertia}
	Let $M$ be a closed, oriented, smooth $n$-manifold and $\Sigma^k f_{M}: \Sigma^k M\to \s^{n+k}$ be the $k$-fold suspension of the degree one map $f_{M}:M\to \s^n.$ Then the concordance inertia group 
		$I_c\left(M\times\s^k\right)$ is equal to $\mathrm{ker}\left(\Theta_{n+k}\xrightarrow{(\Sigma^k f_{M})^*}[\Sigma^k M,Top/O]\right)$ for $n + k\geq 5.$
        \end{lemma}

In order to make explicit computation, we need to know some stable homotopy groups of mapping cones.
    $$\pi_5^s(\mathbb{C}P^2)=\mathbb{Z}/{12}\{\pi\} = \mathbb{Z}/4\{i' \circ 2\nu\}\oplus \mathbb{Z}/3\{i'\circ\alpha_1\},$$
	
    \[\pi_3^s(\Sigma^2 M(\Z/{p^r})) = \begin{cases}\Z/2\{\Sigma^2 i_{2^r} \circ\eta \} &\mbox{if $p=2$ }\\ 0 &\mbox{if $p$ is odd prime,}\end{cases}\]
    
\[\pi_5^s(\Sigma^2 M(\mathbb{Z}/{p^r}))= \begin{cases}
0&\mbox{if $p>3$ is odd prime}\\
\mathbb{Z}/3\{\Sigma^2 i_{3^r}\circ\alpha_1\}&\mbox{if $p=3$}\\
\mathbb{Z}/\gcd(8,2^r)\{\Sigma^2 i_{2^r}\circ \nu\}\oplus \mathbb{Z}/2\{\widetilde{\eta_{2^r}^2}\} &\mbox{if $p=2$},
\end{cases}\]

\[\pi_5^s(\Sigma^3 M(\mathbb{Z}/{p^r}))=\begin{cases}
0 &\mbox{if $p\geq 3$ is prime}\\
\Z/4\{\tilde{\eta}_{2}\} &\mbox{if $p=2$ and $r=1$}\\
\Z/2\{\tilde{\eta}_{2^r}\}\oplus \Z/2\{\Sigma^3 i_{2^r}\circ \eta^2\} &\mbox{if $p=2$ and $r\geq2$},
\end{cases}\]
In the above $i': \mathbb{S}^2\hookrightarrow \mathbb{C}P^2,~i_{p^r}:\mathbb{S}^0\hookrightarrow M(\mathbb{Z}/{p^r})$ are inclusions, $\tilde{\eta}_{2^r}$ (respectively, $\widetilde{\eta_{2^r}^2}$) refers to a homotopy class which gives $\eta$ (respectively, $\eta^2$) after composition onto $\mathbb{S}^3$ and $\nu$ denotes the generator of $(\pi_{3}^s)_{(2)}.$

Let $C_{i_{2^r}\circ\eta}$ be the cofiber of the composition $\mathbb{S}^3\xrightarrow{\eta}\mathbb{S}^2\xhookrightarrow{\Sigma^2 i_{2^r}}\Sigma^2 M(\mathbb{Z}/{2^r})$. Then 
\[\pi_5^s(C_{i_{2^r}\circ\eta})=\begin{cases}
\mathbb{Z}/2\{i^r\circ\widetilde{\eta_2^2}\}&\mbox{if $r=1$}\\
\mathbb{Z}/2\{i^r\circ \Sigma^2 i_{2^r}\circ 4\nu\}\oplus \mathbb{Z}/2\{i^r\circ\widetilde{\eta_{2^r}^2}\}&\mbox{if $r=2$}\\
\mathbb{Z}/4\{i^r\circ \Sigma^2 i_{2^r}\circ 2\nu\}\oplus \mathbb{Z}/2\{i^r\circ\widetilde{\eta_{2^r}^2}\}&\mbox{if $r\geq 3$}
\end{cases}\]
where $i^r: \Sigma^2 M(\mathbb{Z}/{2^r})\hookrightarrow C_{i_{2^r}\circ\eta}$ is the inclusion.

We also recall higher order Bockstein $\beta_r: H^n(X;\mathbb{Z}/2)\to H^{n+1}(X;\mathbb{Z}/2),$ used to detect $2$-torsion elements in $H_*(X;\mathbb{Z}/2).$   

$\beta_1: H^n(X;\mathbb{Z}/2)\to H^{n+1}(X;\mathbb{Z}/2)$ is the usual Bockstein homomorphism associated to the short exact sequence $0\to \mathbb{Z}/2\to \mathbb{Z}/4\to \mathbb{Z}/2\to 0.$ For $r\geq 2, ~\beta_r: H^n(X;\mathbb{Z}/2)\to H^{n+1}(X;\mathbb{Z}/2)$ is defined on $\cap_{i< r} \mathrm{ker}(\beta_i)$ \cite{harper}. Moreover, if $a\in H^*(M;\mathbb{Z}/2)$ comes from free integral homology class, then $a\in \mathrm{ker}(\beta_r)$ and $a\notin \mathrm{Im}(\beta_r)$ for each $r\geq 1$ \cite{tangora}.

\subsection{Stable homotopy type of \texorpdfstring{$M$}{}} Now, we compute the concordance inertia group and concordance smooth structure set of $M\times \s^k$, where $M$ is a simply connected, closed, smooth $6$-manifold and $ k\geq 1.$ By Poincar\'e\xspace duality and the Universal coefficient theorem, the homology of $M$ is given by
	\begin{equation} \label{eqn1}
	H_q(M;\z) =
	\begin{cases}
	\mathbb{Z}^{m}\oplus \bigoplus\limits_{t=1}^{s_1} \z/{p_{t}^{r_t}}\oplus\bigoplus\limits_{t=1}^{s_2}{\mathbb{Z}/{2^{r_t}}} & \text{{\rm if}~q=2}\\
	\mathbb{Z}^{2d}\oplus \bigoplus\limits_{t=1}^{s_1} \z/{p_{t}^{r_t}}\oplus\bigoplus\limits_{t=1}^{s_2}{\mathbb{Z}/{2^{r_t}}} & \text{{\rm if}~q=3}\\
	\mathbb{Z}^{m} & \text{{\rm if}~q=4}\\
	\mathbb{Z} & \text{{\rm if}~ q=0,6}\\
	0 & \text{otherwise}\\
	\end{cases}       
	\end{equation}
	where $m, d$ are non-negative integers, $p_t$ are odd prime numbers and $r_t\geq 1$ are integers. It follows from \cite[Theorem 1] {wall}that $M \cong  N \# \big (\#_{l=1}^{d} \s^3\times \s^3 \big )$ where $N$ is a closed, oriented and simply connected manifold satisfying $H_q(N;\z)= H_q(M;\z)$ for $q\neq 3$, $H_3(N;\z)= \bigoplus\limits_{t=1}^{s_1} \z/{p_{t}^{r_t}}\oplus\bigoplus\limits_{t=1}^{s_2}{\mathbb{Z}/{2^{r_t}}}$. Moreover, $N$ is unique up to oriented diffeomorphism. By fixing minimal CW-structure on $N$, it can easily be seen that $N \# \big (\#_{l=1}^{d} \s^3\times \s^3 \big )$ can be obtained by a cofiber sequence 
	\begin{equation*} 
	\s^5\xrightarrow{f} N^{(5)}\vee (\bigvee_{l=1}^{2d} \s^3)\xhookrightarrow{} N \# \big(\#_{l=1}^{d} \s^3\times \s^3 \big ),
	\end{equation*}
	where the attaching map $f:\s^5\to N^{(5)}\vee (\bigvee\limits_{l=1}^{2d} \s^3)$ is homotopic to the composition $(f'\vee \omega)\circ c : \s^5\to N^{(5)}\vee (\bigvee\limits_{l=1}^{2d} \s^3)$. Here, the map $c:\s^5\to \s^5\vee \s^5$ is the suspension comultiplication, $\omega : \s^5\to \bigvee\limits_{l=1}^{2d} \s^3$ and $f':\s^5\to N^{(5)}$ are the attaching maps of the top cell in $\#_{l=1}^{d}  (\s^3\vee \s^3)$ and $N$, respectively. Now by using the fact that $\Sigma \omega $ is null homotopic, we obtain that $M$ is stably homotopic to 
    \begin{equation}\label{stableM^6}
        M\simeq N \vee \bigvee\limits_{l=1}^{2d} \mathbb{S}^3.
    \end{equation}

\subsection{The top cell attaching map of \texorpdfstring{$1$}{}-connected \texorpdfstring{$6$}{}-manifolds}
Throughout this subsection, $M$ is defined as a simply connected, smooth, closed $6$-dimensional manifold with its third Betti number $b_3(M)=0$ and no $2$-torsion in its third homology group $H_3(M;\mathbb{Z})$. Based on the homology given in (\ref{eqn1}), the $3$-skeleton of $M$ can be described as follows:
\begin{equation}\label{3skeletonM^6}
M^{(3)}\simeq \bigvee_{i=1}^m \mathbb{S}^2 \vee \bigvee_{j=1}^{s_1}\Sigma^2 M(\mathbb{Z}/{p_j^{r_j}}) \vee \bigvee_{j=1}^{s_2}\Sigma^2 M(\mathbb{Z}/{2^{r_j}}) \vee \bigvee_{j=1}^{s_1+s_2}\mathbb{S}^3
\end{equation} 
The stable homotopy type of the $4$-skeleton $M^{(4)}$ of $M$ is determined by the attaching map $\phi_{4}:\bigvee\limits_{i=1}^{m+s_1+s_2}\s^3\to M^{(3)}.$ From \eqref{3skeletonM^6}, we note that the stable nontrivial components of $\phi_4$ may include:
\begin{itemize}
    \item A non-zero multiple of $\eta$ onto some of the copies of $\mathbb{S}^2.$
    \item A non-trivial multiple of $i_{2^r}\circ\eta$ onto one of the copies of $\Sigma^2 M(\mathbb{Z}/{2^{r_j}}).$
    \item A degree $p^{r_j}_j, 1\leq j\leq s_1+s_2$ onto each copy of $\mathbb{S}^3.$
\end{itemize}

Consequently, based on (\ref{eqn1}), the stable homotopy type of the $4$-skeleton $M^{(4)}$ can be expressed as
\begin{align}\label{4skeletonM^6}
M^{(4)}=  M^{(5)}\simeq \bigvee_{i=1}^c \mathbb{C}P^2\vee\bigvee_{i=1}^{m-c}\mathbb{S}^2\vee\bigvee_{j=\tilde{s}_1+1}^{s_1} \Sigma^2 M(\mathbb{Z}/{p_j^{r_j}})\vee \bigvee_{j=1}^{\tilde{s}_1}\Sigma^2 M(\mathbb{Z}/{3^{r_j}}) \vee \bigvee_{j=1}^{s_2-l}\Sigma^2 M(\mathbb{Z}/{2^{r_j}})\nonumber\\                
\vee \bigvee_{j=1}^{\tilde{s}_1}\Sigma^3 M(\mathbb{Z}/{3^{r_j}})\vee \bigvee_{j=\tilde{s}_1+1}^{s_1}\Sigma^3 M(\mathbb{Z}/{p_j^{r_j}})\vee \bigvee_{j=1}^{s_2}\Sigma^3 M(\mathbb{Z}/{2^{r_j}})\vee \bigvee_{w=1}^m \mathbb{S}^4\vee \bigvee_{j=1}^l C_{i_{2^{r_j}}\circ\eta}, 
\end{align}
where $0\leq c\leq m,~ 0\leq l\leq s_2,$ and $0\leq \tilde{s}_1\leq s_1$ are integers, $p_j> 3$ are odd primes, and $C_{i_{2^r}\circ\eta}$ is the cofiber of the composition $\mathbb{S}^3\xrightarrow{\eta}\mathbb{S}^2\xhookrightarrow{i_{2^r}}\Sigma^2 M(\mathbb{Z}/{2^r})$. We now examine each component of the attaching map  $\phi_{6}:\s^5\to M^{(5)}$ of the top cell. 

Since the steenrod operation $Sq^4: H^2(M;\mathbb{Z}/2)\to H^6(M;\mathbb{Z}/2)$ is trivial, stably the map $\phi_6:\mathbb{S}^5\to M^{(5)}$ can be written as follows:	
\begin{align}\label{phi_6}
\phi_6 = \sum_{i=1}^c \left(a_i (i' \circ 2\nu) +a_i' (i'\circ\alpha_1)\right) + \sum_{i=1}^{m-c} (2 b_i \nu + b_i' \alpha_1)+ \sum_{j=1}^{\tilde{s}_1} c_j (\Sigma^2 i_{3^{r_j}}\circ\alpha_1)  \nonumber \\
+\sum_{j=1}^{s_2-l}\left(2 x_j (\Sigma^2 i_{2^{r_j}}\circ 2^{\zeta(r_j)}\nu)+\tilde{x}_j \widetilde{\eta_{2^{r_j}}^2}\right)+\sum_{j=1}^{s_2} \left(y_j \tilde{\eta}_{2^{r_j}}+ y_j'(\Sigma^3 i_{2^{r_j}}\circ\eta^2)\right)+\sum_{w=1}^m z_{w} \eta \nonumber\\
+\sum_{j=1}^l \left(d_j (i^{r_j}\circ \widetilde{\eta_{2^{r_j}}^2})+ 2 d^{r_j} (i^{r_j}\circ \Sigma^2 i_{2^{r_j}}\circ 2^{1+\delta_{r_j}}\nu) \right),
\end{align}
where  $a_i, b_i\in \mathbb{Z}/4; a_i', b_i', c_j\in \mathbb{Z}/3;~d_j, \tilde{x}_j, z_{w}\in \mathbb{Z}/2,~d^{r_j}=0$ if $r_j=1$ and $d^{r_j}\in \mathbb{Z}/2$ if $r_j=2$ and $d^{r_j}\in \mathbb{Z}/4$ if $r_j\geq2,$ while $\delta_{r_j}=1$ if $r_j=2$ and equals to $0$ for all $r_j\neq 2;$ the function $\zeta(r_j)=2,1,0$ for $r_j=1,2$ and $\geq 3$ respectively, $x_j\in \mathbb{Z}/2, \mathbb{Z}/4$ and $\mathbb{Z}/8$ for $r_j=1,2$ and $r_j\geq 3$ respectively; $y_j\in \mathbb{Z}/4$ and $\mathbb{Z}/2$ for $r_j=1$ and $r_j\geq2$ respectively, $y_j'\in (1-\delta_{r_j})\mathbb{Z}/2 ~(r_j=1 ~\text{applies} ~y_j=0).$

We now introduce the following conditions on cohomology operations:

\textbf{Condition A:} For any $u, v\in H^3(M;\mathbb{Z}/2)$ satisfying $\Theta(u)\neq 0,~\Theta(v)=0,~u+v\notin \mathrm{Im}(\beta_s)$ for any $s\geq 1,$  while there exist $u',~v'\in H^3(M;\mathbb{Z}/2)$ with $\Theta(u')\neq 0,~\Theta(v')=0$ such that $\beta_r(u'+v')\neq 0$ for some $r.$

\textbf{Condition B:} There exist $u, v \in H^3(M;\mathbb{Z}/2)$ satisfying $\Theta(u)\neq 0,~\Theta(v)=0$ and $u+v\in \mathrm{Im}(\beta_r)$ for some $r.$

\textbf{Condition C:} For every $u\in H^4(M;\mathbb{Z}/2)$ with $Sq^2(u)\neq 0,$ any $v\in \mathrm{ker}(Sq^2)$ satisfies $u+v\notin \mathrm{Im}(\beta_r)$ for any $r\geq 1.$

\textbf{Condition D:} There exist $u,v\in H^4(M;\mathbb{Z}/2)$ such that $Sq^2(u)\neq 0,~ Sq^2(v)=0,~ u+v \in \mathrm{Im}(\beta_r)$ for some $r.$

\underline{\textbf{$M$ is a spin manifold}}

 If $Sq^2: H^4(M;\mathbb{Z}/2)\to H^6(M;\mathbb{Z}/2)$ is trivial, then $y_j, z_w$ in \eqref{phi_6} must be zero. Consequently, \eqref{phi_6} simplifies to
\begin{align}\label{spin}
\phi_6 = \sum_{i=1}^c \left(a_i (i' \circ 2\nu) +a_i' (i'\circ\alpha_1)\right) + \sum_{i=1}^{m-c} (2 b_i \nu + b_i' \alpha_1)+ \sum_{j=1}^{\tilde{s}_1} c_j (\Sigma^2 i_{3^{r_j}}\circ\alpha_1)\nonumber \\
+\sum_{j=1}^{s_2-l}\left(2 x_j (\Sigma^2 i_{2^{r_j}}\circ 2^{\zeta(r_j)}\nu)+\tilde{x}_j \widetilde{\eta_{2^{r_j}}^2}\right)+\sum_{j=1}^{s_2} y_j'(\Sigma^3 i_{2^{r_j}}\circ\eta^2)\nonumber\\
+\sum_{j=1}^l \left(d_j (i^{r_j}\circ \widetilde{\eta_{2^{r_j}}^2})+ 2 d^{r_j} (i^{r_j}\circ \Sigma^2 i_{2^{r_j}}\circ 2^{1+\delta_{r_j}}\nu) \right)  .  
\end{align}

\begin{itemize}
    \item If the secondary cohomology operation $\Theta$ acts trivially on $H^3(M;\mathbb{Z}/2),$ then, by \cite{harper} and \cite[Lemma 3.3]{LiZhu24}, all $d_j, \tilde{x}_j, y_j'$ in \eqref{spin} vanish, reducing \eqref{spin} to 
\begin{align}\label{spin_noeta2}
\phi_6 = \sum_{i=1}^c \left(a_i (i' \circ 2\nu) +a_i' (i'\circ\alpha_1)\right)  + \sum_{i=1}^{m-c} (2 b_i \nu + b_i' \alpha_1)+ \sum_{j=1}^{\tilde{s}_1} c_j (\Sigma^2 i_{3^{r_j}}\circ\alpha_1) \nonumber \\
+\sum_{j=1}^{s_2-l}2 x_j (\Sigma^2 i_{2^{r_j}}\circ 2^{\zeta(r_j)}\nu)+\sum_{j=1}^{l} 2 d^{r_j} (i^{r_j}\circ \Sigma^2 i_{2^{r_j}}\circ 2^{1+\delta_{r_j}}\nu)  . 
\end{align}
\item If the secondary cohomology operation $\Theta$ acts non-trivially on $H^3(M;\mathbb{Z}/2),$ then at least one of $d_j,\tilde{x}_j, y_j'$ in \eqref{spin} is non-zero.

\begin{itemize}
    \item[(a)] If $M$ satisfies Condition A, then $y_j'=1$ for some $1\leq j\leq s_2$, while $d_j=0$ for all $1\leq j\leq l$ and $\tilde{x}_j=0$ for all $1\leq j\leq s_2-l.$ Then by \cite[Lemma 4.1]{LiZhu24}, \eqref{spin} further reduces to 
\begin{align}\label{spin_conditionA}
\phi_6 = \sum_{i=1}^c \left(a_i (i' \circ 2\nu) +a_i' (i'\circ\alpha_1)\right)  + \sum_{i=1}^{m-c} (2 b_i \nu + b_i' \alpha_1) + \sum_{j=1}^{\tilde{s}_1} c_j (\Sigma^2 i_{3^{r_j}}\circ\alpha_1)  \nonumber \\
+\sum_{j=1}^{s_2-l} 2 x_j (\Sigma^2 i_{2^{r_j}}\circ 2^{\zeta(r_j)}\nu) + \Sigma^3 i_{2^{r_{j_0}}}\circ\eta^2+\sum_{j=1}^l 2 d^{r_j} (i^{r_j}\circ \Sigma^2 i_{2^{r_j}}\circ 2^{1+\delta_{r_j}}\nu),  
\end{align}
where $j_0$ is the index such that $r_{j_0}=\max\{r_j: y_{j}'=1,~1\leq j\leq s_2\} =\max\{r_j: \beta_{r_j}(u'+v')\neq0, ~1\leq j\leq s_2\}.$
\item[(b)] If $M$ satisfies Condition B, then either $d_{j}=1$ or $\tilde{x}_j=1$ for some $j.$ By \cite[Lemma 4.1]{LiZhu24}, $y_j'=0$ for all $1\leq j\leq s_2$. 
\begin{itemize}
    \item[(i)] If $d_{j}=1$  for some $j$ and $\tilde{x}_j=0$ for all $1\leq j\leq s_2-l,$ then using \cite[Lemma 4.2]{LiZhu24} we may assume that $d_{j_1}=1$ and $d_{j}=0$ for $j\neq j_1,$ where $j_1$ is the index such that $r_{j_1}=\mathrm{min}\{r_j: d_{j}=1 ~\mathrm{for} ~1\leq j\leq l\}= \mathrm{min}\{r_j: u+v \in \mathrm{Im}(\beta_{r_j}) ~\mathrm{for}~ 1\leq j\leq l\}.$
    \item[(ii)] If $\tilde{x}_j=1$ for some $j$ and $d_j=0$ for all $1\leq j\leq l,$  then by \cite[Lemma 4.1]{LiZhu24} we may assume $\tilde{x}_{j_2}=1,$ and $\tilde{x}_j=0$ for all $j\neq j_2,$ where $j_2$ is the index such that $r_{j_2}=\mathrm{min}\{r_j:\tilde{x}_j=1 ~\mathrm{for}~ 1\leq j\leq s_2-l \}= \mathrm{min}\{r_j: u+v \in \mathrm{Im}(\beta_{r_j}) ~\mathrm{for}~ 1\leq j\leq s_2-l\}.$
    \item[(iii)] Suppose $d_{j_3}=1$ and $\tilde{x}_{j_4}=1$ for some $j_3, j_4.$ If $\mathrm{min}(r_{j_3}, r_{j_4})=r_{j_3},$ then by \cite[Lemma 4.1]{LiZhu24}, we may assume  $d_{j_3}=1$ and $d_{j}=0$ for $j\neq j_3,$ and $\tilde{x}_{j}=0$ for all $1\leq j\leq s_2-l.$ If $\mathrm{min}(r_{j_3}, r_{j_4})=r_{j_4},$ then again \cite[Lemma 4.1]{LiZhu24} implies $\tilde{x}_{j_4}=1,$ and $\tilde{x}_j=0$ for all $j\neq j_4$ and $d_j=0$ for all $1\leq j\leq l.$   
\end{itemize}

Define \begin{equation}\label{spin_conditionB_coefficient}
  r=\begin{cases}
r_{j_1}&\mbox{if $\tilde{x}_j=0$ for all $1\leq j\leq s_2-l$}\\
r_{j_2}&\mbox{if $d_j=0$ for all $1\leq j\leq l$}\\
\mathrm{min}(r_{j_3}, r_{j_4})&\mbox{if $d_{j_3}=1$ and $\tilde{x}_{j_4}=1$}
\end{cases}  
\end{equation}
When $r=r_{j_1}$ or $r_{j_3},$ equation \eqref{spin} simplifies to
\begin{align}\label{spin_conditionB_dj}
\phi_6 = \sum_{i=1}^c \left(a_i (i' \circ 2\nu) +a_i' (i'\circ\alpha_1)\right) + \sum_{i=1}^{m-c} (2 b_i \nu + b_i' \alpha_1)+ (i^{r}\circ \widetilde{\eta_{2^{r}}^2}) + \sum_{j=1}^{\tilde{s}_1} c_j (\Sigma^2 i_{3^{r_j}}\circ\alpha_1) \nonumber \\
+\sum_{j=1}^{s_2-l}(2 x_j (\Sigma^2 i_{2^{r_j}}\circ 2^{\zeta(r_j)}\nu)+\sum_{j=1}^{l}  2 d^{r_j} (i^{r_j}\circ \Sigma^2 i_{2^{r_j}}\circ 2^{1+\delta_{r_j}}\nu).
\end{align}
When $r=r_{j_2}$ or $r_{j_4},$ equation \eqref{spin} can be written as
\begin{align}\label{spin_conditionB_xj}
\phi_6 = \sum_{i=1}^c \left(a_i (i' \circ 2\nu) +a_i' (i'\circ\alpha_1)\right) + \sum_{i=1}^{m-c} (2 b_i \nu + b_i' \alpha_1)+ \sum_{j=1}^{\tilde{s}_1} c_j (\Sigma^2 i_{3^{r_j}}\circ\alpha_1) \nonumber \\
+\sum_{j=1}^{s_2-l} 2 x_j (\Sigma^2 i_{2^{r_j}}\circ 2^{\zeta(r_j)}\nu)+ \widetilde{\eta_{2^{r}}^2}+\sum_{j=1}^{l} 2 d^{r_j} (i^{r_j}\circ \Sigma^2 i_{2^{r_j}}\circ 2^{1+\delta_{r_j}}\nu) . 
\end{align}
\end{itemize}
\end{itemize}

\underline{\textbf{$M$ is a non-spin manifold}}

If $Sq^2: H^4(M;\mathbb{Z}/2)\to H^6(M;\mathbb{Z}/2)$ is non-trivial, then at least one of $y_j, z_w$ in \eqref{phi_6} must be $1.$
\begin{itemize}
    \item If $M$ satisfies Condition C, then $y_j=0$ for all $1\leq j\leq s_2$ and $z_w=1$ for at least one $1\leq w\leq m.$ Using \cite[Lemma 4.1]{LiZhu24}, we may assume that $d_j, \tilde{x}_j, y_j'$ are zero for all $j.$ Furthermore, from \cite[\S 3]{rhuang}, we can set $z_{w_0}=1$ and $z_w=0$ for all $w\neq w_0.$ Consequently, \eqref{phi_6} reduces to the form:
    \begin{align}\label{nonspin_conditionC}
\phi_6 = \sum_{i=1}^c \left(a_i (i' \circ 2\nu) +a_i' (i'\circ\alpha_1)\right) + \sum_{i=1}^{m-c} (2 b_i \nu + b_i' \alpha_1) + \sum_{j=1}^{\tilde{s}_1} c_j (\Sigma^2 i_{3^{r_j}}\circ\alpha_1) \nonumber \\
+\sum_{j=1}^{s_2-l} 2 x_j (\Sigma^2 i_{2^{r_j}}\circ 2^{\zeta(r_j)}\nu) +\eta+\sum_{j=1}^l 2 d^{r_j} (i^{r_j}\circ \Sigma^2 i_{2^{r_j}}\circ 2^{1+\delta_{r_j}}\nu).  
\end{align}
\item If $M$ satisfies Condition D, then $y_j=1$ for some $1\leq j\leq s_2.$ Let $j_5$ be the index such that $r_{j_5}= \min\{r_j: y_j=1 ~\text{for}~1\leq j\leq s_2\}=\min\{r_j: u+v\in \mathrm{Im}(\beta_{r_j})~\text{for}~1\leq j\leq s_2\}.$ Then by \cite[Lemma 4.1]{LiZhu24}, we may assume that $y_{j_5}=1, y_j=0$ for all $j\neq j_5,$ and $d_j,\tilde{x}_j, y_j', z_w$ in \eqref{phi_6} are all zero. Under these assumptions, \eqref{phi_6} takes the form:
\begin{align}\label{nonspin_conditionD}
\phi_6 = \sum_{i=1}^c \left(a_i (i' \circ 2\nu) +a_i' (i'\circ\alpha_1)\right) + \sum_{i=1}^{m-c} (2 b_i \nu + b_i' \alpha_1)+ \sum_{j=1}^{\tilde{s}_1} c_j (\Sigma^2 i_{3^{r_j}}\circ\alpha_1) \nonumber \\
+\sum_{j=1}^{s_2-l} 2 x_j (\Sigma^2 i_{2^{r_j}}\circ 2^{\zeta(r_j)}\nu)+\tilde{\eta}_{2^{r_{j_5}}}+\sum_{j=1}^{l} 2 d^{r_j} (i^{r_j}\circ \Sigma^2 i_{2^{r_j}}\circ 2^{1+\delta_{r_j}}\nu).  
\end{align}
\end{itemize}

\subsection{Concordance inertia group of \texorpdfstring{$M\times\s^k$}{}}
We note from \eqref{stableM^6} and \eqref{phi_6} that there is a long exact sequence 
\begin{equation}\label{M_leq}
 \begin{tikzcd}[column sep=1em,font=\small]
	\cdots & {[\Sigma^{k+1}M^{(5)}, Top/O]} &{ \pi_{6+k}(Top/O)} & {[\Sigma^k M, Top/O]}\\
	{[\Sigma^k M^{(5)}, Top/O]} & {\pi_{5+k}(Top/O)} &{[ \Sigma^{k-1} M, Top/O]}& \cdots
	\arrow[from=1-1, to=1-2]
	\arrow["(\Sigma^{k+1}\phi_6)^*",from=1-2, to=1-3]
	\arrow["(\Sigma^k f_{M})^*",from=1-3, to=1-4]
	\arrow["(\Sigma^k \tilde{i})^*",from=1-4, to=2-1, overlay, out=-7, in=171]
	\arrow["(\Sigma^k \phi_6)^*",from=2-1, to=2-2]
	\arrow[from=2-2, to=2-3]
	\arrow[from=2-3, to=2-4] 
	\end{tikzcd}
    \end{equation}
    induced from the cofiber sequence $\mathbb{S}^5\xrightarrow{\phi_6}M^{(5)}\xhookrightarrow{\tilde{i}}M\xrightarrow{f_{M}}\mathbb{S}^6\cdots$ Now combining Lemma \ref{concordanceinertia} and \eqref{M_leq}, we get 
\begin{equation}\label{Concordance_ Inertia_expression}
    I_c(M\times\mathbb{S}^k)=\mathrm{Im}\left([\Sigma^{k+1}M^{(5)}, Top/O]\xrightarrow[]{(\Sigma^{k+1}\phi_6)^*} \pi_{6+k}(Top/O)\right).
\end{equation}

We now compute the image of each component of $\Sigma^{k+1}\phi_6$ given in \eqref{phi_6} along $Top/O.$ From \cite[Lemma 3.1]{SBRKAS}, we have 
\begin{lemma}\label{image_eta}
    The image of $\eta^*: \pi_{5+k}(Top/O)\to\pi_{6+k}(Top/O)$ is 
    \begin{itemize}
        \item[(a)] zero if $k=1,2,6,7,8,10.$
        \item[(b)] $\mathbb{Z}/2$ if $3,4,5,9.$ 
    \end{itemize}
\end{lemma}
From \cite[Lemma 3.1]{AS24}, we get
\begin{lemma}\label{imagenu}\noindent
\begin{itemize}
    \item[(a)]  The image of $(2\nu_{3+k})^*: \pi_{3+k}(Top/O)\to \pi_{6+k}(Top/O)$ is zero for all $1\leq k\leq 10.$
    \item[(b)] The map $\alpha_1^*: \pi_{3+k}(Top/O)\to\pi_{6+k}(Top/O)$ has trivial image for $1\leq k\leq 10$ with $ k\neq 7$ and image $\mathbb{Z}/3$ for $k=7.$ 
\end{itemize}
\end{lemma}

\begin{lemma}\label{lemma2.16}
		The image of the map $(\Sigma^{k+1} \pi)^*:[\Sigma^{k+1}\mathbb{C}P^2,Top/O]\to \pi_{6+k}(Top/O)$ is zero for all $1\leq k\leq 10, k\neq 7$ and is $\z/3$ for $k=7.$
\end{lemma}
Using the Lemma \ref{imagenu}, we derive the image of the attaching map corresponding to one of the Moore space components.
\begin{cor}\label{imagenu_extended}
     Let $r$ be a positive integer. 
     \begin{itemize}
         \item[(i)] The map $(\Sigma^{k+1}i'\circ 2\nu)^*: [\Sigma^{k+1}\mathbb{C}P^2, Top/O]_{(2)}\to \pi_{6+k}(Top/O)_{(2)}$ is trivial for all $1\leq k\leq 10.$
         \item[(ii)] The map $(\Sigma^{k+3}i_{2^r}\circ 2\nu)^*:[\Sigma^{3+k}M(\mathbb{Z}/{2^r}), Top/O]\to \pi_{6+k}(Top/O)$ has image zero for all $1\leq k\leq 10.$
         \item[(iii)] The image of $(\Sigma^{k+1} i^r\circ \Sigma^{k+3}i_{2^r}\circ 2\nu)^*: [\Sigma^{k+1}C_{i_{2^r}\circ\eta}, Top/O]\to \pi_{6+k}(Top/O)$ is trivial for all $1\leq k\leq 10.$
         \item[(iv)]  The image of the map $(\Sigma^{k+3}i_{3^r}\circ\alpha_1)^*:[\Sigma^{3+k} M(\z/{3^r}),Top/O]\to \pi_{6+k}(Top/O)$ is given by:
             \begin{itemize}
                \item[(a)] Zero if $k\neq 7.$  
                \item[(b)] $\mathbb{Z}/3$ if $k=7.$
             \end{itemize} 
        \item[(v)] The image of $(\Sigma^{k+1}i'\circ \alpha_1)^*:[\Sigma^{k+1}\mathbb{C}P^2, Top/O]_{(3)}\to \pi_{6+k}(Top/O)_{(3)}$ is: 
     \begin{itemize}
         \item[(a)] Trivial for all $1\leq k\leq 10, k\neq 7.$
         \item[(b)] $\mathbb{Z}/3$ for $k=7.$
     \end{itemize}
     \end{itemize}
\end{cor}
\begin{proof}
The results {(i)}, {(ii)}, {(iii)} are consequence of Lemma \ref{imagenu}{(a)}.

The results {(iv)}(a) and {(v)}(a) are derived from the fact that the map $(\alpha_1)^*:\pi_{3+s}(Top/O)\to \pi_{6+s}(Top/O)$ is zero for all $1\leq s\leq 10$, except when $s=7$, as per Lemma \ref{imagenu}(b).

For $k=7,$ consider the long exact sequence induced from $\mathbb{S}^0\xrightarrow{\times 3^r}\mathbb{S}^0\xhookrightarrow{i_{3^r}}M(\mathbb{Z}/{3^r})$ and $\mathbb{S}^3\xrightarrow{\eta}\mathbb{S}^2\xhookrightarrow{i'}\mathbb{C}P^2$ along $Top/O.$ From this, it follows that the image of both maps $(\Sigma^8 i_{3^r})^*:[\Sigma^{10}M(\mathbb{Z}/{3^r},Top/O]\to \Theta_{10}$ and $(\Sigma^8 i')^*: [\Sigma^8\mathbb{C}P^2, Top/O]\to \Theta_{10}$ is $\mathbb{Z}/3\subset \Theta_{10}$. Consequently, Lemma \ref{imagenu}(b) establishes statements {(iv)}(b) and {(v)}(b). 
\end{proof}

\begin{lemma}\label{sigma4moore}
    Let $r$ be any positive integer.
    \begin{itemize}
        \item[(i)] The image of $(\Sigma^{k+4}i_{2^r}\circ \eta^2)^*:[\Sigma^{k+4}M(\mathbb{Z}/{2^r}), Top/O]\to \pi_{6+k}(Top/O)$ is 
        \begin{itemize}
            \item[(a)] zero for $1\leq k\leq 10, k\neq 5.$
            \item[(b)]  $\mathbb{Z}/2$ for $k=5.$
        \end{itemize}
        \item[(ii)] The map $(\tilde{\eta}_{2^r})^*: [\Sigma^{4+k}M(\mathbb{Z}/{2^r}), Top/O]\to \pi_{6+k}(Top/O)$ has image
        \begin{itemize}
            \item[(a)] zero for $k=1,2,6,7,8.$
            \item[(b)] $\mathbb{Z}/2$ for $k=4,9.$
        \end{itemize}
        \item [(iii)] The image of $(\tilde{\eta}_{2^r})^*:[\Sigma^7 M(\z/{2^r}),Top/O]\to \pi_9(Top/O)$ equals
				\begin{itemize}
					\item [(a)] $\z/2\oplus\z/2$ if $r=1$ and $2.$
					\item [(b)] $\z/2$ if $r\geq 3.$
			\end{itemize}
			\item[(iv)] The image of $(\tilde{\eta}_{2^r})^*:[\Sigma^9 M(\z/{2^r}),Top/O]\to \pi_{11}(Top/O)$ equals 
			\begin{itemize}
				\item[(a)] $\z/4$ if $r=1.$
				\item [(b)] $\z/2$ if $r\geq 2.$
			\end{itemize}
            \item[(v)] The image of $(\tilde{\eta}_{2^r})^*:[\Sigma^{14} M(\z/{2^r}),Top/O]\to \pi_{16}(Top/O)$ is 
            \begin{itemize}
                \item[(a)] $\mathbb{Z}/2$ for $r=1.$
                \item[(b)] zero for $r\geq 2.$
            \end{itemize}
    \end{itemize}
\end{lemma}
\begin{proof}
By \cite[Corollary 3.3]{SBRKAS}, the map $(\eta^2)^*:\pi_{4+k}(Top/O)\to \pi_{6+k}(Top/O)$ is trivial for $1\leq k\leq 10, k\neq 5,$ which establishes {(i)}{(a)}. For $k=5,$ \cite[Corollary 3.3]{SBRKAS} states that the map $(\eta^2)^*:\pi_9(Top/O)\to \pi_{11}(Top/O)$ has image $\mathbb{Z}/2$. Additionally, the map $(\Sigma^9 i_{2^r})^*:[\Sigma^9 M(\mathbb{Z}/{2^r}), Top/O]\to \pi_9(Top/O)$ is onto, which implies that the image of $(\Sigma^{9}i_{2^r}\circ \eta^2)^*:[\Sigma^{9}M(\mathbb{Z}/{2^r}), Top/O]\to \pi_{11}(Top/O)$ is $\mathbb{Z}/2.$ 

  The results in {(ii)}-{(iv)} follow directly from \cite[Lemma 4.2]{SBRKAS}.

Consider the case $k=10$ and $r=1$. By combining \cite[Proposition 1.7]{toda} and \cite[Page 45]{kubo}, computing the image of $(\tilde{\eta}_{2})^*:[\Sigma^{14} M(\mathbb{Z}/{2}),Top/O]\to \Theta_{16}$ reduces to evaluating the Toda bracket $\langle\eta, 2, \kappa\rangle$ or $\langle\eta, 2, \kappa+\sigma^2\rangle,$ where $\pi_{14}^s=\mathbb{Z}/2\{\kappa\}\oplus\mathbb{Z}/2\{\sigma^2\}.$ From \cite[Table A3.3]{ravenel}, both Toda brackets are nontrivial, so the image of $(\tilde{\eta}_{2})^*$ is $\mathbb{Z}/2.$

Since $\eta^*: \pi_{15}(Top/O)\to\pi_{16}(Top/O)$ is trivial and $\eta= q_{15}\circ\tilde{\eta}_{2^r},$ the restriction of $(\tilde{\eta}_{2^r})^*: [\Sigma^{14}M(\mathbb{Z}/{2^r}), Top/O]\to \pi_{16}(Top/O)$ to $\pi_{15}(Top/O)$ is trivial. 

We note that for $r\geq 2,$ there exists a unique element $B(\chi_r^1):M(\mathbb{Z}/2)\to M(\mathbb{Z}/{2^r})$ such that the following diagram commutes:
\begin{equation*}
\begin{tikzcd}
{\s^{14}}\arrow[r,"\times 2^r"]\arrow[d,-,double equal sign distance,double]&{\s^{14}}\arrow[hookrightarrow]{r}{\Sigma^{14}i_{2^r}}&{\Sigma^{14} M(\mathbb{Z}/{2^r})}\arrow[r,"q_{15}"]&{\s^{15}}\arrow[r,"\times 2^r"]\arrow[d,-,double equal sign distance,double]&{\s^{15}}\\
{\mathbb{S}^{14}}\arrow[r,"\times 2"']&{\mathbb{S}^{14}} \arrow[hookrightarrow]{r}[swap]{\Sigma^{14}i_{2}}\arrow[u,"\times 2^{r-1}"']&{\Sigma^{14} M(\z/2)}\arrow[r,"q_{15}"']\arrow[u,dashed,"\Sigma^{14} B(\chi^1_{r})"']&{\mathbb{S}^{15}}\arrow[r,"\times 2"']&{\mathbb{S}^{15}}\arrow[u,"\times 2^{r-1}"']
\end{tikzcd}
\end{equation*} which induces the following commutative diagram of long exact sequences along $Top/O.$ 
\begin{center}
	\begin{tikzcd}[column sep=2.0em]
		\arrow[r,"\times 2^r"]&{\pi_{15}(Top/O)}\arrow[r,"(q_{15})^*"]\arrow[d,-,double equal sign distance,double]&{[\Sigma^{14} M(\z/{2^r}), Top/O]}\arrow[d,"(\Sigma^{14} B(\chi^1_{r}))^*"']\arrow[r,"(\Sigma^{14}i_{2^r})^*"]&{\pi_{14}(Top/O)}\arrow[d,"0"']\arrow[r,"0"]&{\pi_{14}(Top/O)}\arrow[d,-,double equal sign distance,double]\\
		\arrow[r,"\times 2"']&{\pi_{15}(Top/O)}\arrow[r,"(q_{15})^*"']&{[\Sigma^{14} M(\z/2), Top/O]}\arrow[r,"(\Sigma^{14}i_{2})^*"']&{\pi_{14}(Top/O)}\arrow[r,"0"']&{\pi_{14}(Top/O)}
	\end{tikzcd}
\end{center}
This diagram shows that if for any $a\in [\Sigma^{14}M(\mathbb{Z}/{2^r}), Top/O]$ satisfies $(\Sigma^{14}i_{2^r})^*(a)\neq 0,$ then $(\Sigma^{14}B(\chi_r^1))^*(a)=0.$ Combining this with the fact that $(q_{15}\circ \tilde{\eta}_{2^r})^*:\pi_{15}(Top/O)\to \pi_{16}(Top/O)$ is a trivial map, it follows from  $\tilde{\eta}_{2^r} = \Sigma^{14}B(\chi_r^1)\circ \tilde{\eta}_{2}$ that the map  $(\tilde{\eta}_{2^r})^*:[\Sigma^{14} M(\z/{2^r}),Top/O]\to \pi_{16}(Top/O)$ is trivial for all $r\geq 2.$ This completes the proof.
\end{proof}
 \begin{lemma}\label{imageeta^2tilde}
     The map $(\widetilde{\eta_{2^r}^2})^*:[\Sigma^{k+3} M(\mathbb{Z}/{2^r}), Top/O]\to \pi_{6+k}(Top/O)$ has the following image: 
\begin{itemize}
    \item[(i)] Trivial for all $1\leq k\leq 10, k\neq 4,5.$
    \item[(ii)] For $k=4:$
    \begin{enumerate}
        \item[(a)] $\mathbb{Z}/2$ if $r=1,2.$
        \item[(b)] Trivial if $r\geq 3.$ 
    \end{enumerate}
    \item[(iii)] $\mathbb{Z}/2$ for $k=5.$ 
\end{itemize}
for any integer $r\geq 1.$
 \end{lemma}
\begin{proof}
From \cite[Lemma 4.2]{SBRKAS}, we note that the map $(\tilde{\eta}_{2^r})^*:[\Sigma^{k+3} M(\mathbb{Z}/{2^r}), Top/O]\to \pi_{5+k}(Top/O)$ is trivial for $k=1,2,3,7,8,9.$ Since $\widetilde{\eta_{2^r}^2}= \eta\circ \tilde{\eta}_{2^r},$ it follows that the map $(\widetilde{\eta_{2^r}^2})^*:[\Sigma^{k+3} M(\mathbb{Z}/{2^r}), Top/O]\to \pi_{6+k}(Top/O)$ has a trivial image for the same values of $k.$ 
 
  By Lemma \ref{image_eta}, the image of $\eta^*: \pi_{5+k}(Top/O)\to \pi_{6+k}(Top/O)$ is trivial for $k= 6,$ and $10$, which implies that the image of $(\widetilde{\eta_{2^r}^2})^*$ is also trivial for these values of $k.$  
    
Consider $k=4$ and assume $r=1,2.$ Then we note that $[\Sigma^7 M(\mathbb{Z}/{2^r}), Top/O]$ fits into the following short exact sequence 
\begin{equation}\label{moore_expression}
    0\to \mathbb{Z}/2\to [\Sigma^7 M(\mathbb{Z}/{2^r}), Top/O]\to \mathbb{Z}/{2^r}\to 0.
\end{equation} 
Using the computation $[\Sigma^7 M(\mathbb{Z}/{2^r}), G/O]=\mathbb{Z}/{2^r}\oplus\mathbb{Z}/2$ and the injectivity of $\psi_*:[\Sigma^7 M(\mathbb{Z}/{2^r}), Top/O]\to [\Sigma^7 M(\mathbb{Z}/{2^r}), G/O]$, it follows from \eqref{moore_expression} that 
\[[\Sigma^7 M(\mathbb{Z}/{2^r}), Top/O]\cong[\Sigma^7 M(\mathbb{Z}/{2^r}), G/O]=\begin{cases}
    \mathbb{Z}/2\oplus\mathbb{Z}/2 ~\text{if}~ r=1,\\
    \mathbb{Z}/4\oplus\mathbb{Z}/2~ \text{if}~ r=2.
\end{cases}\]
Since the map $\eta^*:\pi_8(G/O)\to\pi_9(G/O)$ is onto and $q_8\circ\tilde{\eta}_{2}=\eta,$ the map $(\tilde{\eta}_{2})^*:[\Sigma^7 M(\mathbb{Z}/2), G/O]\to \pi_9(G/O)$ is also onto. 
For $r=1$ and $2,$ consider the following commutative diagram
   \begin{center}
        \begin{tikzcd}
            {[\Sigma^7 M(\mathbb{Z}/{2^r}), Top/O]}\arrow[r,"(\tilde{\eta}_{2^r})^*"]\arrow[d,"\psi_*","\cong"']&{\pi_9(Top/O)}\arrow[d,two heads,"\psi_*"]\\
            {[\Sigma^7 M(\mathbb{Z}/{2^r}), G/O]}\arrow[r,"(\tilde{\eta}_{2^r})^*"']&{\pi_9(G/O)}
        \end{tikzcd}
    \end{center}

   Using \cite[Lemma 4.2 {(b)}]{SBRKAS} and the commutative diagram above, we conclude that the image of $(\tilde{\eta}_{2^r})^*:[\Sigma^7 M(\mathbb{Z}/{2^r}), Top/O]\to\pi_9(Top/O)$ is $\mathbb{Z}/2\{[\nu^3]\}\oplus\mathbb{Z}/2\{[\mu]\}.$ Since $\eta\circ\nu=0,$ it follows from $(\widetilde{\eta_{2^r}^2})^*=  \eta^*\circ(\tilde{\eta}_{2^r})^*$ that the image of $(\widetilde{\eta_{2^r}^2})^*:[\Sigma^{7} M(\mathbb{Z}/{2^r}), Top/O]\to \pi_{10}(Top/O)$ is $\mathbb{Z}/2\{[\eta\circ\mu]\}$ for $r=1$ and $2.$

   For $r\geq 3,$ \cite[Lemma 4.2 (b)]{SBRKAS} shows that the map $(\tilde{\eta}_{2^r})^*: [\Sigma^7 M(\mathbb{Z}/{2^r}), Top/O]\to \pi_8(Top/O)$ has image $\mathbb{Z}/2\{[\nu^3]\}$. Since $\eta\circ\nu=0$ \cite{toda} and $\widetilde{\eta_{2^r}^2}= \tilde{\eta}_{2^r}\circ \eta,$ the map $(\widetilde{\eta_{2^r}^2})^*:[\Sigma^{7} M(\mathbb{Z}/{2^r}), Top/O]\to \pi_{10}(Top/O)$  is trivial for all $r\geq 3.$

   For $k=5,$ Lemmas \ref{image_eta} and \ref{sigma4moore} show that both $(\tilde{\eta}_{2^r})^*: [\Sigma^8 M(\mathbb{Z}/{2^r}), Top/O]\to \pi_{10}(Top/O)$ and $\eta^*: \pi_{10}(Top/O)\to \pi_{11}(Top/O)$ have image $\mathbb{Z}/2.$ Since $(\widetilde{\eta_{2^r}^2})^*= \eta^*\circ (\tilde{\eta}_{2^r})^*,$ the image of $(\widetilde{\eta_{2^r}^2})^*: [\Sigma^8 M(\mathbb{Z}/{2^r}), Top/O]\to \pi_{11}(Top/O)$ is $\mathbb{Z}/2.$ 
\end{proof}
 Using the above lemma, we get 
 \begin{lemma}\label{image_i_eta^2tilde}
    Let $r$ be a positive integer. The image of the map $(\Sigma^{k+1}i^r\circ\widetilde{\eta_{2^r}^2})^*:[\Sigma^{k+1}C_{i_{2^r}\circ\eta}, Top/O]\to \pi_{6+k}(Top/O)$ is given as follows: 
     \begin{itemize}
    \item[(i)] Trivial for all $1\leq k\leq 10, k\neq 4,5.$
    \item[(ii)] For $k=4:$
    \begin{enumerate}
        \item[(a)] $\mathbb{Z}/2$ if $r=1,~2.$
        \item[(b)] Trivial if $r\geq 3.$ 
    \end{enumerate}
    \item[(iii)] $\mathbb{Z}/2$ for $k=5.$ 
\end{itemize}
 \end{lemma}
 \begin{proof}
     Statement {(i)} follows directly from Lemma \ref{imageeta^2tilde} {(i)}.

    For $k=4,$ consider the long exact sequence induced from $\mathbb{S}^3\xrightarrow{\Sigma^2 i_{2^r}\circ\eta}\Sigma^2 M(\mathbb{Z}/{2^r})\xhookrightarrow{i^r} C_{i_{2^r}\circ\eta},$ which shows that $(\Sigma^5 i^r)^*: [\Sigma^5 C_{i_{2^r}\circ\eta}, Top/O]\to [\Sigma^7 M(\mathbb{Z}/{2^r}), Top/O]$ is onto. Together with Lemma \ref{imageeta^2tilde} {(ii)}, this establishes the result for $k=4$.

     For $k=5,$ the image of $(\Sigma^6 i^r)^*: [\Sigma^6 C_{i_{2^r}\circ\eta}, Top/O]\to [\Sigma^8 M(\mathbb{Z}/{2^r}), Top/O]$ lies within that of $(q_9)^*:\pi_9(Top/O)\to [\Sigma^8 M(\mathbb{Z}/{2^r}), Top/O].$ By Lemma \ref{imageeta^2tilde}{(iii)}, the image of $(\widetilde{\eta_{2^r}^2})^* : [\Sigma^8 M(\mathbb{Z}/{2^r}), Top/O]\to \pi_{11}(Top/O)$ coincides with that of the composition $(q_9\circ\widetilde{\eta_{2^r}^2} )^*=(\eta^2)^*:\pi_9(Top/O)\to \pi_{11}(Top/O)$. Combining these two observations we get the result for $k=5.$
 \end{proof}
The following theorem determines the concordance inertia group of $M\times\mathbb{S}^k$ for $1\leq k\leq 10.$ 
 \begin{theorem} \label{prop2.9}
Let $M$ be a simply connected, closed, smooth $6$-manifold with homology of the form \eqref{eqn1}. Then:
		\begin{itemize}
			\item[(i)] $I_c(M\times\s^k)=0$ for $k=1,2,6,8.$
            \item[(ii)] For $k=3:$
            \begin{itemize}
                \item[(a)] If $Sq^2$ acts trivially on $H^*(M;\mathbb{Z}/2),$ then $I_c(M\times\mathbb{S}^3)=0$  
               \item[(b)] If $Sq^2$ acts nontrivially on $H^*(M;\mathbb{Z}/2)$ and $M$ satisfies condition C, then $I_c(M\times\mathbb{S}^3)=\mathbb{Z}/2.$
                \item[(c)] Suppose $Sq^2$ acts nontrivially on $H^*(M;\mathbb{Z}/2)$ and $M$ satisfies Condition D. Then
                \begin{enumerate}
                    \item $I_c(M\times\mathbb{S}^3)=\mathbb{Z}/2\oplus\mathbb{Z}/2,$ if $r_{j_5}=1$ or $2$ in \eqref{nonspin_conditionD}. 
                    \item $I_c(M\times\mathbb{S}^3)=\mathbb{Z}/2,$ if $r_{j_5}\geq 3$ in \eqref{nonspin_conditionD}. 
                \end{enumerate}
            \end{itemize}
            \item[(iii)] For $k=4:$
            \begin{itemize}
                \item[(a)] If $Sq^2$ is trivial on $H^4(M;\mathbb{Z}/2)$ and $\Theta$ is trivial on $H^3(M;\mathbb{Z}/2),$ then $I_c(M\times\mathbb{S}^4)=0.$
                \item[(b)] Suppose $Sq^2$ be trivial on $H^4(M;\mathbb{Z}/2)$ and $\Theta$ is non-trivial on $H^3(M;\mathbb{Z}/2).$
                \begin{enumerate}
                    \item $I_c(M\times\mathbb{S}^4)=0$ if $M$ satisfies condition A.
                    \item If $M$ satisfies Condition B, then 
                    \begin{enumerate}
                        \item[(I)] $I_c(M\times\mathbb{S}^4)=0,$ if $r\geq 3$ in \eqref{spin_conditionB_coefficient}.
                        \item[(II)] $I_c(M\times\mathbb{S}^4)=\mathbb{Z}/2$ if $r$ is either $1$ or $2$ in \eqref{spin_conditionB_coefficient}. 
                    \end{enumerate}
                \end{enumerate}
                \item[(c)] If $Sq^2$ acts nontrivially on $H^4(M;\mathbb{Z}/2),$ then $I_c(M\times\mathbb{S}^4)=\mathbb{Z}/2.$
            \end{itemize}
           \item[(iv)] For $k=5:$
            \begin{itemize}
                \item[(a)] Suppose $Sq^2: H^4(M;\mathbb{Z}/2)\to H^6(M;\mathbb{Z}/2)$ be trivial.
                \begin{enumerate}
                    \item If $\Theta$ is trivial on $H^3(M;\mathbb{Z}/2),$ then $I_c(M\times\mathbb{S}^5)=0.$
                    \item If $\Theta$ acts nontrivially on $H^3(M;\mathbb{Z}/2)$, then $I_c(M\times\mathbb{S}^5)=\mathbb{Z}/2.$
                \end{enumerate}
                \item[(b)] If $Sq^2: H^4(M;\mathbb{Z}/2)\to H^6(M;\mathbb{Z}/2)$ is nontrivial and $M$ satisfies Condition C, then $I_c(M\times\mathbb{S}^5)=\mathbb{Z}/2.$
                \item[(c)] If $Sq^2: H^4(M;\mathbb{Z}/2)\to H^6(M;\mathbb{Z}/2)$ is nontrivial and $M$ satisfies Condition D, then 
                \begin{enumerate}
                    \item $I_c(M\times\mathbb{S}^5)=\mathbb{Z}/4$ if $r_{j_5}=1$ in \eqref{nonspin_conditionD}. 
                    \item $I_c(M\times\mathbb{S}^5)=\mathbb{Z}/2$ if $r_{j_5}\geq2$ in \eqref{nonspin_conditionD}. 
                \end{enumerate}
            \end{itemize}
            \item[(v)] For $k=7:$ 
			\begin{itemize}
				\item [(a)] If $3 \mid p_1(M),$ then $I_c(M\times\s^7)=0.$ 
				\item[(b)] If $3 \nmid p_1(M),$ then $I_c(M\times\s^7)=\z/3.$ 
			\end{itemize}
		 \item[(vi)] For $k=9:$
            \begin{itemize}
                \item[(a)] If $Sq^2$ acts trivially on $H^4(M;\mathbb{Z}/2),$ then $I_c(M\times\mathbb{S}^9)=0.$
                \item[(b)] If $Sq^2$ acts nontrivially on $H^4(M;\mathbb{Z}/2),$ then $I_c(M\times\mathbb{S}^9)=\mathbb{Z}/2.$
            \end{itemize}
          \item[(vii)] For $k=10:$
          \begin{itemize}
              \item[(a)] If $Sq^2: H^4(M;\mathbb{Z}/2)\to H^6(M;\mathbb{Z}/2)$ is trivial, then $I_c(M\times\mathbb{S}^{10})=0.$
              \item[(b)] If $Sq^2: H^4(M;\mathbb{Z}/2)\to H^6(M;\mathbb{Z}/2)$ is nontrivial and $M$ satisfies Condition C, then $I_c(M\times\mathbb{S}^{10})=0.$
              \item[(c)] If $Sq^2: H^4(M;\mathbb{Z}/2)\to H^6(M;\mathbb{Z}/2)$ is nontrivial and $M$ satisfies Condition D, then 
              \begin{enumerate}
                  \item $I_c(M\times\mathbb{S}^{10})=\mathbb{Z}/2$ if $r_{j_5}=1$ in \eqref{nonspin_conditionD}. 
                  \item $I_c(M\times\mathbb{S}^{10})= 0$ if $r_{j_5}\geq 2$ in \eqref{nonspin_conditionD}. 
              \end{enumerate}
          \end{itemize}
		\end{itemize}
	\end{theorem}
\begin{proof} 
Using \eqref{phi_6} and \eqref{Concordance_ Inertia_expression}, it follows from Lemma \ref{imagenu} {(a)}, Corollary \ref{imagenu_extended} {(i)}, {(ii)} that computing $I_c(M\times\mathbb{S}^k)$ for $1\leq k\leq 10$ reduces to determining the image of  
\begin{align}\label{attach}
(\Sigma^{k+1}\phi_6)^* = \sum_{i=1}^{c} a_i' (\Sigma^{k+1}i'\circ\alpha_1)^* + \sum_{i=1}^{m-c}  b_i' \alpha_1^*+ \sum_{j=1}^{\tilde{s}_1} c_j (\Sigma^{k+3}i_{3^{r_j}}\circ\alpha_1)^*+\sum_{j=1}^{s_2-l} \tilde{x}_j \widetilde{\eta_{2^{r_j}}^2}^* \nonumber \\
+\sum_{j=1}^{s_2} \left(y_j \tilde{\eta}_{2^{r_j}}^*+ y_j'(\Sigma^{k+4}i_{2^{r_j}}\circ\eta^2)^*\right)+\sum_{w=1}^m z_{w} \eta^*+\sum_{j=1}^{l} d_j (\Sigma^{k+1}i^{r_j}\circ \widetilde{\eta_{2^{r_j}}^2})^*,  
\end{align} induced  along $Top/O$-level.

For $k=7,$ by combining Lemmas \ref{image_eta}, \ref{sigma4moore}, \ref{imageeta^2tilde}, and \ref{image_i_eta^2tilde}, we see that $(\Sigma^8 \phi_6)^*$ in \eqref{attach} further reduces to 
 \begin{align*}
(\Sigma^{8}\phi_6)^* = \sum_{i=1}^{c} a_i' (\Sigma^{k+1}i'\circ\alpha_1)^* + \sum_{i=1}^{m-c}  b_i' \alpha_1^*+\sum_{j=1}^{\tilde{s}_1} c_j (\Sigma^{k+3}i_{3^{r_j}}\circ\alpha_1)^*.   
\end{align*} 
\begin{itemize}
    \item If $3\mid p_1(M),$ then $a_i',b_i', c_j$ all vanish identically, which implies that $I_c(M\times\mathbb{S}^7)=0$ in this case. 
    \item If $3\nmid p_1(M),$ then by \cite[Lemma 4.4]{LiZhu24} and \cite[\S 3]{rhuang}, exactly one of $a_i', b_i',c_j$ is non-trivial. Consequently, it follows from Lemma \ref{imagenu}{(b)} and Lemma \ref{imagenu_extended} {(iv)}, {(v)} that $I_c(M\times\mathbb{S}^7)=\mathbb{Z}/3.$ 
\end{itemize}

For all $1\leq k\leq 10$ with $k\neq 7,$ applying Lemma \ref{imagenu}{(b)} and Lemma \ref{imagenu_extended} {(iv)}, {(v)}, the map \eqref{attach} simplifies to 
\begin{align}\label{attach2}
(\Sigma^{k+1}\phi_6)^* = \sum_{j=1}^{s_2-l} \tilde{x}_j (\widetilde{\eta_{2^{r_j}}^2})^*+\sum_{j=1}^{s_2} \left(y_j \tilde{\eta}_{2^{r_j}}^*+ y_j'(\Sigma^{k+4}i_{2^{r_j}}\circ\eta^2)^*\right)+\sum_{w=1}^m z_{w} \eta^* \nonumber \\
+\sum_{j=1}^{l} d_j (\Sigma^{k+1}i^{r_j}\circ \widetilde{\eta_{2^{r_j}}^2})^*.  
\end{align} 

For $k=1,2,6,$ and $8$ the result follows from \eqref{attach2} by applying Lemmas \ref{image_eta}, \ref{sigma4moore}, \ref{imageeta^2tilde}, and \ref{image_i_eta^2tilde}.
 
If $Sq^2: H^4(M;\mathbb{Z}/2)\to H^6(M;\mathbb{Z}/2)$ is trivial, then combining \eqref{spin} and \eqref{attach2}, we obtain
\begin{align}\label{attach_spin}
(\Sigma^{k+1}\phi_6)^* = \sum_{j=1}^{s_2-l} \tilde{x}_j (\widetilde{\eta_{2^{r_j}}^2})^*+\sum_{j=1}^{s_2} y_j'(\Sigma^{k+4}i_{2^{r_j}}\circ\eta^2)^*+\sum_{j=1}^{l} d_j (\Sigma^{k+1}i^{r_j}\circ \widetilde{\eta_{2^{r_j}}^2})^*.  
\end{align} 
\begin{itemize}
    \item Since all components in \eqref{attach_spin} have trivial images for $k=3, 9,$ and $10$ by Lemmas \ref{sigma4moore}{(i)}, \ref{imageeta^2tilde}{(i)}, and \ref{image_i_eta^2tilde}{(i)}, respectively, the concordance inertia group of $M\times \mathbb{S}^k$ is zero for these values of $k.$  
    \item Moreover, if the secondary cohomology operation $\Theta$ on $H^3(M;\mathbb{Z}/2)$ is trivial, then from \eqref{spin_noeta2}, we observe that \eqref{attach_spin} further reduces to $(\Sigma^{k+1}\phi_6)^*=0.$ Consequently, $I_c(M\times\mathbb{S}^k)$ is trivial for all $1\leq k\leq 10.$ 
    \item Suppose the secondary cohomology operation $\Theta$ on $H^3(M;\mathbb{Z}/2)$ is nontrivial and $M$ satisfies Condition A. Then, combining \eqref{spin_conditionA} and \eqref{attach_spin}, we obtain $$I_c(M\times\mathbb{S}^k)=\mathrm{Im}\left([\Sigma^{4+k}M(\mathbb{Z}/{2^{r_{j_0}}}), Top/O]\xrightarrow{(\Sigma^{k+4}i_{2^{r_{j_0}}}\circ\eta^2)^*}\pi_{6+k}(Top/O)\right).$$ By Lemma \ref{sigma4moore} {(i)}, it follows that $I_c(M\times\mathbb{S}^k)$ is trivial for $k=4$ and isomorphic to $\mathbb{Z}/2$ for $k=5.$
    \item If $\Theta$ acts nontrivially on $H^3(M;\mathbb{Z}/2)$ and $M$ satisfies Condition B, then combining \eqref{spin_conditionB_dj}, \eqref{spin_conditionB_xj} and \eqref{attach_spin}, we see that $I_c(M\times\mathbb{S}^k)$ is either
$$\mathrm{Im}\left([\Sigma^{k+1}C_{i_{2^{r}\circ\eta}}, Top/O]\xrightarrow{(\Sigma^{k+1}i^r\circ\widetilde{\eta_{2^r}^2})^*}\pi_{6+k}(Top/O)\right), ~\text{if}~r=r_{j_1}~\text{or}~ r_{j_3} $$
or
$$    \mathrm{Im}\left([\Sigma^{k+3}M(\mathbb{Z}/{2^r}), Top/O]\xrightarrow{(\widetilde{\eta_{2^r}^2})^*}\pi_{6+k}(Top/O)\right),~\text{if}~r=r_{j_2}~\text{or} r_{j_4}.$$
Now the result for $k=4$ and $5$ follows from  Lemma \ref{imageeta^2tilde} and Lemma \ref{image_i_eta^2tilde}.
\end{itemize}

Suppose $Sq^2$ acts nontrivially on $H^4(M;\mathbb{Z}/2).$ Then at least one of $y_j$ and $z_w$ in \eqref{attach2} must be $1.$ Furthermore,
\begin{itemize}
    \item If $M$ satisfies Condition C, then the combination of \eqref{nonspin_conditionC} and \eqref{attach2} gives $$I_c(M\times\mathbb{S}^k)=\mathrm{Im}\left(\pi_{5+k}(Top/O)\xrightarrow{\eta^*}\pi_{6+k}(Top/O)\right).$$ Thus, by Lemma \ref{image_eta}, we obtain $I_c(M\times\mathbb{S}^k)=\mathbb{Z}/2$ for $k=3,4,5,9$ and $I_c(M\times\mathbb{S}^{10})=0.$
    \item If $M$ satisfies condition D, then \eqref{nonspin_conditionD} and \eqref{attach2} yield $$I_c(M\times\mathbb{S}^k)=\mathrm{Im}\left([\Sigma^{4+k}M(\mathbb{Z}/{2^{r_{j_5}}}), Top/O]\xrightarrow{(\widetilde{\eta}_{2^{r_{j_5}}})^*}\pi_{6+k}(Top/O)\right).$$
    Now, the computation for $k=3,4,5,9,10$ follows from Lemma \ref{sigma4moore}.
\end{itemize}
\end{proof}
Since $I_c(M\times\mathbb{S}^k)\subseteq I(M\times\mathbb{S}^k)\subseteq \Theta_{6+k},$ it follows from the above theorem that:
 \begin{cor}\label{InertiagroupMS7}
     For any simply connected, closed, smooth $6$-dimensional manifold with $3\nmid p_1(M),$ the inertia group of $M\times \mathbb{S}^7$ is $\Theta_{13}.$ In particular, $I(\mathbb{C}P^3\times\mathbb{S}^7)=\Theta_{13}.$  
 \end{cor}    
We observe that for $k=17,23,33,39,43,49,59,69,78,79,87,88,89,$ and $98$ there is an element $x$ in $(\mathit{coker}(J_{3+k}))_{(3)}$ such that $\alpha_1\circ x\neq 0$ in $(\mathit{coker}(J_{6+k}))_{(3)}$ \cite{ravenel}. Thus, for these specified values of $k,$ the image of $(\alpha_1)^*:\Theta_{6+k}\to \Theta_{3+k}$ is $\z/3.$ Now applying \cite[Lemma 4.4]{LiZhu24} and \cite[\S 3]{rhuang} in \eqref{phi_6}, we obtain from \eqref{Concordance_ Inertia_expression} that 
	\begin{cor}\label{extrainerM^6}
		Let $M$ be a $1$-connected, closed, smooth $6$-manifold such that $3\nmid p_1(M).$ Then $\z/3 \subseteq I(M\times\s^k)$ for $k=17,23,33,39,43,49,59,69,78,79,87, 88,89,$ and $98.$
	\end{cor}
\subsection{Computation of concordance structure set}
Recall from the theorems of \cite[Theorem IV.10.1, Section IV.10.12]{kirby} and \cite{hirsch} that for any closed, oriented, smooth manifold $M$ of dimension $\leq 6$, there exists an isomorphism $[M,Top/O]\cong H^3(M,\mathbb{Z}/2).$ Hence, to calculate the concordance structure set $[M\times \mathbb{S}^k, Top/O]$, it suffices to determine $[\Sigma^k M,Top/O]$, as per Corollary \ref{concorMtimesSk}.
\begin{thm} \label{maintheo}
		Let $M$ be a closed, smooth manifold of dimension $6$ with its homology given by \eqref{eqn1}. Then there exists a split short exact sequence
        $$0\to\Theta_7\to[\Sigma M,Top/O]\to[\Sigma M^{(5)},Top/O]\cong \bigoplus\limits_{i=1}^{s_2+m} \z/2\to0.$$
		
        Furthermore, $\mathcal{C}(M \times \mathbb{S}^1) \cong \z/{28} \oplus \bigoplus\limits_{i=1}^{s_2+m} \z/2.$ 
	\end{thm}

From \eqref{stableM^6}, \eqref{4skeletonM^6}, and \eqref{M_leq}, we have the following short exact sequence 
\begin{equation}\label{eqn5}
		0 \to \Theta_7 \to [\Sigma M,Top/O] \to [\Sigma M^{(5)},Top/O] \cong \bigoplus\limits_{i=1}^{s_2+m} \mathbb{Z}/2 \to 0
	\end{equation}
    
	In order to prove the above theorem, we need to show that this short exact sequence splits. Clearly, it is sufficient to work 2-locally. Also $Top/O$ is $\Omega ^{\infty}(top/o)$ where $top/o$ is a spectrum. Let $\mathcal{P}_7(Top/O)$ be $7$-th postnikov section of $Top/O$ i.e.,
	\begin{equation*} 
		\pi_k(\mathcal{P}_7(Top/O))\cong
		\begin{cases}
			\pi_k(Top/O) &\text{if $k\leq 7$}\\
			0 & \text{otherwise}
		\end{cases}
	\end{equation*} Also $ \Sigma M$ being $7$-dimensional, $[ \Sigma M,Top/O]\cong [ \Sigma M,\mathcal{P}_7(Top/O)]$ and $\mathcal{P}_7(Top/O)\simeq\Omega^{\infty}(\tau_{\leq 7}top/o)$ where $\tau_{\leq 7}top/o$ is the 7-th  postnikov section of the spectrum $top/o.$ So 
	\begin{equation*}
		\pi_k(to_{\leq 7}) \cong
		\begin{cases}
			\z /2 & \text{if $k=3$}\\
			\z /4 & \text{if $k=7$}\\
			0 & \text{otherwise}
		\end{cases}
	\end{equation*} where $to_{\leq 7} = \displaystyle (\tau_{\leq 7}top/o)_{(2)}.$\\
	We prove that $\displaystyle \{{ \Sigma M,to_{\leq 7}}\} \cong \mathbb{Z}/4 \oplus \bigoplus\limits_{i=1}^{s_2+m}\mathbb{Z}/2.$ Otherwise, from the short exact sequence \eqref{eqn5}, we get $\displaystyle {\{\Sigma M,to_{\leq 7}\}} \cong \mathbb{Z}/8 \oplus \bigoplus\limits_{i=2}^{s_2+m}\mathbb{Z}/2.$
	As a spectrum, we have a following cofiber sequence 
	\begin{equation*} 
		to_{\leq 7} \to \Sigma ^3 H \z/2 \xrightarrow{\delta_{to}} \Sigma^8 H \z/4.
	\end{equation*} \\
	Define the following
	\begin{center}
		$\delta_{\Tilde{t}}: \Sigma^3 H \z/2 \xrightarrow{\delta_{to}} \Sigma^8 H \z/4 \to \Sigma^8 H \z/2$ and $\Tilde{t}=$ homotopy fiber of $\delta_{\Tilde{t}}.$ 
	\end{center}
	\begin{lemma}\label{lemma3.6}
		\begin{enumerate}[label=(\alph*)]
			\item  ${\{ \Sigma M,to_{\leq 7}\}} = \z/8 \oplus \bigoplus\limits_{i=1}^{s_2+m-1}\z/2$ if and only if ${\{ \Sigma M,\Tilde{t}\}} \cong \z/4 \oplus \bigoplus\limits_{i=1}^{s_2+m-1}\z/2.$
			\item  ${\{\Sigma M,to_{\leq 7}\}} = \z/4 \oplus \bigoplus\limits_{i=1}^{s_2+m}\z/2$ if and only if ${\{\Sigma M,\Tilde{t}\}} \cong \z/2 \oplus \bigoplus\limits_{i=1}^{s_2+m}\z/2.$
		\end{enumerate}
	\end{lemma}
	\begin{proof}
		We have the following commutative diagram of homotopy cofibrations
		\begin{equation*}
			\begin{tikzcd}
				{to_{\leq 7}} \arrow[d,dashed,"\phi"'] \arrow[r] & {\Sigma^3 H \z/2} \arrow[d,-,double equal sign distance,double] \arrow[r,"\delta_{to}"] & {\Sigma^8 H \z/4} \arrow[d]\\
				{\Tilde{t}} \arrow[r] & {\Sigma^3 H \z/2} \arrow[r,"\delta_{\Tilde{t}}"'] & {\Sigma^8 H \z/2}
			\end{tikzcd}
		\end{equation*}
		The above diagram gives the spectrum map $\phi.$ Once again, consider the diagram
		\begin{equation*}
			\begin{tikzcd}[column sep=1em]
				{0}  \arrow[r] & {\{ \Sigma M,\Sigma^7 H \z/4\} \cong \z/4} \arrow[d] \arrow[r] & {\{ \Sigma M,to_{\leq 7}\}} \arrow[d,"\phi^*"] \arrow[r] & {\{\Sigma M, \Sigma^3 H \z/2\} \cong \bigoplus\limits_{i=1}^{s_2+m}\z/2 } \arrow[d,"\cong"] \arrow[r] & 0\\
				0 \arrow[r] & {\{ \Sigma M,\Sigma^7 H \z/2\} \cong \z/2} \arrow[r] & {\{\Sigma M, \Tilde{t}\}} \arrow[r] & {\{ \Sigma M, \Sigma^3 H \z/2\} \cong \bigoplus\limits_{i=1}^{s_2+m}\z/2} \arrow[r] & 0
			\end{tikzcd}
		\end{equation*}
		This diagram shows that $\phi^*$ is surjective and $\{\Sigma M, \Tilde{t}\} \cong \{ \Sigma M, to_{\leq 7}\} /A,$ where $A$ is generated by the image of $\bar{2}$ under the map $\{\Sigma M, \Sigma^7 H \z/4\} \to \{\Sigma M,to_{\leq 7}\}.$
		\begin{enumerate}[label=(\alph*)]
			\item If $\{\Sigma M,to_{\leq 7}\} \cong \z/8 \oplus \bigoplus\limits_{i=1}^{s_2+m-1}\z/2,$ then the image of $\bar{2}\in \{\Sigma M, \Sigma^7 H\z/4\}$ must be the class $(\bar{4},0,0,....,0)$ in $\{\Sigma M, to_{\leq 7}\}.$ Hence $\{\Sigma M,\Tilde{t}\} \cong \z/4 \oplus \bigoplus\limits_{i=1}^{s_2+m-1}\z/2.$
			\item If $\{\Sigma M,to_{\leq 7}\} \cong \z/4 \oplus \bigoplus\limits_{i=1}^{s_2+m}\z/2,$ then $\bar{2} \in \z/4 \cong \{\Sigma M,\Sigma^7 H \z/4\}$ maps to $(\bar{2},0,0,....,0) \in \{\Sigma M, to_{\leq 7}\}.$ Hence $\{\Sigma M,\Tilde{t}\} \cong \z/2 \oplus \bigoplus\limits_{i=1}^{s_2+m}\z/2.$
		\end{enumerate}
	\end{proof}
	\begin{lemma} 
		The map $\delta_{\Tilde{t}}: \Sigma^3 H \z/2 \to \Sigma^8 H \z/2$ is either equal to $Sq^5$ or $0$ up to homotopy.
	\end{lemma}
	\begin{proof}
		From the definition of $\delta_{\Tilde{t}}$ we have $\Sigma^3 H \z/2 \xrightarrow{\delta_{\Tilde{t}}} \Sigma^8 H \z/2 \xrightarrow{Sq^1} \Sigma^9 H \z/2$ zero.
		
		Now $\{\Sigma^3 H \z/2,\Sigma^8 H \z/2\} \cong \mathcal{A}_{2}^{(5)},$ the degree $5$ homogeneous part of the mod $2$ Steenrod algebra $\mathcal{A}_2.$ As a basis of $\mathcal{A}_2$ is given by admissible monomials, 
		\begin{center}
			$\mathcal{A}_2^{(5)}= \z/2$\{admissible monomials in degree 5\}
		\end{center}
		Note that $Sq^4 Sq^1$ and $Sq^5$ are only 2 admissible monomials in degree 5 and $Sq^1 Sq^4 Sq^1 = Sq^5 Sq^1 \neq 0$ and $Sq^1 Sq^5 = 0.$ This proves the lemma.
	\end{proof}
	\begin{proof}[Proof of Theorem \ref{maintheo} ]
		Our goal is to prove $\{\Sigma M,\Tilde{t}\} \cong \z/2 \oplus \bigoplus\limits_{i=1}^{s_2+m}\z/2$ by using the fact that $Sq^4$ acts trivially on $H^*(\Sigma M; \mathbb{Z}/2).$
		
		If $\delta_{\Tilde{t}}=0,$ then there is nothing to prove as $\Tilde{t} \simeq \Sigma^3 H \z/2 \vee \Sigma^7 H \z/2$ implies that $\{\Sigma M, \Tilde{t}\} \cong H^3(\Sigma M;\z/2) \oplus H^7(\Sigma M;\z/2).$
		
		If $\delta_{\Tilde{t}}= Sq^5= Sq^1 Sq^4,$ then we consider the following diagram
		\begin{equation}\label{commutative_diagram}
			\begin{tikzcd}
				{\Tilde{t}} \arrow[r] \arrow[d,"\psi"'] & {\Sigma^3 H \z/2} \arrow[r,"\delta_{\Tilde{t}}"] \arrow[d,"Sq^4"] & {\Sigma^8 H \z/2} \arrow[d,-,double equal sign distance,double]\\
				{\Sigma^7 H \z/4} \arrow[r] & {\Sigma^7 H \z/2} \arrow[r,"Sq^1"'] & {\Sigma^8 H \z/2}
			\end{tikzcd}
		\end{equation}
		Note that $Sq^1=$Bockstein homomorphism $\beta,$ so we get a cofiber sequence 
		\begin{center}
			$H \z/2\to H \z/4 \to H \z/2 \xrightarrow{Sq^1} \Sigma H \z/2 \to ....$
		\end{center}
		\eqref{commutative_diagram} induces the following diagram
		\begin{equation*}
			\begin{tikzcd}[column sep=tiny]
				0\arrow[r]&{\{\Sigma M,\Sigma^7 H \z/2\}\cong\z/2}\arrow[r]\arrow[d,"\cong"']&{\{\Sigma M,\Tilde{t}\}}\arrow[r]\arrow[d,"\psi^*"]&{\{\Sigma M,\Sigma^3 H \z/2\}\cong\bigoplus\limits_{i=1}^{s_2+m}\z/2}\arrow[r]\arrow[d,"0"]&0\\
				0\arrow[r]&{\{\Sigma M,\Sigma^7 H \z/2\}\cong\z/2}\arrow[r]&{\{\Sigma M,\Sigma^7 H \z/4\}\cong\z/4}\arrow[r]&{\{\Sigma M,\Sigma^7 H \z/2\}\cong\z/2}\arrow[r]&0
			\end{tikzcd}
		\end{equation*}
		The map $\{\Sigma M,\Sigma^3 H \z/2\} \to \{\Sigma M,\Sigma^7 H \z/2\}$ is zero as $Sq^4$ operates trivially on $M.$
  
		Suppose $\{\Sigma M,\Tilde{t}\} \cong \z/4 \oplus \bigoplus\limits_{i=1}^{s_2+m-1}\z/2.$ Then the top left arrow sends $\bar{1}$ to the class $(\bar{2},0,0,....,0)$ and the bottom left arrow sends $\bar{1}$ to $\bar{2}.$ Hence $\psi^*(\bar{2},0,0,....,0)=\bar{2}$ and so $\psi^*(\bar{1},0,...,0)=\bar{1}.$ But the top right arrow carries $(\bar{1},0,0,....,0)$ to a non-zero class and the bottom right arrow carries $\bar{1}$ to $\bar{1}.$ This implies that the right most vertical arrow cannot be zero. Thus, we conclude $\{\Sigma M,\Tilde{t}\} \cong \z/2 \oplus \bigoplus\limits_{i=1}^{s_2+m}\z/2.$ The theorem is now derived from Lemma \ref{lemma3.6}.
	\end{proof}

    \begin{cor} \label{cor3.10}
		If $M$ be a simply connected, closed, smooth $6$-dimensional manifold, then the following short exact sequence 
       $$0\to [M\times\mathbb{S}^1,PL/O]\to [M\times\mathbb{S}^1,Top/O]\to [M\times\mathbb{S}^1,Top/PL]\to 0$$
       splits.
  
  Hence, $[M \times \s^1,Top/O] \cong \Theta_7 \oplus H^2(M; \z/2).$
	\end{cor}
\begin{proof}
       We note that the fiber sequence $PL/O\to Top/O\to Top/PL$ induces the following long exact sequence

        \begin{tikzcd}[column sep=1em,font=\small]
	\cdots & {[M\times\mathbb{S}^1, \Omega(Top/PL)]} &{[M\times\mathbb{S}^1,PL/O]} & {[M\times\mathbb{S}^1,Top/O]}\\
	{[M\times\mathbb{S}^1,Top/PL]} & {[M\times\mathbb{S}^1,B(PL/O)]} &{[ M\times \mathbb{S}^1,B(Top/O)]}& \cdots
	\arrow[from=1-1, to=1-2]
	\arrow[from=1-2, to=1-3]
	\arrow[from=1-3, to=1-4]
	\arrow[from=1-4, to=2-1, overlay, out=-7, in=171]
	\arrow[from=2-1, to=2-2]
	\arrow[from=2-2, to=2-3]
	\arrow[from=2-3, to=2-4] 
	\end{tikzcd}
      
To show the above sequence splits at $[M\times\mathbb{S}^1,Top/O],$ we need to establish its splitting for each component of $[M\times\mathbb{S}^1,Top/O],$ as indicated by Corollary \ref{concorMtimesSk}. Since $M$ is a $6$-dimensional manifold, it suffices to show the following long exact sequence
\begin{center}
    \begin{tikzcd}[column sep=1em,font=\small]
	 {0\cong H^1(M;\z/2)} &{[\Sigma M,PL/O]\cong \Theta_7} & {[\Sigma M,Top/O]}\\
	{[\Sigma M,Top/PL]\cong H^2(M;\z/2)} & {[M,\Omega B(PL/O)]\cong [M,PL/O]} &{[M,Top/O]}
	\arrow[from=1-1, to=1-2]
	\arrow[from=1-2, to=1-3]
	\arrow[from=1-3, to=2-1, overlay, out=-5, in=172]
	\arrow[from=2-1, to=2-2]
         \arrow[from=2-2, to=2-3]
	\end{tikzcd}
    \end{center}induced from the fiber sequence $PL/O\to Top/O\to Top/PL$ admits splitting at $[\Sigma M,Top/O].$ However, its splitting follows from Theorem \ref{maintheo}. 
	\end{proof}
     \begin{prop}\label{TopmodO}
    Let $M$ be a simply connected, closed, smooth manifold of dimension $6.$ Then, for any $k\geq 2,$ the short exact sequence $$0\to [M\times\mathbb{S}^k,PL/O]\to [M\times\mathbb{S}^k,Top/O]\to [M\times\mathbb{S}^k,Top/PL]\to 0$$ splits. 
\end{prop}
\begin{proof} 
Upon incorporating the split short exact sequence \eqref{note2.3} for $Y= PL/O, Top/O$ and $Top/PL,$ we obtain the following commutative diagram
    \begin{center}
        \begin{tikzcd}[sep=.6em, font=\scriptsize]
            &{\widetilde{H}^{2-k}(M;\z/2)\cong[\Sigma^k M,\Omega (Top/PL)]}\arrow[r]\arrow[d,"0"']&{[M\times\mathbb{S}^k,\Omega (Top/PL)]}\arrow[r]\arrow[d]&{[M,\Omega (Top/PL)]\oplus\pi_{k+1}(Top/PL)}\arrow[d,"0"]\\
            {0}\arrow[r]&{[\Sigma^k M,PL/O]}\arrow[d]\arrow[r]&{[M\times\mathbb{S}^k,PL/O]}\arrow[d]\arrow[r]&{[M,PL/O]\oplus\pi_k(PL/O)}\arrow[d]\arrow[r]&{0}\\
            {0}\arrow[r]&{[\Sigma^k M,Top/O]}\arrow[d]\arrow[r]&{[M\times\mathbb{S}^k,Top/O]}\arrow[d]\arrow[r]&{[M,Top/O]\oplus\pi_k(Top/O)}\arrow[d]\arrow[r]&{0}\\
             {0}\arrow[r]&{[\Sigma^k M,Top/PL]\cong \widetilde{H}^{3-k}(M;\mathbb{Z}/2)}\arrow[d]\arrow[r]&{[M\times\mathbb{S}^k,Top/PL]}\arrow[d]\arrow[r]&{[M,Top/PL]\oplus\pi_k(Top/PL)}\arrow[d]\arrow[r]&{0}\\
             &{[\Sigma^k M,B(PL/O)]}\arrow[r]&{[M\times\mathbb{S}^k,B(PL/O)]}\arrow[r]&{[M,B(PL/O)]\oplus\pi_{k-1}(PL/O)}
        \end{tikzcd}
    \end{center}  
  where each column is induced from the fiber sequence $\cdots\to\Omega (Top/PL)\to PL/O\to Top/O\to Top/PL.$ 

   Since $Top/PL= \Sigma^3 H\mathbb{Z}/2$ and $PL/O$ is $6$-connected, $[\Sigma^k M,PL/O]\cong [\Sigma^k M,Top/O]$ and $\pi_k(Top/O)\cong \pi_k(PL/O)$ for $k\geq 2.$ Consequently, $[M, Top/O]\cong [M, Top/PL].$

 Applying four lemma in the above diagram, we see that $[M\times \mathbb{S}^k,Top/O]$ fits into the following short exact sequence
 \begin{equation}\label{Mtimes Sk}
        0\to [M\times\mathbb{S}^k,PL/O]\to  [M\times\mathbb{S}^k,Top/O]\to  [M\times\mathbb{S}^k,Top/PL]\to 0.
    \end{equation}
    By considering the splitting of the short exact sequences 
    $$0\to [M\vee\mathbb{S}^k,PL/O] \to [M\vee\mathbb{S}^k,Top/O]\to  [M\vee\mathbb{S}^k,Top/PL]\to 0$$ and $$0\to [\Sigma^k M,PL/O]\to [M\times\mathbb{S}^k,PL/O]\to [M,PL/O]\oplus\pi_{k}(PL/O)\to 0$$ in the above diagram, we get \eqref{Mtimes Sk} splits. 
\end{proof}

  Now instead of finding $\mathcal{C}(M\times\mathbb{S}^k)$ for any simply connected, closed, smooth $6$-manifold, we focus on finding $\mathcal{C}(\mathbb{C}P^3\times\mathbb{S}^k)$ and thereby calculating $[\Sigma^k \mathbb{C}P^3,Top/O].$ To do so we use the following long exact sequence 
  \begin{equation}\label{longexactCP3}
      \begin{tikzcd}[column sep=1em,font=\small]
	\cdots & {[\Sigma^{k+1}\mathbb{C}P^2, Top/O]} &{\pi_{6+k}(Top/O)} & {[\Sigma^k\mathbb{C}P^3,Top/O]}\\
	{[\Sigma^k\mathbb{C}P^2,Top/O]} & {\pi_{5+k}(Top/O)} &{[\Sigma^{k-1}\mathbb{C}P^3, Top/O]}& \cdots
	\arrow[from=1-1, to=1-2]
	\arrow["(\Sigma^{k+1}\pi)^*",from=1-2, to=1-3]
	\arrow["(\Sigma^k f_{\mathbb{C}P^3})^*",from=1-3, to=1-4]
	\arrow["(\Sigma^k i')^*",from=1-4, to=2-1, overlay, out=-7, in=171]
	\arrow["(\Sigma^k \pi)^*"',from=2-1, to=2-2]
	\arrow["(\Sigma^{k-1}f_{\mathbb{C}P^3})^*"',from=2-2, to=2-3]
	\arrow[from=2-3, to=2-4] 
	\end{tikzcd}
\end{equation} induced from $\mathbb{S}^5\xrightarrow{\pi}\mathbb{C}P^2\xhookrightarrow{i'} \mathbb{C}P^3.$

 \begin{prop}\label{CP3computation}
 	\noindent
 	\begin{itemize}
 		\item [(i)] $[\Sigma^k \mathbb{C}P^3,Top/O]\cong \Theta_{6+k},$ for $k=2$ and $8.$
             \item[(ii)] $[\Sigma^k \mathbb{C}P^3,Top/O]\cong [\Sigma^k \mathbb{C}P^2,Top/O],$ if $k=6$ and $7.$
 		\item [(iii)] $[\Sigma^k \mathbb{C}P^3,Top/O]\cong \Theta_{6+k}\oplus[\Sigma^k\mathbb{C}P^2,Top/O]$ for $k=3,4,5,9$ and $10.$
 	\end{itemize}
 \end{prop}
\begin{proof}
  We note the result for $k=2,6,7$ and $8$ can be deduced from the long exact sequence (\ref{longexactCP3}) using \cite[Proposition 3.7]{SBRKAS} and Lemma \ref{lemma2.16}. 
  
	Since $\mathbb{C}P^3/{\mathbb{S}^2}\simeq \mathbb{S}^4\vee\mathbb{S}^6,$ we have another cofiber sequence involving the complex $\mathbb{C}P^3,$
	\begin{equation}\label{longexactCP3new}
		\s^2\xhookrightarrow{\tilde{\tilde{i}}} \mathbb{C}P^3\xrightarrow{q}\s^4\vee\s^6\xrightarrow{\eta+ \phi}\s^3\hookrightarrow\cdots,
	\end{equation}
 where $(\phi)_{(2)}$ is stably $2\nu$ and $(\phi)_{(3)}$ is stably $\alpha_1.$
 
 Now for $k=3$ and $10,$ utilizing the long exact sequence induced from this cofiber sequence, we have $[\Sigma^k \mathbb{C}P^3,Top/O]\cong \Theta_{4+k}\oplus\Theta_{6+k}$ which is again isomorphic to $[\Sigma^k\mathbb{C}P^2,Top/O]\oplus\Theta_{6+k}$ by \cite[Proposition 3.7 {(2)} and {(9)}]{SBRKAS}. 
 
 For $k=4,$ there is a ladder of exact sequence 
 \begin{center}
 	\begin{tikzcd}[row sep=1 em]
 		&{0}\arrow[d]\arrow[r]&{[\Sigma^5 \mathbb{C}P^3,G/Top]}\arrow[d,"\omega_*"]\arrow[r]&{0}\arrow[d]\\
 		{\Theta_7}\arrow[r]\arrow[d,"\psi_*"']&{\Theta_8\oplus\Theta_{10}}\arrow[r,"(\Sigma^4 q)^*"]\arrow[d,"\psi_*"']&{[\Sigma^4 \mathbb{C}P^3,Top/O]}\arrow[r]\arrow[d,"\psi_*"]&{\Theta_6\cong 0}\\
 		{0\cong\pi_7(G/O)}\arrow[r]&{\pi_8(G/O)\oplus\pi_{10}(G/O)}\arrow[r,"(\Sigma^4 q)^*"']&{[\Sigma^4 \mathbb{C}P^3,G/O]}
 	\end{tikzcd}
 \end{center} where the rows are induced from the cofiber sequence (\ref{longexactCP3new}) and columns are induced from the fiber sequence $\Omega(G/Top)\xrightarrow{\omega}Top/O\xrightarrow{\psi}G/O.$ Now clearly the above diagram and \cite[Proposition 3.7 {(3)}]{SBRKAS} gives $[\Sigma^4 \mathbb{C}P^3,TOP/O]\cong\Theta_{10}\oplus\Theta_8\cong\Theta_{10}\oplus [\Sigma^4\mathbb{C}P^2,Top/O].$

 For $k=5,$ we have the following short exact sequence induced from cofiber sequence (\ref{longexactCP3new}) along $Top/O$ 
 \begin{equation}\label{cp3s5}
 0\to\z/2\oplus\z/2\oplus\z/{992}\to[\Sigma^5 \mathbb{C}P^3,Top/O]\to\z/{28}\to0    
 \end{equation}
  by using Lemma \ref{image_eta} and Lemma \ref{imagenu}. Examining \eqref{cp3s5}, it suffices to compute $2$-primary component of $[\Sigma^5 \mathbb{C}P^3,Top/O].$ Now the above short exact sequence narrows down the possibilities for $[\Sigma^5 \mathbb{C}P^3,Top/O]_{(2)}$ to three abelian groups: $\z/2\oplus\z/2\oplus\z/{2^5}\oplus\z/4, \z/2\oplus\z/8\oplus\z/{2^5}$ or $\z/2\oplus\z/2\oplus\z/{2^7}.$ 
 
 Again by \cite[Proposition 3.7 (4)]{SBRKAS}, and Theorem \ref{prop2.9} in \eqref{longexactCP3}, we have the following short exact sequence
 $$0\to \Theta_{11}\to [\Sigma^5 \mathbb{C}P^3,Top/O]\to \z/{56}\oplus\z/2\to 0.$$
 This short exact sequence constrains $[\Sigma^5 \mathbb{C}P^3,Top/O]_{(2)}$ to one of the three possible forms $\z/{2^5}\oplus\z/{2^3}\oplus\z/2, \z/{2^8}\oplus\z/{2}$ or $\z/{2^3}\oplus\z/{2^6}.$
 Thus combining possibilities from both short exact sequences, we conclude that $[\Sigma^5 \mathbb{C}P^3,Top/O]\cong \Theta_{11}\oplus [\Sigma^5\mathbb{C}P^2,Top/O].$
 
 	For $k=9,$ we have the following short exact sequence induced from the cofiber sequence \eqref{longexactCP3new} along $Top/O$ 
$$0\to\Theta_{15}\oplus\Theta_{13}\xrightarrow{(\Sigma^9 q)^*}[\Sigma^9 \mathbb{C}P^3,Top/O]\xrightarrow{(\Sigma^9 \tilde{\tilde{i}})^*}\Theta_{11}\to0,$$
  as $\Theta_{12}=0$ and any map $\Theta_{11}\to\Theta_{14}$ is zero. Hence it is evident that $[\Sigma^9 \mathbb{C}P^3,Top/O]_{(2)}$ can be of the form $\z/{2^6}\oplus\z/2\oplus\z/{2^5}, \z/{2^{11}}\oplus\z/2$ or $\z/{2^6}\oplus\z/{2^6}.$
 
 Since the image of the map $\pi_{12}(G/O)\to\pi_{15}(G/O)$ is zero, $[\Sigma^9 \mathbb{C}P^3,G/O]_{(2)}\cong \z/2,$ and $[\Sigma^{10} \mathbb{C}P^3,G/O]\cong\pi_{16}(G/O)\oplus\pi_{14}(G/O)\oplus\pi_{12}(G/O)$ from long exact sequence induced from cofiber sequence (\ref{longexactCP3new}) along $G/O.$ The same cofiber sequence over $G/Top$ gives $[\Sigma^{10} \mathbb{C}P^3,G/Top]\cong\pi_{16}(G/Top)\oplus\pi_{14}(G/Top)\oplus\pi_{12}(G/Top).$ Hence, from the splitting, we obtain the image of  $(\phi_*):[\Sigma^{10} \mathbb{C}P^3, G/O]\to[\Sigma^{10} \mathbb{C}P^3, G/Top]$ is $8128\z\oplus\z/2\oplus 992\z.$ Therefore the map $(\omega_*)_{(2)}: [\Sigma^{10}\mathbb{C}P^3, G/Top]_{(2)}\to [\Sigma^{9}\mathbb{C}P^3, Top/O]_{(2)}$ has image $\mathbb{Z}/{2^6}\oplus\mathbb{Z}/{2^5}.$ Hence the possibility of $[\Sigma^9 \mathbb{C}P^3,Top/O]_{(2)}$ being  $\z/{2^{11}}\oplus\z/2$ is ruled out. Suppose $[\Sigma^9 \mathbb{C}P^3, Top/O]_{(2)}$ be $\mathbb{Z}/{2^6}\oplus\mathbb{Z}/{2^6}.$
 Now, to show this cannot occur, consider the following ladder of exact sequences induced from the cofiber sequence \eqref{longexactCP3new} and the fiber sequence $\cdots\to \Omega(G/Top)\xrightarrow{\omega}Top/O\xrightarrow{\psi}G/O\xrightarrow{\phi} G/Top.$ 
 
 	\begin{tikzcd}[column sep=1em, font= \small]
 		0\arrow[r]\arrow[d]&{\pi_{16}(G/O)_{(2)}\oplus\pi_{14}(G/O)_{(2)}}\arrow[r,"(\Sigma^{10}q)_{(2)}^*"]\arrow[d,"(\phi_*)_{(2)}"']&{[\Sigma^{10}\mathbb{C}P^3,G/O]_{(2)}}\arrow[r,"(\Sigma^{10}\tilde{\tilde{i}})_{(2)}^*"]\arrow[d,"(\phi_*)_{(2)}"]&{\pi_{12}(G/O)_{(2)}}\arrow[r,"0"]\arrow[d,"(\phi_*)_{(2)}"]&{\pi_{15}(G/O)_{(2)}}\arrow[d]\\
 		0\arrow[r]\arrow[d]&{\z_{(2)}\oplus\z/2}\arrow[r,"(\Sigma^{10}q)_{(2)}^*"]\arrow[d,"(\omega_*)_{(2)}"']&{[\Sigma^{10} \mathbb{C}P^3,G/Top]_{(2)}}\arrow[r,"(\Sigma^{10}\tilde{\tilde{i}})_{(2)}^*"]\arrow[d,"(\omega_*)_{(2)}"]&{\z_{(2)}}\arrow[r]\arrow[d,"(\omega_*)_{(2)}"]&{0}\arrow[d]\\
 		0\arrow[r]\arrow[d]&{\z/{2^6}\oplus\z/2}\arrow[r,"(\Sigma^{9}q)_{(2)}^*"]\arrow[d,"(\psi_*)_{(2)}"',two heads]&{\z/{2^6}\oplus\z/{2^6}}\arrow[r,"(\Sigma^{9}\tilde{\tilde{i}})_{(2)}^*"]\arrow[d,"(\psi_*)_{(2)}", two heads]&{\z/{2^5}}\arrow[r,"0"]\arrow[d]&{\z/2}\\
 		{\z_{(2)}}\arrow[r,"0"']&{\z/2}\arrow[r,"\cong"']&{[\Sigma^9 \mathbb{C}P^3,G/O]_{(2)}}\arrow[r]&{0} \\
 	\end{tikzcd}
 
From the above diagram, we observe $(\Sigma^9 \tilde{\tilde{i}})_{(2)}^*\circ (\omega_*)_{(2)}:[\Sigma^{10}\mathbb{C}P^3, G/Top]_{(2)}\to (\Theta_{11})_{(2)}$  has image of order $2^4.$ But this contradicts that the map $\omega_*\circ(\Sigma^{10} \tilde{\tilde{i}})^*:[\Sigma^{10} \mathbb{C}P^3, G/Top]\to \Theta_{11}$ is an onto map. Therefore $[\Sigma^{9} \mathbb{C}P^3, Top/O]_{(2)} $ cannot be $\z/{2^6}\oplus\z/{2^6}.$
 
 Thus $[\Sigma^9 \mathbb{C}P^3, Top/O]\cong \Theta_{15}\oplus\Theta_{13}\oplus\Theta_{11}$ which is same as $\Theta_{15}\oplus [\Sigma^9\mathbb{C}P^2,Top/O]$ by \cite[Proposition 3.7 {(8)}]{SBRKAS}. 
\end{proof}


\section{Diffeomorphic Classification of \texorpdfstring{$\mathbb{C}P^3\times\mathbb{S}^k$}{}}
In this section, we calculate $\mathcal{S}(\mathbb{C}P^3\times\mathbb{S}^k)$ for all $1\leq k\leq 7,$ using the information obtained from Proposition \ref{CP3computation}. We use the notations and Facts of \cite[Section 5]{SBRKAS} in this section.

Additionally, we briefly recall the following facts.

\begin{fact}{\cite[Proposition 10.2]{bel}}\label{normal invariant of SG}
    Let $M$ be a closed, connected, smooth $n$-dimensional manifold. Suppose $k\geq 2$ and $n+k\geq 5.$ Then the normal invariant of any self-homotopy equivalence $f: M\times\mathbb{S}^k\to M\times\mathbb{S}^k$ arising from an element of $[M, SG_{k+1}]$ lies in the image of $[M, G/O] \subset [M\times\mathbb{S}^k, G/O].$ 
\end{fact}
\begin{fact}{\cite[Page-144]{schultz}}\label{identity element description}
    Let $M$ be a closed, oriented, smooth manifold of dimension $n\geq 5.$ Let $\Theta^{Top}: [\Sigma M, G/Top]\to L_{n+1}(\pi_1(M))$ be the topological surgery obstruction map, and $\omega : \Omega (G/Top)\to Top/O$ be the canonical map in the extended fiber sequence $Top/O\to G/O \to G/Top.$ Then the concordance classes of smoothing on $M$ that are equivalent to $(M, id)$ in $\mathcal{S}^{Diff}(M)$ are the elements of the set $\omega_* (\ker(\Theta^{Top}))\subset [M, Top/O],$ where $\mathcal{S}^{Diff}(M)$ is the \textit{simple structure set} \cite{Luck_surgery}. 
\end{fact}
The following lemma plays a key role in proving classification results.
\begin{lemma}\label{psicp3sk}
    \noindent
    \begin{itemize}
        \item[(i)] The map $\psi_*:[\mathbb{C}P^3\times\mathbb{S}^1, Top/O]\to [\mathbb{C}P^3\times\mathbb{S}^1, G/O]$ is trivial.
        \item[(ii)]  The image of $\psi_*:[\mathbb{C}P^3\times\mathbb{S}^2, Top/O]\to[\mathbb{C}P^3\times\mathbb{S}^2, G/O]$ is $\mathbb{Z}/2\{(f_{\mathbb{C}P^3\times\mathbb{S}^2})^*([\epsilon])\}.$  
        \item[(iii)] The image of $\psi_* :[\mathbb{C}P^3\times\mathbb{S}^3, Top/O] \to [\mathbb{C}P^3\times\mathbb{S}^3,G/O]$ is $\mathbb{Z}/2\oplus\mathbb{Z}/2.$ Moreover, the image is generated by $(f_{\mathbb{C}P^3\times\mathbb{S}^3})^*([\epsilon\eta])$ and $(f_{\mathbb{C}P^3\times\mathbb{S}^3})^*([\mu]).$ 
        \item[(iv)] The image of $\psi_*: [\mathbb{C}P^3\times \mathbb{S}^4, Top/O]\to [\mathbb{C}P^3\times\mathbb{S}^4, G/O]$ is $\mathbb{Z}/2\{(\Sigma^4 Pr\circ \Sigma^4 q)^*[\epsilon]\}\oplus\mathbb{Z}/2\{(f_{\mathbb{C}P^3\times\mathbb{S}^4})^*[\mu]\}\oplus\mathbb{Z}/3\{(f_{\mathbb{C}P^3\times\mathbb{S}^4})^*[\beta_1]\}.$ 
        \item[(v)] $\psi_*:[\mathbb{C}P^3\times\mathbb{S}^5, Top/O]\to[\mathbb{C}P^3\times\mathbb{S}^5, G/O]$ is a trivial map. 
        \item[(vi)] The image of $\psi_*: [\mathbb{C}P^3\times\mathbb{S}^6, Top/O]\to [\mathbb{C}P^3\times\mathbb{S}^6, G/O]$ is $\mathbb{Z}/3,$ generated by $(\Sigma^6 Pr \circ \Sigma^6 q)^*[\beta_1].$ 
        \item[(vii)] The image of $\psi_* : [\mathbb{C}P^3\times\mathbb{S}^7, Top/O]\to [\mathbb{C}P^3\times\mathbb{S}^7, G/O]$ is $\mathbb{Z}/2$ and is generated by $[\nu^3].$  
    \end{itemize}
    where $Pr: \mathbb{C}P^3/{\mathbb{C}P^1}\simeq\mathbb{S}^4\vee \mathbb{S}^6\to \mathbb{S}^4$ denotes the projection map.
\end{lemma}
\begin{proof}
    Since the short exact sequence \eqref{note2.3} splits along both $Top/O$ and $G/O,$ there is a commutative diagram
    \begin{center}
        \begin{tikzcd}
             {0}\arrow[r]&{[\Sigma^k \mathbb{C}P^3, Top/O]}\arrow[r,"p^*"]\arrow[d,"\psi_*"']&{[\mathbb{C}P^3\times\mathbb{S}^k,Top/O]}\arrow[r,"i^*"]\arrow[d,"\psi_*"]&{[\mathbb{C}P^3\vee \s^k, Top/O]}\arrow[r]\arrow[d,"\psi_*"]&{0}\\
           {0}\arrow[r]&{[\Sigma^k\mathbb{C}P^3, G/O]}\arrow[r,"p^*"']&{[\mathbb{C}P^3\times\s^k, G/O]}\arrow[r,"i^*"']&{[\mathbb{C}P^3\vee\s^k, G/O]}\arrow[r]&{0}
        \end{tikzcd}
    \end{center}
    As $[\mathbb{C}P^3, Top/O], \pi_l(Top/O), l=1,2,4,6$ and $\pi_j(G/O), j=3,5,7$ are all vanish, the image of $\psi_*:[\mathbb{C}P^3\times\mathbb{S}^k, Top/O]\to [\mathbb{C}P^3\times\mathbb{S}^k, G/O]$ coincides with that of $\psi_* : [\Sigma^k \mathbb{C}P^3, Top/O]\to [\Sigma^k \mathbb{C}P^3, G/O]$ for $1\leq k\leq 7.$ 

{(i)} As $[\Sigma \mathbb{C}P^3, G/O]$ is zero, $\psi_*:[\mathbb{C}P^3\times\mathbb{S}^1, Top/O]\to [\mathbb{C}P^3\times\mathbb{S}^1, G/O]$ is a zero map.

{(ii)} From Proposition \ref{CP3computation} {(i)}, we have $[\Sigma^2\mathbb{C}P^3, Top/O]=\Theta_8.$ Hence the image of $\psi_* : [\Sigma^2 \mathbb{C}P^3, Top/O]\to [\Sigma^2 \mathbb{C}P^3, G/O]$ is equal to the image of $(\Sigma^2f_{\mathbb{C}P^3})^*\circ\psi_*:\Theta_8\to [\Sigma^2\mathbb{C}P^3, G/O],$ which is $\mathbb{Z}/2\{(\Sigma^2 f_{\mathbb{C}P^3})^*([\epsilon])\}.$ Since $(f_{\mathbb{C}P^3\times\mathbb{S}^2})^*=p^*\circ (\Sigma^k f_{\mathbb{C}P^3})^*,$ the result follows.

{(iii)} We note from the long exact sequence \eqref{longexactCP3} with $X=G/O$ that $(\Sigma^3 f_{\mathbb{C}P^3})^*: \pi_9(G/O)\to [\Sigma^3 \mathbb{C}P^3, G/O]$ is an isomorphism, as  $(\Sigma^4\pi)^*: [\Sigma^4\mathbb{C}P^2, G/O] \to \pi_9(G/O)$ is a zero map and $[\Sigma^3\mathbb{C}P^2, G/O]=0$. Therefore Proposition \ref{CP3computation} {(iii)} for $k=3$ implies that the image of $\psi_* : [\Sigma^3 \mathbb{C}P^3, Top/O]\to [\Sigma^3 \mathbb{C}P^3, G/O]$ is equal to the image of $(f_{\mathbb{C}P^3\times\mathbb{S}^3})^*\circ \psi_*: \Theta_9\to [\Sigma^3\mathbb{C}P^3, Top/O],$ which is $\mathbb{Z}/2\{f_{\mathbb{C}P^3\times\mathbb{S}^3}^*([\epsilon\eta])\}\oplus \mathbb{Z}/2\{f_{\mathbb{C}P^3\times\mathbb{S}^3}^*([\mu])\}.$ 

{(iv)} We observe from \cite[Theorem 5.4]{SBRKAS} that $\psi_*:[\Sigma^4\mathbb{C}P^2, Top/O]\to [\Sigma^4 \mathbb{C}P^2, G/O]$ has image $(\Sigma^4 f_{\mathbb{C}P^2})^*[\epsilon].$ combining this with the splitting of $[\Sigma^4\mathbb{C}P^3, Top/O]$ given in Proposition \ref{CP3computation}, the result follows.

{(v)} Since $\eta^*: \pi_8(G/O)\to \pi_9(G/O)$ is onto, the long exact sequence induced from \eqref{longexactCP3new} implies that $[\Sigma^5\mathbb{C}P^3, G/O]$ is trivial, establishing the result.

{(vi)} Since $\psi_*:\Theta_{10}\to \pi_{10}(G/O)$ is an isomorphism and $[\Sigma^6\mathbb{C}P^3, Top/O]\cong \mathbb{Z}/3 \subset \Theta_{10}$ by Proposition \ref{CP3computation} {(ii)}, the result follows. 

{(vii)} Using Proposition \ref{CP3computation} {(ii)}, we need to find the image $\psi_*: [\Sigma^7\mathbb{C}P^2, Top/O]\to [\Sigma^7\mathbb{C}P^2, G/O].$ But $[\Sigma^7\mathbb{C}P^2, G/O]$ coincides with $\mathbb{Z}/2\{[\nu^3]\}.$ Therefore from the splitting \cite[Proposition 3.7 {(6)}]{SBRKAS}, we have the result.

\end{proof}

We now proceed to classify all manifolds homeomorphic to $\mathbb{C}P^3 \times \mathbb{S}^1,$ up to diffeomorphism.
\begin{lemma}\label{identity element}
The image of the restricted map $\omega_* : \ker (\Theta^{Top})\to \mathcal{C}(\mathbb{C}P^3\times\mathbb{S}^1)$ is $\mathbb{Z}/7\oplus\mathbb{Z}/2.$ 
    
\end{lemma}
\begin{proof}
  Using Sullivan's identification, we obtain
  {\tiny
    \begin{align*}
        [\mathbb{C}P^3\times\mathbb{S}^1, \Omega (G/Top)]_{(2)}&\xrightarrow {\cong} H^7(\mathbb{C}P^3\times\mathbb{S}^1;\mathbb{Z}_{(2)})\oplus H^5(\mathbb{C}P^3\times\mathbb{S}^1;\mathbb{Z}/2)\oplus H^3(\mathbb{C}P^3\times\mathbb{S}^1;\mathbb{Z}_{(2)})\oplus H^1(\mathbb{C}P^3\times\mathbb{S}^1;\mathbb{Z}/2) \\
           f&\mapsto(f^*(l_8), f^*(\kappa_6), f^*(l_4), f^*(\kappa_2))
    \end{align*}
    }
 From \cite{wallsurgerybook}, the surgery obstruction is given by
   \[
    \Theta_{(2)}^{Top}(0,0,n,0)= \langle (1+\frac{4}{3}z^2)2nz\sigma, [\mathbb{C}P^3\times\mathbb{S}^1\times\mathbb{S}^1]\rangle= \frac{8}{3}n,\]
\[\Theta_{(2)}^{Top}(m,0,0,0)= m,~ \Theta_{(2)}^{Diff}(0,\bar{1},0,0)=0 ~\text{and}~ \Theta_{(2)}^{Diff}(0,0,0,\bar{1})=0,\]
where $z\in H^2(\mathbb{C}P^3;\mathbb{Z})$ and $\sigma\in H^1(\mathbb{S}^1;\mathbb{Z})$ are the generators. Therefore the kernel of the surgery obstruction $\Theta_{(2)}^{Top}$ is clearly $\frac{8}{3}\mathbb{Z}_{(2)}\oplus\mathbb{Z}/2\oplus\mathbb{Z}_{(2)}\oplus\mathbb{Z}/2.$ Now consider the following commutative diagram
\begin{center}
    \begin{tikzcd}
        &{[\mathbb{C}P^3\times\mathbb{S}^1,\Omega (G/Top)]}\arrow[r,"(P_{(2)}^{G/Top})_*"]\arrow[d,"\Theta^{Top}"']&{[\mathbb{C}P^3\times\mathbb{S}^1,\Omega (G/Top)]_{(2)}}\arrow[d,"\Theta_{(2)}^{Top}"]\\
       {0}\arrow[r] &{L_8(\mathbb{Z})\cong \mathbb{Z}}\arrow[r]&{L_8(\mathbb{Z})_{(2)}\cong\mathbb{Z}_{(2)}}
    \end{tikzcd}
\end{center}
Since $\mathbb{C}P^3\times\mathbb{S}^1$ has no odd torsion in homology, $(P_{(2)}^{G/Top})_*:[\mathbb{C}P^3\times\mathbb{S}^1,\Omega (G/Top)]\to [\mathbb{C}P^3\times\mathbb{S}^1,\Omega (G/Top)]_{(2)}$ is a monomorphism \cite{rudyak15} and $[\mathbb{C}P^3\times\mathbb{S}^1, \Omega (G/Top)]\cong \bigoplus\limits_{i=1}^2 H^{4i-1}(\mathbb{C}P^3\times\mathbb{S}^1;\mathbb{Z})\oplus \bigoplus\limits_{j=0}^1 H^{4j+1}(\mathbb{C}P^3\times\mathbb{S}^1;\mathbb{Z}/2).$ As the kernel of $\Theta_{(2)}^{Top}:[\mathbb{C}P^3\times\mathbb{S}^1,\Omega G/Top]_{(2)}\to L_8(\mathbb{Z})_{(2)}$ is $\frac{8}{3}\mathbb{Z}_{(2)}\oplus\mathbb{Z}/2\oplus\mathbb{Z}_{(2)}\oplus\mathbb{Z}/2,$ and $(P_{(2)}^{G/Top})_*$ is injective, the kernel of $\Theta_{(2)}^{Top}\circ (P_{(2)}^{G/Top})_*:[\mathbb{C}P^3\times\mathbb{S}^1,\Omega (G/Top)]\to L_8(\mathbb{Z})_{(2)}$ is $8\mathbb{Z}\oplus\mathbb{Z}/2\oplus 3\mathbb{Z}\oplus\mathbb{Z}/2.$ Noting that $L_8(\mathbb{Z})\to L_8(\mathbb{Z})_{(2)}$ is injective, it follows from the commutative diagram that the kernel of $\Theta^{Top}:[\mathbb{C}P^3\times\mathbb{S}^1, \Omega(G/Top)]\to L_8(\mathbb{Z})$ coincides with the kernel of the composition $\Theta_{(2)}^{Top}\circ (P_{(2)}^{G/Top})_*$.

Since $[\mathbb{C}P^3\times\mathbb{S}^1, \Omega (G/Top)]\cong \bigoplus\limits_{i=1}^2 H^{4i-1}(\mathbb{C}P^3\times\mathbb{S}^1;\mathbb{Z})\oplus \bigoplus\limits_{j=0}^1 H^{4j+1}(\mathbb{C}P^3\times\mathbb{S}^1;\mathbb{Z}/2), $ and $\displaystyle (\tau_{\leq 7}(top/o))_{(2)}\simeq \Sigma^7 H\mathbb{Z}/4\vee\Sigma^3 H\mathbb{Z}/2,$ together with $\{\Sigma^3 H\mathbb{Z}, \Sigma^7 H\mathbb{Z}/{28}\}=0,$ the image of $\omega_*(\ker(\Theta^{Top}))$ is generated by $\{(\bar{8},\bar{3})\}\in H^7(\mathbb{C}P^3\times\mathbb{S}^1;\mathbb{Z}/{28})\oplus H^3(\mathbb{C}P^3\times\mathbb{S}^1;\mathbb{Z}/2).$

    This completes the proof.
\end{proof}

\begin{thm}\label{classificationCP3S1}
\noindent
\begin{itemize}
    \item[(i)] $\mathbb{Z}/4 \subset \Theta_7$ is mapped nontrivially under the forgetful map $\mathcal{F} : \mathcal{C}(\mathbb{C}P^3\times\mathbb{S}^1)\to \mathcal{S}^{Diff}(\mathbb{C}P^3\times\mathbb{S}^1).$ 
    \item[(ii)] Any closed, oriented manifold homeomorphic to $\mathbb{C}P^3\times\mathbb{S}^1$ is diffeomorphic to exactly one of the manifolds $\mathbb{C}P^3\times\mathbb{S}^1$ or $(\mathbb{C}P^3\times\mathbb{S}^1)\# \Sigma,$ where $\Sigma$ is a nontrivial element of $\mathbb{Z}/4\subset \Theta_7.$ 
\end{itemize}
\end{thm}
\begin{proof}
    Combining Lemma \ref{identity element} with Fact \ref{identity element description}, we obtain statement {(i)}.

    Statement {(ii)} follows from {(i)}.
\end{proof}

We recall that a self-homotopy equivalence $f: X \times Y \to X \times Y$ is diagonalizable if the compositions $p_X \circ f \circ i_X: X \to X$ and $p_Y \circ f \circ i_Y: Y \to Y$ are self-homotopy equivalences of $X$ and $Y$, respectively. Here, $p_X$ and $p_Y$ denote the projections of $X \times Y$ onto $X$ and $Y$, respectively, while $i_X$ and $i_Y$ denote the inclusions of $X$ and $Y$ into $X \times Y$. Here are some important facts about self-homotopy equivalences of the product $X\times Y.$
\begin{fact}\label{self homotopy equivalence of product}
    \noindent
    \begin{itemize}
        \item[(a)] If every element of $\mathcal{E}(X\times Y)$ is diagonalizable, then any $f\in \mathcal{E}(X\times Y)$ can be written as $f_1.f_2,$ where $f_1\in \mathcal{E}_{X}(X\times Y)$ and $f_2\in \mathcal{E}_{Y}(X\times Y)$ \cite[Theorem 2.5]{PP99}. Here, $\mathcal{E}_X(X\times Y)$ (respectively, $\mathcal{E}_Y(X\times Y)$) consists of all $g\in \mathcal{E}(X\times Y)$ such that $g=(p_X, p_Y\circ g)$ (respectively, $g=(p_{X}\circ g, p_{Y})$).

      \item[(b)] If $X$ is a connected CW complex, then $\mathcal{E}_X(X\times Y)$ fits into the following short exact sequence
       
       $$ 0\to [X, E_1(Y)]\to \mathcal{E}_X(X\times Y) \to \mathcal{E}(Y)\to 0,$$ where $E_1(Y)$ denotes the identity component of the self-maps of 
$Y$ \cite[Proposition 2.3 {(d)}]{PP99}.
    \end{itemize}
\end{fact}

Now to classify smooth manifolds homeomorphic to $\mathbb{C}P^3\times\mathbb{S}^2,$ we first analyze the self-homotopy equivalences of $\mathbb{C}P^3\times\mathbb{S}^2.$

\begin{lemma}\label{diagonalizability}
    Any self-homotopy equivalence of $\mathbb{C}P^3\times\mathbb{S}^2$ is diagonalizable.
\end{lemma}
\begin{proof}
Let \( f: \mathbb{CP}^3 \times S^2 \to \mathbb{CP}^3 \times S^2 \) be a self-homotopy equivalence. To prove that \( f \) is diagonalizable, we examine the induced map on cohomology.

The cohomology ring of \( \mathbb{CP}^3 \times S^2 \) is:
\[ H^*(\mathbb{CP}^3 \times S^2; \mathbb{Z}) \cong \mathbb{Z}[x]/(x^4) \otimes \mathbb{Z}[y]/(y^2), \]
where \( x \) generates \( H^2(\mathbb{CP}^3; \mathbb{Z}) \) and \( y \) generates \( H^2(S^2; \mathbb{Z}) \). Thus:
\[ H^2(\mathbb{CP}^3 \times S^2; \mathbb{Z}) \cong \mathbb{Z} \oplus \mathbb{Z} \]
with generators \( \alpha \) corresponding to \( x \) and \( \beta \) corresponding to \( y \).

The induced map \( f^* \) on cohomology is determined by its action on these generators:
\[ f^*(\alpha) = a \alpha + b \beta \]
\[ f^*(\beta) = c \alpha + d \beta \]
which can be represented by the matrix:
\[ f^* = \begin{pmatrix}
    a & b \\
    c & d
\end{pmatrix}. \]

To ensure \( f \) is a self-homotopy equivalence, \( f^* \) must preserve the cup product structure. Specifically, we examine the image of \( \alpha^3 \beta \) under \( f^* \):
\[
f^*(\alpha^3 \beta) = f^*(\alpha)^3 f^*(\beta).
\]

Since $\alpha^4=0$ and $\beta^2=0$, it follows that:
\begin{align*}
    &f^*(\alpha)^3 = (a \alpha + b \beta)^3 = a^3 \alpha^3 + 3 a^2 b\alpha^2\beta
\end{align*}

\begin{align*} &f^*(\alpha^3\beta) =f^*(\alpha)^3 f^*(\beta)= (a^3 \alpha^3 + 3a^2 b \alpha^2 \beta)(c \alpha + d \beta)= (a^3 d+ 3a^2bc)\alpha^3\beta
\end{align*}

For \( f^* \) to be a self-homotopy equivalence, \( f^*(\alpha^3 \beta) \) must be \( \pm \alpha^3 \beta \), leading to two cases :
\begin{itemize}
\item If \( f^*(\alpha^3 \beta) = \alpha^3 \beta \), then:
\[
a^2(3 b c + a d) = 1.
\]
   This yields: either
\[
(a, b) = (1, k_1) \text{ and } (c, d) = (k_2, 1 - 3k_1 k_2).
\]
or
\[
(a,b)=(-1, k_1) \text{and} (c,d)=(k_2,-1+ 3k_1 k_2)
\]
\item  If \( f^*(\alpha^3 \beta) = -\alpha^3 \beta \), then:
\[
a^2(3 b c + a d) = -1.
\]
This yields: either 
\[
(a,b)=(1,k_1) \text{and} (c,d)=(k_2, -1-3k_1k_2)
\]
or
\[
(a, b) = (-1, k_1) \text{ and } (c, d) = (k_2, 1 + 3k_1 k_2).
\]

\end{itemize}
Since any map $\mathbb{C}P^3\to \mathbb{S}^2$ can be written as $g\circ f_{\mathbb{C}P^3},$ where $g\in \pi_6(\mathbb{S}^2),$ it follows that $c=0.$
Therefore,
\[
d = \pm 1.
\]
With \( c = 0 \), the matrix \( f^* \) simplifies to:
\[
\begin{pmatrix}
    \pm 1 & b \\
    0 & \pm 1
\end{pmatrix}.
\]
Consequently, both compositions \( p_{\mathbb{CP}^3} \circ f \circ i_{\mathbb{CP}^3}: \mathbb{CP}^3 \to \mathbb{CP}^3 \) and \( p_{S^2} \circ f \circ i_{S^2}: S^2 \to S^2 \) are self-homotopy equivalences.

Hence, \( f \) is diagonalizable, completing the proof.
\end{proof}

\begin{prop}\label{selfhomotopycp3s2}
    There is no self-homotopy equivalence $f$ of $\mathbb{C}P^3\times\mathbb{S}^2$ whose normal invariant $\eta^{Diff}(f)= (f_{\mathbb{C}P^3\times\mathbb{S}^2})^*([\epsilon]).$
\end{prop}
\begin{proof}
     We note from Fact \ref{self homotopy equivalence of product}, Lemma \ref{diagonalizability} and \cite[Fact 5.2 {(d)} and {(e)}]{SBRKAS} that in order to compute the normal invariant of a self-homotopy equivalence of $\mathbb{C}P^3\times\mathbb{S}^2,$ it suffices to compute the normal invariant of the self-homotopy equivalences coming from each part of the following two split short exact sequences 
    \begin{equation*} 
        0\to [\mathbb{C}P^3, SG_3]\to \mathcal{E}_{\mathbb{C}P^3}(\mathbb{C}P^3\times\mathbb{S}^2) \to \mathcal{E}(\s^2)\to 0,
    \end{equation*}
     and 
    \begin{equation*} 
        0\to \pi_2(F_{\mathbb{S}^1})\cong \z/2 \to \mathcal{E}_{\mathbb{S}^2}(\mathbb{C}P^3\times\mathbb{S}^2) \to \mathcal{E}(\mathbb{C}P^3)\to 0.
    \end{equation*}
  
 Since the self-homotopy equivalences of $\mathbb{C}P^3\times\mathbb{S}^2$ arising from $\mathcal{E}(\mathbb{S}^2)$ and $\mathcal{E}(\mathbb{C}P^3)$ are homotopic to diffeomorphisms, their normal invariants are zero. Furthermore, by Fact \ref{normal invariant of SG}, the normal invariant of self-homotopy equivalences coming from $[\mathbb{C}P^3, SG_2]$ does not lie in $\pi_8(G/O)\subset [\mathbb{C}P^3\times\mathbb{S}^2, G/O].$ Therefore, considering the composition formula of the normal invariant, it is enough to determine the normal invariant of self-homotopy equivalences coming from $\pi_2(F_{\mathbb{S}^1}).$

     Now according to \cite{RS87}, the normal invariant of $f_{\eta}\in \mathcal{E}(\mathbb{C}P^3\times\mathbb{S}^2)$ induced by the element $\eta\in \pi_2(F_{\mathbb{S}^1})$ corresponds to the image of the composition $\Sigma^2 \mathbb{C}P_{+}^3\xrightarrow{\Sigma t} \mathbb{S}^1\xrightarrow{\eta}\mathbb{S}^0$ in $[\mathbb{C}P^3\times\mathbb{S}^2,G/O]$ under the canonical map $j:G\to G/O$ where $t$ is the Umkehr map \cite{JBRS74}. Since $\{\Sigma^2 \mathbb{C}P^3_{+}, \mathbb{S}^0\}= \pi_6^s\oplus\pi_8^s\oplus \pi-2^s,$ we work locally at the prime $2.$ Since $\eta\circ\nu=0$ \cite{toda}, it follows that $\eta\circ \Sigma t\vert_{\mathbb{S}^4}=0$. Therefore, calculating $\eta\circ\Sigma t\vert_{\Sigma^2 \mathbb{C}P^2}$ reduces to determining the toda bracket $\langle\eta,\nu,\eta\rangle,$ which equals to $\nu^2$ by \cite[Lemma 5.12]{toda}. Hence, $(\Sigma^2 i')^*(\eta^{Diff}(f_{\eta}))$ is non-trivial and is equal to $(\Sigma^2 f_{\mathbb{C}P^2})^*([\nu^2]),$ where $ i': \mathbb{C}P^2\hookrightarrow  \mathbb{C}P^3$ is the inclusion. Now since $\eta\circ\Sigma t\vert_{\mathbb{S}^4}=0,$ there exists a map $\mathcal{T}: \mathbb{S}^6\vee \mathbb{S}^8\to \mathbb{S}^0$ such that the following diagram commutes: 
     \begin{center}
         \begin{tikzcd}
             {\s^5\vee\s^7} \arrow[r]& {\s^4} \arrow[dr,"\nu"']\arrow[hookrightarrow,"\Sigma^2 \tilde{\tilde{i}}"]{r} & {\Sigma^2\mathbb{C}P^3} \arrow[r,"\Sigma^2 q"] \arrow[d,"\Sigma t"'] & \s^6\vee\s^8\arrow[r,"{(\eta, 2\nu)}"]\arrow[ddl,dashed,"\mathcal{T}"]& {\s^5}\\
             && {\s^1}\arrow[d,"\eta"']\\
             && {\s^0}
         \end{tikzcd}
     \end{center}
     where $\mathcal{T}\vert_{\mathbb{S}^6}:\mathbb{S}^6\to \mathbb{S}^0$ is given by $\nu^2,$ while $\mathcal{T}\vert_{\mathbb{S}^8}$ is determined by the Toda bracket $\langle 2\nu,\nu,\eta\rangle.$ We note that $\langle2\nu,\nu,\eta \rangle = \{\bar{\nu},\epsilon\}$ \cite{toda} and this maps nontrivially to $[\epsilon]\in \pi_8(G/O)$. Thus the required map $\eta\circ\Sigma t : \Sigma^2\mathbb{C}P^3\to \mathbb{S}^0 $ is the image of one of the elements $\nu^2, \bar{\nu}+\nu^2, \epsilon + \nu^2$ under the map $(\Sigma^2 q)^*: \pi_6^s\oplus\pi_8^s\to [\Sigma^2 \mathbb{C}P^3, G].$ Thus, $j_*(\eta\circ\Sigma t)\neq (f_{\mathbb{C}P^3\times\mathbb{S}^2})^*([\epsilon]).$ This completes the proof.
\end{proof}
\begin{thm}\label{classificationcp3s2}
  Let $N$ be a closed, oriented, smooth manifold homeomorphic to $\mathbb{C}P^3\times\mathbb{S}^2.$ Then, it is (oriented) diffeomorphic to exactly one of the manifolds $\mathbb{C}P^3\times\mathbb{S}^2$ or $(\mathbb{C}P^3\times\mathbb{S}^2)\# \Sigma,$ where $\Sigma$ is the $8$-dimensional exotic sphere.
\end{thm}

\begin{proof}
Let $[(N,g)]$ be an element of $\mathcal{C}(\mathbb{C}P^3\times\mathbb{S}^2).$ Then, by Lemma \ref{psicp3sk} {(ii)}, we have $\eta^{Diff}(g)$ is either $0$ or $(f_{\mathbb{C}P^3\times\mathbb{S}^2})^*([\epsilon]).$ 

If $\eta^{Diff}(g)$ is trivial, then from the smooth surgery exact sequence for $\mathbb{C}P^3\times\mathbb{S}^2,$ the pairs $(N,g)$ and $(\mathbb{C}P^3\times\mathbb{S}^2, id)$ represent the same element in $\mathcal{S}^{Diff}(\mathbb{C}P^3\times\mathbb{S}^2).$

If $\eta^{Diff}(g)=(f_{\mathbb{C}P^3\times\mathbb{S}^2})^*([\epsilon]),$ then from Lemma \ref{selfhomotopycp3s2} we see that for any $f\in \mathcal{E}(\mathbb{C}P^3\times\mathbb{S}^2),$ the normal invariant $\eta^{Diff}(f\circ g)\neq 0.$ Hence again by the surgery exact sequence for $\mathbb{C}P^3\times\mathbb{S}^2,$ the element $[(N,g)] =[(\mathbb{C}P^3\times\mathbb{S}^2\#\Sigma, h_{\Sigma})]$ in $\mathcal{S}^{Diff}(\mathbb{C}P^3\times\mathbb{S}^2),$ where $\Sigma$ is the exotic $8$-sphere and $h_{\Sigma}: (\mathbb{C}P^3\times\mathbb{S}^2)\#\Sigma\to \mathbb{C}P^3\times\mathbb{S}^2$ is the canonical homeomorphism.

Therefore any closed, oriented, smooth manifold homeomorphic to $\mathbb{C}P^3\times\mathbb{S}^2$ is oriented diffeomorphic to either $\mathbb{C}P^3\times\mathbb{S}^2$ or $(\mathbb{C}P^3\times\mathbb{S}^2)\#\Sigma.$
\end{proof}
As a direct consequence of Theorem \ref{classificationcp3s2}, we have
\begin{cor}\label{inertiagroupcp3s2}
    The inertia group of $\mathbb{C}P^3\times\mathbb{S}^2$ is zero.
\end{cor}
We now turn to the classification of smooth structures on manifolds homeomorphic to $\mathbb{C}P^3\times\mathbb{S}^3.$
\begin{lemma}\label{schultzresult}
\noindent
\begin{itemize}
    \item[(1)]  There exists a self-homotopy equivalence $f_{\epsilon\eta}: \mathbb{C}P^3\times\mathbb{S}^3\to \mathbb{C}P^3\times\mathbb{S}^3$ such that $\eta^{Diff}(f_{\epsilon\eta})= (f_{\mathbb{C}P^3\times\mathbb{S}^3})^* ([\epsilon\eta]).$
   \item[(2)] Let $f$ be a self-homotopy equivalence of $\mathbb{C}P^3\times\mathbb{S}^3$ that is induced by an element of $\pi_3(F_{\mathbb{S}^1}(\mathbb{C}^4))$ such that its normal invariant is non-trivial. Then $\eta^{Diff}(f)$ has odd filtration.
\end{itemize}
\end{lemma}
\begin{proof}
    The results in {(1)} and {(2)} are a consequence of \cite[Proposition 3.5]{RS87}.
\end{proof}

Utilizing the above lemma and applying the same techniques as in \cite[Proposition 10.1 and 10.5]{bel}, we derive the following results.
\begin{lemma}\label{codimensionresult} 
\noindent
\begin{itemize}
   \item[(a)]  Let $f \in \mathcal{E}(\mathbb{C}P^3\times\mathbb{S}^3)$ arising from an element of $\pi_3(F_{\mathbb{S}^1}(\mathbb{C}^4))\cong \pi_3(F_{\mathbb{S}^1}).$ Then $\eta^{Diff}(f)=0$ if and only if $f$ is homotopic to a diffeomorphism.

    \item[(b)]  Let $f \in \mathcal{E}(\mathbb{C}P^3\times\mathbb{S}^3)$ coming from an element of $[\mathbb{C}P^3, SG_4].$ If the normal invariant of $f$ is non-trivial, then it has even filtration.
\end{itemize}
\end{lemma}

\begin{lemma}\label{dichotomy}
    Any self-homotopy equivalence $f$ of $\mathbb{C}P^3\times\mathbb{S}^3$ is homotopic to a diffeomorphism if and only if its normal invariant is zero, i.e., $\eta^{Diff}(f)=0.$ 
\end{lemma}
\begin{proof}
This proof follows from \cite[Proposition 5.1 and 5.2]{bel}, together with Lemma \ref{schultzresult} and \ref{codimensionresult}. 
\end{proof}
\begin{lemma}\label{bp10htopyinertia}
     $bP_{10}\cap I_h(\mathbb{C}P^3\times\mathbb{S}^3)=\emptyset.$ 
\end{lemma}
\begin{proof}
From the smooth surgery exact sequence of $\mathbb{C}P^3\times\mathbb{S}^3,$ it is enough to show that the surgery obstruction map $\Theta^{Diff}: [\Sigma(\mathbb CP^3\times \mathbb S^3), G/O]\to L_{10}(e)$ is zero map. This map fits into the following commutative diagram
\begin{center}
    \begin{tikzcd}
        {[\Sigma(\mathbb{C}P^3\times\mathbb{S}^3), G/O]}\arrow[r,"\Theta^{Diff}"]\arrow[d,"p^*"']&{L_{10}(e)}\arrow[r,"\omega^{Diff}"]\arrow[d,"\cong"]&{\mathcal{S}^{Diff}(\mathbb{C}P^3\times\mathbb{S}^3)}\\
        {[\mathbb{C}P^3\times\mathbb{S}^3\times\mathbb{S}^1, G/O]}\arrow[r,"\bar{\Theta}^{Diff}"']&{L_{10}(\mathbb{Z})}
    \end{tikzcd}
\end{center}
 where the map $p^*:[\Sigma (\mathbb{C}P^3\times\mathbb{S}^3), G/O]\to [\mathbb{C}P^3\times\mathbb{S}^3\times\mathbb{S}^1, G/O]$ is injective by \eqref{note2.3}, and the isomorphism $L_{10}(e)\to L_{10}(\mathbb{Z})$ is established in \cite{JS69}. Now for any $f:\Sigma(\mathbb{C}P^3\times\mathbb{S}^3)\to G/O,$ the surgery obstruction is given by \cite{gw} as
 \begin{align*} 
       \Theta^{Diff}(f)=\Theta^{Diff}(p^*(f))& = \langle \upsilon_{4}^2(\mathbb{C}P^3\times\mathbb{S}^3 \times\mathbb{S}^1) (p^*(f))^*(\mathcal{K}_2), \;[\mathbb{C}P^3\times\mathbb{S}^3\times\mathbb{S}^1]\rangle=0,
    \end{align*}
where the $4$-th Wu class $\upsilon_4(\mathbb{C}P^3\times\mathbb{S}^3 \times\mathbb{S}^1)$ of the manifold $\mathbb{C}P^3\times\mathbb{S}^3 \times\mathbb{S}^1$ equals to $0$ \cite{MilSta} and $\mathcal{K}_2 \in H^2(G/O;\mathbb{Z}/2)$ is a suitable class. This completes the proof.
\end{proof}

\begin{lemma}\label{bp10}
    $bP_{10}\cap I(\mathbb{C}P^3\times\mathbb{S}^3)=\emptyset.$ 
\end{lemma}
\begin{proof}   
    Suppose $[\Sigma]$ is the non-trivial element of $bP_{10}$ contained within the inertia group $I(\mathbb{C}P^3\times\mathbb{S}^3),$ and let $h_{\Sigma}: (\mathbb CP^3\times \mathbb S^3)\# \Sigma \to \mathbb CP^3\times \mathbb S^3$ be the canonical homeomorphism corresponding to $\Sigma$. Then there exists an (oriented) diffeomorphism $\phi : (\mathbb CP^3\times \mathbb S^3)\# \Sigma \to \mathbb CP^3\times \mathbb S^3.$ Then $\phi \circ h_{\Sigma}^{-1} \in \mathcal{E}(\mathbb{C}P^3\times\mathbb{S}^3).$ Now using the composition formula for normal invariant, we get
    \begin{align*}
        \eta^{Diff}(\phi\circ h_{\Sigma}^{-1})=& \eta^{Diff}(\phi)+ (\phi^{-1})^* \eta^{Diff}(h_{\Sigma})=0.
    \end{align*}
    Therefore, it follows from Lemma \ref{dichotomy} that $h_{\Sigma}$ is homotopic to a diffeomorphism, a contradiction to the Lemma \ref{bp10htopyinertia}. 
\end{proof}

\begin{thm} 
    The homotopy inertia group $I_h(\mathbb{C}P^3\times\mathbb{S}^3)$ is zero.
\end{thm}
\begin{proof}
Since $bP_{10}\cap I_h(\mathbb{C}P^3\times \mathbb{S}^3)=\emptyset,$ we show that the restriction map $(f_{\mathbb{C}P^3\times\mathbb{S}^3})^*: \Theta_9\setminus {bP_{10}}\to \mathcal{S}^{Diff}(\mathbb{C}P^3\times\mathbb{S}^3)$ is injective. The map $(f_{\mathbb CP^3\times\mathbb S^3})^*: \Theta_9\to \mathcal{S}^{Diff}(\mathbb CP^3\times\mathbb S^3)$ fits into the following commutative diagram whose rows correspond to the smooth surgery exact sequences for $\mathbb S^9$ and $\mathbb CP^3\times\mathbb S^3,$ respectively.
\begin{center}
    \begin{tikzcd}
        {L_{10}(e)}\arrow[r]\arrow[d,-, double equal sign distance, double] & {\Theta_9}\arrow[r,"\eta^{Diff}"]\arrow[d,"(f_{\mathbb CP^3\times\mathbb S^3})^*"'] & {\pi_9(G/O)}\arrow[r]\arrow[d,"(f_{\mathbb CP^3\times \mathbb S^3})^*"]& {0}\\
        {L_{10}(e)}\arrow[r]&{\mathcal{S}^{Diff}(\mathbb CP^3\times\mathbb S^3)}\arrow[r,"\eta^{Diff}"']& {[\mathbb CP^3\times\mathbb S^3, G/O]}\arrow[r]&{0}
    \end{tikzcd}
\end{center}
  We know that the restriction map $\eta^{Diff}:\Theta_9\setminus {bP_{10}}\to \pi_9(G/O)$ is injective. Hence the injectivity of $(f_{\mathbb CP^3\times\mathbb S^3})^*:\Theta_9\setminus {bP_{10}}\to \mathcal{S}^{Diff}(\mathbb CP^3\times\mathbb S^3)$ follows from the above commutative diagram, as  $(f_{\mathbb CP^3\times\mathbb S^3})^*: \pi_9(G/O)\to [\mathbb CP^3\times\mathbb S^3, G/O]$ is injective, observed in Lemma \ref{psicp3sk} {(iii)}.
\end{proof}

\begin{theorem}\label{inertiagroupCP3S3}
    The inertia group of $\mathbb{C}P^3\times\mathbb{S}^3$ is zero.
\end{theorem}
\begin{proof}
Since $\pi_3(\mathbb{C}P^3) = 0$, according to \cite[Proposition 2.1]{PP99}, any self-homotopy equivalence of $\mathbb{C}P^3 \times \mathbb{S}^3$ can be diagonalized. Therefore, to study the normal invariant of $\mathcal{E}(\mathbb{C}P^3 \times \mathbb{S}^3)$, we focus on the normal invariant of the self-homotopy equivalences of $\mathbb{C}P^3 \times \mathbb{S}^3$, induced from $\mathcal{E}(\mathbb{S}^3)$, $\mathcal{E}(\mathbb{C}P^3)$, $[\mathbb{C}P^3, E_1(\mathbb{S}^3)]$, and $\pi_3(E_1(\mathbb{C}P^3))$, by Fact \ref{self homotopy equivalence of product} (b). 
Since any self-homotopy equivalence $g$ of $\mathbb CP^3\times \mathbb S^3$ induced from $\mathcal{E}(\mathbb S^3)$ or $\mathcal{E}(\mathbb CP^3)$ are represented by diffeomorphism, $\eta^{Diff}(g)=0.$

Let $[(\mathbb{C}P^3\times\mathbb{S}^3\#\Sigma_{\mu}, h_{\Sigma_{\mu}})]\in \mathcal{C}(\mathbb{C}P^3\times\mathbb{S}^3),$ where $[\Sigma_{\mu}]\in \Theta_9$ corresponding to the generator $\mu\in \pi_9^s$ and $h_{\Sigma_{\mu}}: (\mathbb CP^3\times\mathbb S^3)\# \Sigma_{\mu}\to \mathbb CP^3\times \mathbb S^3$ is the canonical homeomorphism. Then from Proposition \ref{CP3computation} and Lemma \ref{psicp3sk} {(iii)}, we have $\eta^{Diff}(h_{\Sigma_{\mu}})= (f_{\mathbb CP^3\times \mathbb S^3})^*([\mu]).$ Now considering the above decomposition of $\mathcal{E}(\mathbb CP^3\times \mathbb S^3)$ and combining Fact \ref{normal invariant of SG} and Lemma \ref{schultzresult} {(1)}, we have $\eta^{Diff}(g\circ h_{\Sigma_{\mu}})\neq 0$ for all $g \in \mathcal{E}(\mathbb CP^3\times \mathbb S^3).$ Therefore $\Sigma_{\mu}\notin I(\mathbb CP^3\times \mathbb S^3).$

Let $[(\mathbb{C}P^3\times\mathbb{S}^3\#\Sigma_{\epsilon\eta}, h_{\Sigma_{\epsilon\eta}})]$ be an element of $\mathcal{C}(\mathbb{C}P^3\times\mathbb{S}^3),$ where $[\Sigma_{\epsilon\eta}]$ is the exotic $9$-sphere associated to the generator $\epsilon\eta$ of $\pi_9^s$ and $h_{\Sigma_{\epsilon\eta}}: \mathbb CP^3\times \mathbb S^3\#\Sigma_{\epsilon\eta}\to \mathbb CP^3\times \mathbb S^3$ is the canonical homeomorphism. Then $\eta^{Diff}(h_{\Sigma_{\epsilon\eta}})= (f_{\mathbb CP^3\times \mathbb S^3})^* ([\epsilon\eta]).$ Now using Lemma \ref{schultzresult} {(1)} in the composition formula for normal invariants, we get $\eta^{Diff}(f_{\epsilon\eta}\circ h_{\Sigma_{\epsilon\eta}})=0.$  This implies that  $(\mathbb CP^3\times \mathbb S^3\#\Sigma_{\epsilon\eta}, f_{\epsilon\eta}\circ h_{\Sigma_{\epsilon\eta}})$ is equivalent to $(\mathbb CP^3\times\mathbb S^3\# \Sigma, h_{\Sigma})$ in $\mathcal{S}^{Diff}(\mathbb CP^3\times \mathbb S^3)$ for some $[\Sigma] \in bP_{10}.$  Therefore $[\Sigma_{\epsilon\eta}\# \Sigma ^{-1}] \in I(\mathbb CP^3\times \mathbb S^3).$ Now if $[\Sigma_{\epsilon\eta}]\in I(\mathbb{C}P^3\times\mathbb{S}^3),$ then $[\Sigma]^{-1}\in I(\mathbb{C}P^3\times\mathbb{S}^3),$ which contradicts Lemma \ref{bp10}.

Therefore, based on the above calculation and Lemma \ref{bp10}, $I(\mathbb{C}P^3\times\mathbb{S}^3)$ is zero.
\end{proof}

\begin{rmk}
  The exotic sphere corresponding to the Adams element $\mu \in \pi_9^s$ does not belong to the inertia group of $\mathbb{C}P^3\times\mathbb{S}^3$ follows directly from \cite[Lemma 9.1]{Kawakubo69}.  
\end{rmk}

Using Theorem \ref{inertiagroupCP3S3}, we now classify all smooth manifolds homeomorphic to $\mathbb{C}P^3\times\mathbb{S}^3,$ up to diffeomorphism.

\begin{thm}\label{ClassificationCP3S3}
     Let $N$ be a closed oriented smooth manifold homeomorphic to $\mathbb{C}P^3\times\mathbb{S}^3,$ then it is (oriented) diffeomorphic to exactly one of the manifolds $\mathbb{C}P^3\times\mathbb{S}^3, (\mathbb{C}P^3\times\mathbb{S}^3)\# \Sigma$ for some $[\Sigma]\in \Theta_9.$   
\end{thm}
\begin{proof}
Let $[(N,g)]$ be an element of $ \mathcal{C}(\mathbb{C}P^3\times\mathbb{S}^3).$ 

 Suppose $\eta^{Diff}(g)$ is trivial. Then, from the smooth surgery exact sequence for $\mathbb{C}P^3\times\mathbb{S}^3,$ it follows that $(N,g)$ is equivalent to $(\mathbb{C}P^3\times\mathbb{S}^3\#\Sigma,h_{\Sigma})$ in $\mathcal{S}^{Diff}(\mathbb{C}P^3\times\mathbb{S}^3)$ for some $[\Sigma] \in bP_{10}.$ 

 Suppose $\eta^{Diff}(g)$ is non-trivial. Then from Lemma \ref{psicp3sk} {(iii)}, we note that the pairs $(N,g)$ and $(\mathbb{C}P^3\times\mathbb{S}^3\#\Sigma,h_{\Sigma})$ are equivalent in $\mathcal{S}^{Diff}(\mathbb{C}P^3\times\mathbb{S}^3),$ where $[\Sigma] \in \Theta_9\setminus bP_{10}.$ Now by Theorem \ref{inertiagroupCP3S3}, $N$ is not diffeomorphic to $\mathbb{C}P^3\times\mathbb{S}^3.$ 

 Let $[(\mathbb{C}P^3\times\mathbb{S}^3\#\Sigma_1, h_{\Sigma_1})]$ and $[(\mathbb{C}P^3\times\mathbb{S}^3\#\Sigma_2, h_{\Sigma_2})]$ be two elements of $\mathcal{C}(\mathbb{C}P^3\times\mathbb{S}^3),$ corresponding to non-diffeomorphic elements $[\Sigma_1], [\Sigma_2] \in \Theta_9\subset \mathcal{C}(\mathbb{C}P^3\times\mathbb{S}^3).$ If $(\mathbb{C}P^3\times\mathbb{S}^3\#\Sigma_1, h_{\Sigma_1})$ and $(\mathbb{C}P^3\times\mathbb{S}^3\#\Sigma_2, h_{\Sigma_2})$ represent same element in $\mathcal{S}^{Diff}(\mathbb{C}P^3\times\mathbb{S}^3),$ then $[\Sigma_1\#\Sigma_2]^{-1} \in I(\mathbb{C}P^3\times\mathbb{S}^3).$ Therefore $\Sigma_1$ is diffeomorphic to $\Sigma_2$ by Theorem \ref{inertiagroupCP3S3}, a contradiction. Hence any two elements of $\mathcal{C}(\mathbb{C}P^3\times\mathbb{S}^3)$ arising from $\Theta_9$ are pairwise non-diffeomorphic.

This completes the proof.
\end{proof} 


Before examining the action of self-homotopy equivalences of $\mathbb{C}P^3\times\mathbb{S}^4$ on $\mathcal{C}(\mathbb{C}P^3\times\mathbb{S}^4),$ we first establish the existence of a tangential self-homotopy equivalence of $\mathbb{C}P^3\times\mathbb{S}^4$ with nontrivial normal invariant in $\pi_{10}(G/O)_{(3)}\subset [\mathbb{C}P^3\times\mathbb{S}^4, G/O].$  
\begin{lemma}
    There exists a tangential homotopy equivalence $f:\mathbb{C}P^3\times\mathbb{S}^4\to \mathbb{C}P^3\times\mathbb{S}^4$ such that its normal invariant is given by $\eta^{Diff}(f)= (f_{\mathbb{C}P^3\times\mathbb{S}^4})^*([\beta_1])\neq 0,$ where $\beta_1$ is the generator of $(\pi_{10}^s)_{(3)}$ and $f_{\mathbb{C}P^3\times\mathbb{S}^4}:\mathbb{C}P^3\times\mathbb{S}^4\to \mathbb{S}^{10}$ is the degree one collapse map. 
\end{lemma}

\begin{proof}
    
   By \cite[Fact 5.2 {(d)} and {(e)} ]{SBRKAS}, we have $\pi_4(F_{\mathbb{S}^1}(\mathbb{C}^4))\cong \pi_k^s(\Sigma \mathbb{C}P^{\infty})\oplus \pi_3^s.$ From \cite[Fact 5.2 {(b)}]{SBRKAS}, any self-homotopy equivalence of $\mathbb{C}P^3\times\mathbb{S}^4$ induced from an element of $\pi_4(F_{\mathbb{S}^1}(\mathbb{C}^4))$ is tangential. 
   
   Let $f:\mathbb{C}P^3\times\mathbb{S}^4\to \mathbb{C}P^3\times\mathbb{S}^4$ be a tangential self-homotopy equivalence arising from the generator $\alpha_1$ of $(\pi_3^s)_{(3)}.$ Referring to Fact \cite[Fact 5.2 {(f)}]{SBRKAS}, its normal invariant is the image of the composition  $$\Sigma^4 \mathbb{C}P^3_{+}\xrightarrow{\Sigma^3 t}\mathbb{S}^3\xrightarrow{\alpha_1}\mathbb{S}^0$$ in $[\Sigma^4\mathbb{C}P^3, G/O]\subset [\mathbb{C}P^3\times\mathbb{S}^4, G/O]$ under $j: G\to G/O.$
   
   We note that the restriction of $\alpha_1\circ \Sigma^3 t$ to both $\Sigma^4 \mathbb{C}P^1$ and $\Sigma^4\mathbb{C}P^2$ is trivial. Therefore, it suffices to compute the Toda bracket $\langle\alpha_1, \alpha_1, \alpha_1 \rangle$ to determine whether $\Sigma^3 t\circ\alpha_1 : \Sigma^4 \mathbb{C}P^3\to \mathbb{S}^0$ is trivial. By \cite{SO72}, $\langle\alpha_1, \alpha_1, \alpha_1 \rangle =\beta_1,$ which projects nontrivially to $[\beta_1] \in \pi_{10}(G/O).$ Thus, $\eta^{Diff}(f) = (f_{\mathbb{C}P^3\times\mathbb{S}^4})^*([\beta_1]),$ since the map $(f_{\mathbb{C}P^3\times\mathbb{S}^4})^*: \pi_{10}(G/O)\to [\Sigma^4\mathbb{C}P^3, G/O]$ is injective.
\end{proof}
\begin{lemma}\label{betamap}
    There is a self-homeomorphism of $\mathbb{C}P^3\times\mathbb{S}^4$ with normal invariant given by $(f_{\mathbb{C}P^3\times\mathbb{S}^4})^*([\beta_1]),$ where $\beta_1$ represents the $3$-component of $\pi_{10}^s.$  
\end{lemma}
\begin{proof}
    By the previous lemma, we have a self-homotopy equivalence $f:\mathbb{C}P^3\times\mathbb{S}^4\to \mathbb{C}P^3\times\mathbb{S}^4$ whose normal invariant is given by $(f_{\mathbb{C}P^3\times\mathbb{S}^4})^* ([\beta_1]).$ By Lemma \ref{psicp3sk} {(iv)} and Proposition \ref{CP3computation}, there exists an exotic $10$-sphere $[\Sigma_{\beta_1}]$ in $\Theta_{10}$ such that the normal invariant of the canonical homeomorphism $h_{\Sigma_{\beta_1}}:(\mathbb{C}P^3\times\mathbb{S}^4) \# \Sigma_{\beta_1} \to \mathbb{C}P^3\times\mathbb{S}^4$ is the same as that of $f.$ Hence, by the smooth surgery exact sequence of $\mathbb{C}P^3\times\mathbb{S}^4,$ there exists a diffeomorphism $\mathcal{H}: (\mathbb{C}P^3\times\mathbb{S}^4)\# \Sigma_{\beta_1}\to \mathbb{C}P^3\times\mathbb{S}^4$ such that $f\circ\mathcal{H} \simeq h_{\Sigma_{\beta_1}}.$ This implies that $f\simeq h_{\Sigma_{\beta_1}}\circ \mathcal{H}^{-1},$ showing that $f$ is a self-homeomorphism of $\mathbb{C}P^3\times\mathbb{S}^4.$ This completes the proof.
\end{proof}
For notational convenience, we denote the self-homeomorphism mentioned in the above lemma as $f_{\alpha_1}.$ 

\begin{lemma}\label{numap}
    There exists a self-homeomorphism $f_{2\nu}: \mathbb{C}P^3\times\mathbb{S}^4\to \mathbb{C}P^3\times\mathbb{S}^4$ such that $\eta^{Diff}(f_{2\nu})$ is nontrivial.

    Further, $\eta^{Diff}(f_{2\nu})\in \mathit{Im} \left([\Sigma^4 \mathbb{C}P^2,Top/O]\xrightarrow{\psi_*} [\Sigma^4 \mathbb{C}P^2,G/O]\right) \subset [\mathbb{C}P^3\times\mathbb{S}^4,G/O].$
\end{lemma}
\begin{proof}
We aim to construct tangential self-homotopy equivalence $f_{2\nu}: \mathbb{C}P^3\times\mathbb{S}^4\to \mathbb{C}P^3\times\mathbb{S}^4,$ whose normal invarinat is non-trivial and belongs to the image of $\psi_*:[\Sigma^4\mathbb{C}P^2, Top/O]\to [\Sigma^4\mathbb{C}P^2,G/O].$ Once this is established, we follow argument similar to Lemma \ref{betamap} to show that $f_{2\nu}: \mathbb{C}P^3\times\mathbb{S}^4\to \mathbb{C}P^3\times\mathbb{S}^4$ is a self-homeomorphism.

 We consider the tangential self-homotopy equivalence $f_{2\nu}: \mathbb{C}P^3\times\mathbb{S}^4\to \mathbb{C}P^3\times\mathbb{S}^4,$ induced from the generator $2\nu \in (\pi_3^s)_{(2)} \subset \pi_4(F_{\mathbb{S}^1}).$ By Fact \cite[Fact 5.2 {(f)}]{SBRKAS}, the normal invariant of $f_{2\nu}$ is the image of the composition $\Sigma^4\mathbb{C}P^3_{+}\xrightarrow{\Sigma^3 t}\mathbb{S}^3\xrightarrow{2\nu}\mathbb{S}^0$ under the canonical homomorphisms $[\Sigma^4 \mathbb{C}P^3, G]\xrightarrow{j_*} [\Sigma^4\mathbb{C}P^3, G/O]\xrightarrow{p^*} [\mathbb{C}P^3\times\mathbb{S}^4, G/O].$ We observe that the restriction of $2\nu\circ \Sigma^3 t$ on $\Sigma^4\mathbb{C}P^1$ is $2\nu^2=0,$  while its restriction on $\Sigma^4\mathbb{C}P^2$ is the toda bracket $\langle\eta, \nu, 2\nu\rangle.$ But $\{\epsilon, \bar{\nu}\} \in \langle \eta, \nu, 2\nu\rangle$ \cite{toda,ravenel, DCCN20} and is mapped nontrivially to $[\epsilon]\in \pi_8(G/O).$ Therfore $\eta^{Diff}(f_{2\nu})$ is nontrivial and there exists $\mathcal{T}': \mathbb{S}^{10}\vee\mathbb{S}^8\to \mathbb{S}^0$ such that the following diagram commutes
    \begin{center}
        \begin{tikzcd}
             {\s^7\vee\s^9} \arrow[r,"{(\eta, 2\nu)}"]& {\s^6} \arrow[dr,"\nu"']\arrow[hookrightarrow,"\Sigma^2 \tilde{\tilde{i}}"]{r} & {\Sigma^4\mathbb{C}P^3} \arrow[r,"\Sigma^2 q"] \arrow[d,"\Sigma t"'] & \s^8\vee\s^{10}\arrow[r,"{(\eta, 2\nu)}"]\arrow[ddl,dashed,"\mathcal{T}'"]& {\s^7}\\
             && {\s^3}\arrow[d,"2\nu"']\\
             && {\s^0}
        \end{tikzcd}
    \end{center}
where $\mathcal{T}'\mid_{\mathbb{S}^{10}}:\mathbb{S}^{10}\to \mathbb{S}^0$ is determined by the Toda bracket $\langle2\nu,\nu,2\nu\rangle.$  Using the properties of Toda bracket and the fact $\eta^2\circ\mu\neq 0,$ we get $\langle2\nu,\nu,2\nu\rangle=0.$ Hence $\eta^{Diff}(f_{2\nu})=(\Sigma^4 q)^*\circ (\Sigma^4 Pr)^*([\epsilon]).$       
    \end{proof}

\begin{thm}\label{inertiacp3s4}
    The inertia group of $\mathbb{C}P^3\times\mathbb{S}^4$ is $\mathbb{Z}/3.$
\end{thm}

\begin{proof}
    Let $[\Sigma_{\beta_1}] \in \Theta_{10}$ be an exotic sphere of order $3$ and $h_{\beta_1}: (\mathbb{C}P^3\times\mathbb{S}^4)\# \Sigma_{\beta_1}\to \mathbb{C}P^3\times\mathbb{S}^4$ be the standard homeomorphism. Since $\pi_4(\mathbb{C}P^3)$ is trivial, it follows from \cite[Proposition 2.1]{PP99} and Fact \ref{self homotopy equivalence of product} that it is enough to examine the action of the self-homotopy equivalence of $\mathbb{C}P^3\times\mathbb{S}^4$ arising from $\mathcal{E}(\mathbb{C}P^3), \mathcal{E}(\mathbb{S}^4), [\mathbb{C}P^3, E_1(\mathbb{S}^4)],$ and $\pi_4(E_1(\mathbb{C}P^3)).$ Moreover, since any self-homotopy equivalence of $\mathbb{C}P^3\times\mathbb{S}^4$ coming from $\mathcal{E}(\mathbb{C}P^3), \mathcal{E}(\mathbb{S}^4)$ represents a diffeomorphism, the normal invariant of an element of $\mathcal{E}(\mathbb{C}P^3\times\mathbb{S}^4)$ coming from $[\mathbb{C}P^3, E_1(\mathbb{S}^4)]$ belongs to $[\mathbb{C}P^3, G/O]$ by Fact \ref{normal invariant of SG} and $\eta^{Diff}(h_{\beta_1})\in [\Sigma^4\mathbb{C}P^3, G/O],$ we need to study the action of self-homeomorphism $f_{\alpha_1},$ appearing in Lemma \ref{betamap} on $h_{\beta_1}.$ Now by using the composition formula for normal invariant, we get

\begin{align*}
&\eta^{Diff}\left([(\mathbb{C}P^3\times\mathbb{S}^4)\#\Sigma_{\beta_1},f_{\alpha_1}\circ h_{\beta_1}]\right)\\   
&=\eta^{Diff}\left([\mathbb{C}P^3\times\mathbb{S}^4,f_{\alpha_1}]\right)+ 
(f_{\alpha_1}^{-1})^*\eta^{Diff}\left([(\mathbb{C}P^3\times\mathbb{S}^4)\#\Sigma_{\beta_1},h_{\beta_1}]\right)\\
&=\eta^{Diff}\left([\mathbb{C}P^3\times\mathbb{S}^4,f_{\alpha_1}]\right)+ (f_{\alpha_1}^{-1})^*\circ (f_{\mathbb{C}P^3\times\mathbb{S}^4})^* \left([\Sigma_{\beta_1}]\right)\\
&= (f_{\mathbb{C}P^3\times\mathbb{S}^4})^*[\beta_1]\pm (f_{\mathbb{C}P^3\times\mathbb{S}^4})^*[\beta_1] ,
\end{align*}

If $\eta^{Diff}\left([(\mathbb{C}P^3\times\mathbb{S}^4)\#\Sigma_{\beta_1},f_{\alpha_1}\circ h_{\beta_1}]\right)=0,$ then by the smooth surgery exact sequence $(\mathbb{C}P^3\times\mathbb{S}^4\#\Sigma_{\beta_1}, f_{\alpha_1}\circ h_{\Sigma_{\beta_1}})$ is equivalent to $(\mathbb{C}P^3\times \mathbb{S}^4, id)$ in $\mathcal{S}^{Diff}(\mathbb{C}P^3\times\mathbb{S}^4).$ 

If $\eta^{Diff}\left([(\mathbb{C}P^3\times\mathbb{S}^4)\#\Sigma_{\beta_1},f_{\alpha_1}\circ h_{\beta_1}]\right)\neq0,$ then again by surgery exact sequence for $\mathbb{C}P^3\times\mathbb{S}^4,$ the element $(\mathbb{C}P^3\times\mathbb{S}^4\#\Sigma_{\beta_1}, f_{\alpha_1}\circ h_{\Sigma_{\beta_1}})$ is equivalent to $(\mathbb{C}P^3\times\mathbb{S}^4\#\Sigma_{2\beta_1}, h_{\Sigma_{2\beta_1}})$ in $\mathcal{S}^{Diff}(\mathbb{C}P^3\times\mathbb{S}^4).$ This implies that $[\Sigma_{2\beta_1}\# \Sigma_{\beta_1}^{-1}]\in I(\mathbb{C}P^3\times\mathbb{S}^4).$ But $[\Sigma_{2\beta_1}\# \Sigma_{\beta_1}^{-1}]=[\Sigma_{\beta_1}].$ 

Therefore, $(\Theta_{10})_{(3)} \subseteq I(\mathbb{C}P^3\times\mathbb{S}^4).$ 

    From \cite[Lemma 9.1]{Kawakubo69}, we observe that $I(\mathbb{C}P^3\times\mathbb{S}^4)$ does not contain any homotopy sphere that does not bound a spin manifold. Since $\Theta_{10}\cong \mathit{bspin}_{10}\oplus\mathbb{Z}/2$ \cite{GB69}, $2$-component of $\Theta_{10}$ is not contained in $I(\mathbb{C}P^3\times\mathbb{S}^4).$

   Hence $I(\mathbb{C}P^3\times\mathbb{S}^4)\cong \mathbb{Z}/3.$
\end{proof}

\begin{thm}\label{classificationCP3S4}
    Let $N$ be a closed, smooth, oriented manifold homeomorphic to $\mathbb{C}P^3\times\mathbb{S}^4.$ Then it is (oriented) diffeomorphic to either $\mathbb{C}P^3\times\mathbb{S}^4$ or $(\mathbb{C}P^3\times\mathbb{S}^4)\# \Sigma_{\eta\mu},$ where $[\Sigma_{\eta\mu}]\in \Theta_{10}$ is the exotic sphere of order $2.$
\end{thm}
\begin{proof}
 Let $[(N,g)]$ be an element of $\mathcal{C}(\mathbb{C}P^3\times\mathbb{S}^4).$

   \underline{\textbf{Case-1}} Let $\eta^{Diff}(g)=0.$ Then by the smooth surgery exact sequence for $\mathbb{C}P^3\times\mathbb{S}^4,$ we have that $(N,g)$ is equivalent to $(\mathbb{C}P^3\times\mathbb{S}^4, id)$ in $\mathcal{S}^{Diff}(\mathbb{C}P^3\times\mathbb{S}^4).$

   \underline{\textbf{Case-2}} Let $\eta^{Diff}(g)$ be nontrivial.
   \begin{itemize}
       \item[(i)] Let $[(N,g)]$ be an element of $\mathcal{C}(\mathbb{C}P^3\times\mathbb{S}^4)$ coming from $[\Sigma^4\mathbb{C}P^2,Top/O]\subset [\Sigma^4 \mathbb{C}P^3,Top/O].$ Then, from Lemma \ref{psicp3sk} {(iv)}, the normal invariant of $[(N,g)]$ is $(\Sigma^4 q)^*\circ (\Sigma^4 Pr)^*[\epsilon].$ Now we study the action of the self-homeomorphism $f_{2\nu}$ from Lemma \ref{numap} on $[(N,g)]$ by computing the normal invariant as follows
    \begin{align*}
        \eta^{Diff}(f_{2\nu}\circ g)= & \eta(f_{2\nu})+ (f_{2\nu}^{-1})^* \eta^{Diff}(g)\\
         = & (\Sigma^4 q)^*\circ (\Sigma^4 Pr)^*[\epsilon]\pm (\Sigma^4 q)^*\circ (\Sigma^4 Pr)^*[\epsilon]=0,
    \end{align*} Therefore, by Case 1, $N$ is (oriented) diffeomorphic to $\mathbb{C}P^3\times\mathbb{S}^4.$
    \item[(ii)] Let $[(N,g)]\in \mathcal{C}(\mathbb{C}P^3\times\mathbb{S}^4)$ such that its normal invariant is $(f_{\mathbb{C}P^3\times\mathbb{S}^4})^*[\eta\mu] + (\Sigma^4 q)^*\circ(\Sigma^4 Pr)^*[\epsilon].$ Then 
    \begin{align*}
        \eta^{Diff}(f_{2\nu}\circ g)= & \eta^{Diff}(f_{2\nu})+ (f_{2\nu}^{-1})^* \eta^{Diff}(g)\\
        = & (\Sigma^4 q)^*\circ(\Sigma^4 Pr)^*[\epsilon] \pm \left ((f_{\mathbb{C}P^3\times\mathbb{S}^4})^*[\eta\mu] + (\Sigma^4 q)^*\circ(\Sigma^4 Pr)^*[\epsilon]\right)\\
        =& \pm (f_{\mathbb{C}P^3\times\mathbb{S}^4})^*[\eta\mu]
    \end{align*}
    Then by using Theorem \ref{inertiacp3s4}, we conclude that $N$ is (oriented) diffeomorphic to $(\mathbb{C}P^3\times\mathbb{S}^4)\# \Sigma_{\eta\mu}.$ 
    \item[(iii)] Let $[(N,g)]\in \mathcal{C}(\mathbb{C}P^3\times\mathbb{S}^4)$ such that $\eta^{Diff}(g)= (f_{\mathbb{C}P^3\times\mathbb{S}^4})^*[\beta_1] + (\Sigma^4 q)^*\circ(\Sigma^4 Pr)^*[\epsilon].$ Now 

    \begin{align*}
        \eta^{Diff}(f_{\alpha_1}\circ g)=& \eta^{Diff}(f_{\alpha_1})+(f_{\alpha_1}^{-1})^*\eta^{Diff}(g)\\
        = & (f_{\mathbb{C}P^3\times\mathbb{S}^4})^*[\beta_1]- \left((f_{\mathbb{C}P^3\times\mathbb{S}^4})^*[\beta_1]+ (\Sigma^4 q)^*\circ(\Sigma^4 Pr)^*[\epsilon]\right)= (\Sigma^4 q)^*\circ(\Sigma^4 Pr)^*[\epsilon].
    \end{align*}
    Hence, $N$ is oriented diffeomorphic to $\mathbb{C}P^3\times\mathbb{S}^4$ by Case 2.
\end{itemize}
 Combining the above cases with Theorem \ref{inertiacp3s4} yields the result.
\end{proof}
First, we examine the homotopy inertia group of $\mathbb{C}P^3\times\mathbb{S}^5$ and use it to determine the inertia group of $\mathbb{C}P^3\times\mathbb{S}^5.$
\begin{thm} \label{homotopyinertiacp3s5}
    The homotopy inertia group $I_h(\mathbb{C}P^3\times\mathbb{S}^5)$ is $\mathbb{Z}/{62}.$
\end{thm}
\begin{proof}
Since the homotopy inertia group of $\mathbb{C}P^3\times\mathbb{S}^5$ is equal to $L_{12}(e)/{Im(\Theta^{diff})},$ it suffices to determine the image of the surgery obstruction map $\Theta^{Diff}: [\Sigma(\mathbb{C}P^3\times\mathbb{S}^5), G/O]\to L_{12}(e).$ Now consider the following commutative diagram
\begin{center}
    \begin{tikzcd}
        {[\Sigma(\mathbb{C}P^3\times\mathbb{S}^5), G/O]}\arrow[r,"\Theta^{Diff}"]\arrow[d,"p^*"']&{L_{12}(e)}\arrow[r,"\omega^{Diff}"]\arrow[d,"\cong"]&{\mathcal{S}^{Diff}(\mathbb{C}P^3\times\mathbb{S}^5)}\\
        {[\mathbb{C}P^3\times\mathbb{S}^5\times\mathbb{S}^1, G/O]}\arrow[r,"\Theta^{Diff}"']&{L_{12}(\mathbb{Z})}
    \end{tikzcd}
\end{center}
where $p^*$ is injective by \eqref{note2.3}, the rows are surgery obstruction maps for $\mathbb{C}P^3\times\mathbb{S}^5,~ \mathbb{C}P^3\times\mathbb{S}^5\times\mathbb{S}^1$ respectively and the isomorphism $L_{12}(e)\to L_{12}(\mathbb{Z})$ follows from \cite{JS69}. From the above commutative diagram, it follows that $\Theta^{Diff}(f)= \Theta^{Diff}(p^*(f))$ for any $f\in [\Sigma(\mathbb{C}P^3\times\mathbb{S}^5), G/O].$ According to \cite{gw}, $\Theta^{Diff}(f)$ is given by 
\begin{equation*} 
    \Theta^{Diff}(f)= \frac{1}{8}\langle L(\mathbb{C}P^3\times\mathbb{S}^5\times\mathbb{S}^1)(1-L(\xi)), [\mathbb{C}P^3\times\mathbb{S}^5\times\mathbb{S}^1]\rangle,
\end{equation*} where $\xi$ denotes the image of $f$ under the canonical map $i: G/O\to BSO,$ and $L(\xi)$ is the $L$-genus of $\xi.$

    We note that the image of $i_*:[\Sigma(\mathbb{C}P^3\times\mathbb{S}^5), G/O]\to[\Sigma(\mathbb{C}P^3\times\mathbb{S}^5), BSO]$ is $\mathbb{Z}\oplus \mathbb{Z}.$ Let $\xi_1$ and $\xi_2$ be generators of this image. Now using \cite[Theorem 2 {(vi)}]{Fuji}, \cite[Lemma 2.1 (2)]{Tanaka} and \cite[Corollary 15.5]{MilSta}, we find that the total Pontrjagin classes of $\xi_1$ and $\xi_2$ are given by  $p(\xi_1)=1+p_1(\xi_1)+p_2(\xi_1)+p_3(\xi_1)=1+ 0+ 240. 12. yz+ 504. 40. yz^3$ and $p(\xi_2)=1+p_1(\xi_2)+p_2(\xi_2)+p_3(\xi_2)= 1+ 0+ 0+ 504. 240.yz^3,$ where $y\in H^6(\mathbb{S}^6;\mathbb{Z})$ and $z\in H^2(\mathbb{C}P^3;\mathbb{Z})$ are generators. Thus, if $\xi=m\xi_1+n\xi_2,$ then
      \begin{align*}
        \Theta^{Diff}(f)= & -\frac{1}{8}\langle(1+\frac{4}{3}z^2)\left(\frac{7.240.12. m}{45} yz + \frac{62.504.40 (m-6n)}{945} yz^3\right),[\mathbb{C}P^3\times\mathbb{S}^5\times\mathbb{S}^1]\rangle\\
        = & -\frac{1}{8}\left[1920 m-7936n\right]\\
        =& -240 m+992n = 2^4 (-15m+ 62n).
    \end{align*}
    Hence $I_h(\mathbb{C}P^3\times\mathbb{S}^5)=\mathbb{Z}/{62}\subset \Theta_{11}.$
\end{proof}

\begin{lemma}\label{dichotomycp3s5}
    Let $f\in \mathcal{E}(\mathbb{C}P^3\times\mathbb{S}^5).$ Then $f$ is homotopic to a diffeomorphism if and only if $\eta^{Diff}(f)=0.$ 
\end{lemma}
\begin{proof}
    Note that any $f\in \mathcal{E}(\mathbb{C}P^3\times\mathbb{S}^5)$ is diagonalizable by \cite[Proposition 2.1]{PP99}. Since any self-homotopy equivalence of $\mathbb{C}P^3\times\mathbb{S}^5$ coming from $\mathcal{E}(\mathbb{C}P^3)$ or $\mathcal{E}(\mathbb{S}^5)$ represent diffeomorphism, it follows from \cite[Theorem 2.5 and Proposition 2.3 {(d)}]{PP99} that it suffices to consider the self-homotopy equivalence induced from $\pi_{5}(E_1(\mathbb{C}P^3))$ and $[\mathbb{C}P^3, E_1(\mathbb{S}^5)]\cong [\mathbb{C}P^3, SG_6].$ As $\pi_5(E_1(\mathbb{C}P^3))\cong \pi_5(U(3)),$ any element of $\mathcal{E}(\mathbb{C}P^3\times\mathbb{S}^5)$ that comes from $\pi_5(E_1(\mathbb{C}P^3))$ is homotopic to diffeomorphism. Moreover, by \cite[Proposition 10.2]{bel} self-homotopy equivalence of $\mathbb{C}P^3\times\mathbb{S}^5$ arising from $[\mathbb{C}P^3, SG_6]$ is homotopic to a diffeomorphism if its normal invariant is zero.
\end{proof}

\begin{thm}\label{inertiacp3s5}
    The inertia group of $\mathbb{C}P^3\times\mathbb{S}^5$ is $\mathbb{Z}/{62}.$
\end{thm}
\begin{proof}
    Let $[\Sigma]\in I(\mathbb{C}P^3\times\mathbb{S}^5)$ and $h_{\Sigma}: (\mathbb{C}P^3\times\mathbb{S}^5)\#\Sigma\to\mathbb{C}P^3\times\mathbb{S}^5$ be the canonical homeomorphism. Then $\eta^{Diff}(h_{\Sigma})=0$ and there exists a (oriented) diffeomorphism $g: (\mathbb{C}P^3\times\mathbb{S}^5)\#\Sigma\to \mathbb{C}P^3\times\mathbb{S}^5.$ Hence $g\circ h_{\Sigma}^{-1}\in \mathcal{E}(\mathbb{C}P^3\times\mathbb{S}^5).$ Now by composition formula of normal invariant, we have $$\eta^{Diff}(g\circ h_{\Sigma}^{-1})=\eta^{Diff}(g)+(g^{-1})^*\eta^{Diff}(h_{\Sigma})=0.$$
    Therefore, by Lemma \ref{dichotomycp3s5}, $g\circ h_{\Sigma}^{-1}$ is homotopic to a diffeomorphism, implying that $[\Sigma]\in I_h(\mathbb{C}P^3\times\mathbb{S}^5).$ Consequently, we have from Theorem \ref{homotopyinertiacp3s5} that $I(\mathbb{C}P^3\times\mathbb{S}^5)= I_h(\mathbb{C}P^3\times\mathbb{S}^5)=\mathbb{Z}/{62}.$
\end{proof}

\begin{thm}\label{classificationcp3s5}
    Let $N$ be a closed, smooth, oriented manifold homeomorphic to $\mathbb{C}P^3\times\mathbb{S}^5.$ Then it is (oriented) diffeomorphic to one of the manifolds $\mathbb{C}P^3\times\mathbb{S}^5$ or $(\mathbb{C}P^3\times\mathbb{S}^5)\# \Sigma,$ where $[\Sigma]\in \mathbb{Z}/{2^4}\subset \Theta_{11}$ is an exotic $11$-sphere.  
\end{thm}
\begin{proof}
    Let $[(N,g)]$ be an element of $\mathcal{C}(\mathbb{C}P^3\times\mathbb{S}^5).$ Then it follows from Corollary \ref{concorMtimesSk}, Proposition \ref{CP3computation} {(iii)} and Lemma \ref{psicp3sk} that $\eta^{Diff}(g)$ is trivial. Now the smooth surgery exact sequence implies that $N$ is (oriented) diffeomorphic to $(\mathbb{C}P^3\times\mathbb{S}^5)\#\Sigma^{11},$ for some $[\Sigma^{11}] \in bP_{12}.$ Hence, by Lemma \ref{inertiacp3s5}, $N$ is (oriented) diffeomorphic to either $\mathbb{C}P^3\times\mathbb{S}^5$ or $(\mathbb{C}P^3\times\mathbb{S}^5)\#\Sigma,$ where $[\Sigma]\in \mathbb{Z}/{2^4}\subset \Theta_{11}.$
    
    Let $[\Sigma_1]$ and $[\Sigma_2]$ be two distinct elements of $\mathbb{Z}/{2^4}\subset \Theta_{11}.$ Then if $(\mathbb{C}P^3\times\mathbb{S}^5)\#\Sigma_1$ is diffeomorphic to $(\mathbb{C}P^3\times\mathbb{S}^5)\#\Sigma_2,~ [\Sigma_1\#\Sigma_2^{-1}]\in I(\mathbb{C}P^3\times\mathbb{S}^5).$ Hence $\mid [\Sigma_1\#\Sigma_2^{-1}]\mid$ divides $62.$ Since $[\Sigma_1], [\Sigma_2] \in \mathbb{Z}/{2^4},$ the order of $[\Sigma_1\#\Sigma_2^{-1}]$ is zero. Thus $(\mathbb{C}P^3\times\mathbb{S}^5)\#\Sigma_1$ cannot be diffeomorphic to $(\mathbb{C}P^3\times\mathbb{S}^5)\#\Sigma_2.$ 

\end{proof}
Next, we determine the diffeomorphism classification of all closed, oriented, smooth manifolds homeomorphic to $\mathbb{C}P^3\times\mathbb{S}^6.$
\begin{lemma}\label{pi6fs1}
    Let $f: \mathbb{C}P^3\times\mathbb{S}^6\to \mathbb{C}P^3\times\mathbb{S}^6$ be a self-homotopy equivalence induced by an element of $\pi_6(F_{\mathbb{S}^1}(\mathbb{C}^4)).$ If the normal invariant $\eta^{Diff}(f)$ is nontrivial, then $\eta^{Diff}(f) \in \pi_6(G/O) \subset [\mathbb{C}P^3\times\mathbb{S}^6, G/O].$
\end{lemma}
\begin{proof}
    Let $f\in \mathcal{E}(\mathbb{C}P^3\times\mathbb{S}^6)$ induced from $t\in \pi_6(F_{\mathbb{S}^1}(\mathbb{C}^4)).$ Consider the following commutative diagram 
     \begin{center}
        \begin{tikzcd}
            &{\pi_6(F_{\mathbb{S}^1}(\C^4))}\arrow[dl,"\Psi"']\arrow[dr,"\eta_{rel}^{Diff}\circ \Psi"]\\
            {\mathcal{S}_6^{Diff}(\C P^3)}\arrow[d,"\Gamma"']\arrow[rr,"\eta_{rel}^{Diff}"]&&{[\Sigma^6 \mathbb{C}P^3_{+}, G/O]\cong \z/3\oplus\z\oplus\z\oplus\pi_6(G/O)}\arrow[d,"p^*"]\\
            {\mathcal{S}^{diff}(\C P^3\times\s^6)}\arrow[rr,"\eta^{Diff}"']&&{[\C P^3\times\s^6, G/O]}
        \end{tikzcd}
    \end{center}
 where $\mathcal{S}_6^{Diff}(\mathbb{C}P^3)$ denotes the relative structure set $ \mathcal{S}^{Diff}(\mathbb{C}P^3\times\mathbb{D}^6 ~\mathrm{rel}~ \mathbb{C}P^3\times\mathbb{S}^5)$ \cite{Luck_surgery}, $\Gamma: \mathcal{S}_6^{Diff}(\mathbb{C}P^3)\to \mathcal{S}^{Diff}(\mathbb{C}P^3\times\mathbb{S}^6)$ is the canonical mapping obtained by extension via diffeomorphism, $p^* :[\Sigma^6 \mathbb{C}P^3_{+}, G/O]\to [\mathbb{C}P^3\times\mathbb{S}^6, G/O]$ induced from the quotient map $p: \mathbb{C}P^3\times \mathbb{S}^6 \to \Sigma^6 \mathbb{C}P^3_{+} $ is injective from \eqref{note2.3}, and $[\Sigma^6 \mathbb{C}P^3, G/O]\cong \mathbb{Z}/3\oplus\mathbb{Z}\oplus\mathbb{Z},$ derived from the long exact sequence induced from \eqref{longexactCP3new} along $G/O.$ Given $\eta^{Diff}(f)$ is nontrivial, $\eta_{rel}^{Diff}(f)$ is also nontrivial from the above commutative diagram. Since $\pi_6(F_{\mathbb{S}^1}(\mathbb C^4))\cong \pi_6^s(\Sigma \mathbb{C}P^{\infty})\cong \mathbb Z/2$ \cite{JBRS74, muk2} and both $\Psi: \pi_6(F_{\mathbb{S}^1}(\mathbb{C}^4))\to \mathcal{S}_6^{Diff}(\mathbb{C}P^3),~ \eta_{rel}^{Diff}: \mathcal{S}_6^{Diff}(\mathbb{C}P^3)\to [\Sigma^6\mathbb{C}P_{+}^3, G/O]$ are group homomorphisms, $\eta_{rel}^{Diff}\circ\Psi(t) \in \pi_6(G/O) \subset [\Sigma^6 \mathbb{C}P^3_{+}, G/O],$ and hence $\eta^{Diff}(f)\in \pi_6(G/O) \subset [\mathbb{C}P^3\times\mathbb{S}^6, G/O].$ This completes the proof.
\end{proof}

Let $h:\widetilde{N}\to \mathbb{C}P^3\times\mathbb{S}^6$ be a homeomorphism such that $\eta^{Diff}(h)$ is nontrivial and belongs to $\pi_{10}(G/O)_{(3)}\subset [\mathbb{C}P^3\times\mathbb{S}^6,G/O],$ where $\widetilde{N}$ is a closed, oriented, smooth $12$-dimensional manifold.  
 
\begin{theorem}\label{classificationCP3S6}
    Let $N$ be a closed, oriented, smooth manifold homeomorphic to $\mathbb{C}P^3\times\mathbb{S}^6.$ Then it is diffeomorphic to exactly one of the manifolds $\mathbb CP^3\times \mathbb S^6, ~\widetilde{N}.$
\end{theorem}
\begin{proof}
    Let $[(N,g)]$ be an element of $\mathcal{C}(\mathbb{C}P^3\times\mathbb{S}^6).$ If $\eta^{Diff}(g)$ is trivial, then the surgery exact sequence for $\mathbb{C}P^3\times\mathbb{S}^6$ implies that $(N, g)$ and $(\mathbb{C}P^3\times\mathbb{S}^6, Id)$ represent the same element in $\mathcal{S}^{Diff}(\mathbb{C}P^3\times\mathbb{S}^6).$ 

    Let $\eta^{Diff}(g)$ be nontrivial. Then from Proposition \ref{CP3computation} {(ii)} and Lemma \ref{psicp3sk} {(vi)}, $\eta^{Diff}(g)\in (\pi_{10}(G/O))_{(3)}\subset [\mathbb{C}P^3\times\mathbb{S}^6, G/O].$ Since $\pi_6(\mathbb{C}P^3)=0,$ it follows from \cite[Proposition 2.1]{PP99} and Fact \ref{self homotopy equivalence of product} {(b)} that we need to analyze the action of self-homotopy equivalence of $\mathbb{C}P^3\times\mathbb{S}^6$ arising from $\mathcal{E}(\mathbb{S}^6),\mathcal{E}(\mathbb{C}P^3),[\mathbb{C}P^3, E_1(\mathbb{S}^6)]$ and $\pi_6(E_1(\mathbb{C}P^3))$ separately. We observe that any self-homotopy equivalence induced from $\mathcal{E}(\mathbb{S}^6)$ and $\mathcal{E}(\mathbb{C}P^3)$ is represented by a diffeomorphism. Let $f\in\mathcal{E}(\mathbb{C}P^3\times\mathbb{S}^6)$ induced from an element of $\mathcal{E}(\mathbb{S}^6)$ (respectively $\mathcal{E}(\mathbb{C}P^3)$). Then by the composition formula of normal invariant, we have $\eta^{Diff}(f\circ g)=(f^{-1})^*\eta^{Diff}(g)\neq 0.$

    Let $f\in \mathcal{E}(\mathbb{C}P^3\times\mathbb{S}^6)$ coming from an element of $\pi_6(E_1(\mathbb{C}P^3))\cong \pi_6(F_{\mathbb{S}^1}(\mathbb C^4))$ (respectively, from $[\mathbb{C}P^3, E_1(\mathbb{S}^6)]\cong [\mathbb{C}P^3, SG_7]$). Then $\eta^{Diff}(f\circ g)$ is non-zero because $\eta^{Diff}(f)\in \pi_6(G/O)\subset [\mathbb{C}P^3\times\mathbb{S}^6, G/O]$ by Lemma \ref{pi6fs1} (respectively $\eta^{Diff}(f)\in [\mathbb{C}P^3, G/O]\subset [\mathbb{C}P^3\times\mathbb{S}^6, G/O]$ by Fact \ref{normal invariant of SG}), while $\eta^{Diff}(g)\in [\Sigma^6\mathbb{C}P^3, G/O]\subset [\mathbb{C}P^3\times\mathbb{S}^6, G/O]$ as per Lemma \ref{psicp3sk} {(vi)}. Hence $N$ is not diffeomorphic to $\mathbb{C}P^3\times\mathbb{S}^6$ from the smooth surgery exact sequence of $\mathbb{C}P^3\times\mathbb{S}^6$. 

    Let $[(N_1, g_1)]$ and $[(N_2, g_2)]$ be two elements of $\mathcal{C}(\mathbb{C}P^3\times\mathbb{S}^6)$ such that $\eta^{Diff}(g_2)=- \eta^{Diff}(g_1)\neq 0.$ Let $f=id_{\mathbb{C}P^3}\times (- id_{\mathbb{S}^6}): \mathbb{C}P^3\times\mathbb{S}^6\to \mathbb{C}P^3\times\mathbb{S}^6$ be a homotopy self-equivalence induced from $-id_{\mathbb{S}^6} \in \mathcal{E}(\mathbb{S}^6).$ We claim that $(f^{-1})^*\eta^{Diff}(g_i)=- \eta^{Diff}(g_i)$ for $i=1$ and $2.$ The map $f: \mathbb{C}P^3\times\mathbb{S}^6\to \mathbb{C}P^3\times\mathbb{S}^6$ induces the map $(-id_{\mathbb{S}^6})\wedge id_{\mathbb{C}P^3}: \mathbb{S}^6\wedge \mathbb{C}P^3\to \mathbb{S}^6\wedge\mathbb{C}P^3$ such that the following diagram commutes:
   
    \begin{center}
        \begin{tikzcd}
    {\mathbb{C}P^3\times\s^6}\arrow[r,"p"]\arrow[d,"f=id_{\mathbb{C}P^3}\times (-id_{\s^6})"']&{\s^6\wedge \mathbb{C}P^3}\arrow[d,"(-id_{\mathbb{S}^6})\wedge id_{\mathbb{C}P^3}"]\\
            {\mathbb{C}P^3\times\s^6}\arrow[r,"p"']&{\s^6\wedge\mathbb{C}P^3}
        \end{tikzcd}
    \end{center}
Since $\eta^{Diff}(g_i)\in [\Sigma^6\mathbb{C}P^3, G/O],$ it follows from the above diagram that $(f^{-1})^*\eta^{Diff}(g_i)= ((-id_{\mathbb{S}^6}\wedge id_{\mathbb{C}P^3})^{-1})^* \eta^{Diff}(g_i).$ Moreover, as $\eta^{Diff}(g_i)\in \mathbb{Z}/3 \subset [\Sigma^6\mathbb{C}P^3, G/O]$ and $\psi_*: [\Sigma^6\mathbb{C}P^3, Top/O]\cong \mathbb{Z}/3 \to [\Sigma^6 \mathbb{C}P^3, G/O]$ is injective, it suffices to show that $((-id_{\mathbb{S}^6}\wedge id_{\mathbb{C}P^3})^{-1})^*: [\Sigma^6\mathbb{C}P^3, Top/O]\to [\Sigma^6\mathbb{C}P^3, Top/O]$ takes the generator $\bar{1}$ of $[\Sigma^6\mathbb{C}P^3, Top/O]$ to the generator $-\bar{1}$ of $[\Sigma^6\mathbb{C}P^3, Top/O].$ 

We note that $(-id_{\mathbb{S}^6})\wedge id_{\mathbb{C}P^3}: \mathbb{S}^6\wedge \mathbb{C}P^3\to \mathbb{S}^6\wedge \mathbb{C}P^3$ fits into the following commutative diagram:
\begin{equation}\label{wedge diagram}
    \begin{tikzcd}[column sep=3.2em]
        {\s^6\wedge \mathbb{C}P^3}\arrow[r,"id_{\s^6}\wedge q"]\arrow[d,"(- id_{\s^6})\wedge id_{\mathbb{C}P^3}"']& {\s^6\wedge (\s^4\vee \s^6)}\arrow[r,"id_{\s^6}\wedge Pr"]\arrow[d,"(-id_{\s^6})\wedge (id_{\s^4\vee\s^6})"]&{\s^6\wedge \s^4}\arrow[d,"(-id_{\s^6})\wedge id_{\s^4}"]\\
         {\s^6\wedge \mathbb{C}P^3}\arrow[r,"id_{\s^6}\wedge q"']& {\s^6\wedge (\s^4\vee \s^6)}\arrow[r,"id_{\s^6}\wedge Pr"'] &{\s^6\wedge \s^4}
    \end{tikzcd}
\end{equation}
where the map $(-id_{\mathbb{S}^6})\wedge id_{\mathbb{S}^4}: \mathbb{S}^6\wedge \mathbb{S}^4\to \mathbb{S}^6\wedge\mathbb{S}^4$ is homotopic to the map $-id_{\mathbb{S}^{10}}: \mathbb{S}^{10}\to \mathbb{S}^{10}$ for degree reason. Therefore the diagram \eqref{wedge diagram} gives rise to the subsequent commutative diagram along $Top/O.$
\begin{center}
    \begin{tikzcd}
        {(\pi_{10}(Top/O))_{(3)}}\arrow[rrrr,"(id_{\s^6}\wedge q)_{(3)}^*\circ (id_{\s^6} \wedge Pr)_{(3)}^*"]\arrow[d,"(- id_{\s^{10}})_{(3)}^*"']&&&&{[\s^6\wedge \mathbb{C}P^3, Top/O]_{(3)}}\arrow[d,"(- id_{\s^6}\wedge id_{\mathbb{C}P^3})_{(3)}^*"]\\
         {(\pi_{10}(Top/O))_{(3)}}\arrow[rrrr,"(id_{\s^6}\wedge q)_{(3)}^*\circ (id_{\s^6} \wedge Pr)_{(3)}^*"']&&&&{[\s^6\wedge \mathbb{C}P^3,Top/O]_{(3)}}
    \end{tikzcd}
\end{center}
The map $(- id_{\mathbb{S}^{10}})_{(3)}^* = (- id_{\mathbb{S}^6}\wedge id_{\mathbb{S}^4})_{(3)}^*: \pi_{10}(Top/O)_{(3)}\to \pi_{10}(Top/O)_{(3)}$ sends the generator $\Bar{1}$ of $\pi_{10}(Top/O)_{(3)}$ to the generator $-\Bar{1}$ in $\pi_{10}(Top/O)_{(3)}.$ Consequently, from the above diagram, it is evident that  $(- id_{\mathbb{S}^6}\wedge id_{\mathbb{C}P^3})_{(3)}^*: [\Sigma^6 \mathbb{C}P^3, Top/O]_{(3)}\to [\Sigma^6 \mathbb{C}P^3, Top/O]_{(3)}$ takes the generator $\Bar{1}$ of $[\Sigma^6 \mathbb{C}P^3, Top/O]_{(3)}$ to the generator $-\Bar{1}$ in $[\Sigma^6 \mathbb{C}P^3, Top/O]_{(3)}.$ Therefore $$\eta^{Diff}(f\circ g_1)= (f^{-1})^* \eta^{Diff}(g_1)=-\eta^{Diff}(g_1)=\eta^{Diff}(g_2)$$ and $$\eta^{Diff}(f\circ g_2)= (f^{-1})^* \eta^{Diff}(g_2)=-\eta^{Diff}(g_2)=\eta^{Diff}(g_1).$$ Hence $N_1$ is diffeomorphic to $N_2.$ We take $\widetilde{N}$ to be either $N_1$ or $N_2.$ This completes the proof. 
\end{proof}


We establish the existence of a self-homotopy equivalence of $\mathbb{C}P^3\times\mathbb{S}^7$ with a nontrivial normal invariant, which is essential for the classification.
\begin{lemma}\label{nu2map}
    There exists a self-homotopy equivalence $f_{\nu^2}:\mathbb{C}P^3\times\mathbb{S}^7\to \mathbb{C}P^3\times\mathbb{S}^7$ with nontrivial normal invariant. In particular, the normal invariant of 
$f_{\nu^2}$ equals to $[\nu^3],$ where $[\nu^3]$ represents the image of $\nu^3 \in \pi_9^s$ in $\pi_9(G/O).$ 
\end{lemma}
\begin{proof}
Let $f_{\nu^2}:\mathbb{C}P^3\times\mathbb{S}^7\to \mathbb{C}P^3\times\mathbb{S}^7$ be a self-homotopy equivalence induced by the nontrivial element of $\pi_6^s\subset \pi_7(F_{\mathbb{S}^1}).$ Now, \cite[Fact 5.2 {(f)}]{SBRKAS} implies that its normal invariant is the image of $\Sigma^7 \mathbb{C}P^3_{+}\xrightarrow{\Sigma^6 t}\mathbb{S}^6\xrightarrow{\nu^2}\mathbb{S}^0$ inside $[\Sigma^7\mathbb{C}P_{+}^3,G/O]\subset [\mathbb{C}P^3\times\mathbb{S}^7,G/O].$ Since $(\nu^2\circ \Sigma^6 t)\vert_{\mathbb{S}^9}=\nu^3,$ the normal invariant $\eta^{Diff}(f_{\nu^2})\vert_{\Sigma^7 \mathbb{C}P^1}$ is the image of $\nu^3\in \pi_9^s$ in $\pi_9(G/O).$ We note from the following commutative diagram  
 \begin{center}
     \begin{tikzcd}[column sep=1em]
         {\pi_{10}^s}\arrow[r,"\eta^*"]&{\pi_{11}^s}\arrow[r,"(f_{\mathbb{C}P^2})^*"]\arrow[d,"j_*"']&{[\Sigma^7\mathbb{C}P^2, G]}\arrow[r]\arrow[d,"j_*"]&{\pi_9^s}\arrow[r,"\eta^*"]\arrow[d,"j_*"]&{\pi_{10}^s}\arrow[d,"j_*","\cong"']\\
         &{0}\arrow[r]&{[\Sigma^7\mathbb{C}P^2, G/O]\cong \mathbb{Z}/2\{[\nu^3]\}}\arrow[r]&{\pi_9(G/O)}\arrow[r,"\eta^*"]&{\pi_{10}(G/O)}
     \end{tikzcd}
 \end{center}
that the image of $j_*:[\Sigma^7\mathbb{C}P^2, G]\to [\Sigma^7\mathbb{C}P^2, G/O]$ is $\mathbb{Z}/2\{[\nu^3]\}.$ Therefore, $\eta^{Diff}(f_{\nu^2})$ restricted to $\Sigma^7 \mathbb{C}P^2$ is equal to $[\nu^3].$ Hence, $\eta^{Diff}(f_{\nu^2})$ is nontrivial and is generated by $[\nu^3],$ as $[\Sigma^7\mathbb{C}P^3, G]=[\Sigma^7\mathbb{C}P^2, G].$  
\end{proof}
\begin{prop}\label{zero normal invariant implies diffeo}
    Let $[(N,g)] \in \mathcal{C}(\mathbb{C}P^3\times\mathbb{S}^7)$ such that $\eta^{Diff}(g)=0.$ Then $N$ is (oriented) diffeomorphic to $\mathbb{C}P^3\times\mathbb{S}^7.$ 
\end{prop}
\begin{proof}
    As $\eta^{Diff}(g)=0$, it follows from the smooth surgery exact sequence of $\mathbb{C}P^3\times\mathbb{S}^7$ that the manifold $N$ is diffeomorphic to $\mathbb{C}P^3\times\mathbb{S}^7\# \Sigma$, for some $[\Sigma] \in bP_{14}.$ Since $bP_{14}=0$ \cite{hillhopkinsravenel}, $\Sigma$ is diffeomorphic to the standard sphere $\mathbb{S}^{13},$ and hence $N$ is oriented diffeomorphic to $\mathbb{C}P^3\times\mathbb{S}^7.$ 
\end{proof}

\begin{theorem}\label{classificationCP3S7}
    If a closed, oriented, smooth manifold $N$ is homeomorphic to $\mathbb{C}P^3\times\mathbb{S}^7,$ then $N$ is (oriented) diffeomorphic to $\mathbb{C}P^3\times\mathbb{S}^7.$
\end{theorem}
\begin{proof}
Let $[(N,g)]$ be an element of $\mathcal{C}(\mathbb{C}P^3\times\mathbb{S}^7).$ If $\eta^{Diff}(g)$ is trivial, then from Proposition \ref{zero normal invariant implies diffeo} $N$ is oriented diffeomorphic to $\mathbb{C}P^3\times\mathbb{S}^7.$

 Let $[(N,g)] \in \mathcal{C}(\mathbb{C}P^3\times\mathbb{S}^7)$ be such that $\eta^{Diff}(g)$ is nontrivial. Then by Lemma \ref{psicp3sk} {(vii)} and Proposition \ref{CP3computation} {(ii)},  $\eta^{diff}(g)$ is $[\nu^3].$ Now using Lemma \ref{nu2map} and composition formula for normal invariant, we obtain
\begin{align*}
        \eta^{Diff}(f_{\nu^2}\circ g)= & \eta(f_{\nu^2})+ (f_{\nu^2}^{-1})^* \eta^{Diff}(g)\\
         = & [\nu^3]\pm [\nu^3]=0,
    \end{align*}
Therefore by Proposition \ref{zero normal invariant implies diffeo}, $N$ is (oriented) diffeomorphic to $\mathbb{C}P^3\times\mathbb{S}^7.$
\end{proof}



\begin{thebibliography}{100}
		\bibitem{SBRKAS}
           S.~Basu,~R.~ Kasilingam, and ~A.~ Sarkar, {\em Smooth Structures on $ M \times\mathbb {S}^ k,$} arXiv preprint arXiv:2402.18914 (2024).
           
		\bibitem{JBRS74}
           J.~ C.~ Becker, and ~R.~ E.~ Schultz, {\em Equivariant function spaces and stable homotopy theory I,} Comment. Math. Helv \textbf{49}(1974): 1-34.
           
		\bibitem{bel}
		I.~Belegradek,~S.~Kwasik, and~R.~ Schultz, {\em Codimension two souls and cancellation phenomena,} Advances in Mathematics \textbf{275} (2015): 1-46.
		
		\bibitem{gw2}
		G.~Brumfiel, {\em On the homotopy groups of $BPL$ and $PL/O$,} Annals of Mathematics(2) \textbf{88} (1968): 291-311.
 
		\bibitem{GB69}
        G.~Brumfiel, {\em On the homotopy groups of $BPL$ and $PL/O.$ II,} Topology \textbf{8} (1969): 305-311.
  
		\bibitem{gw}
		G.~Brumfiel, {\em Homotopy equivalences of almost smooth manifolds,} Comm. Math. Helv. \textbf{46} (1971): 381--407.

   
              \bibitem{DCCN20}
              D.~Crowley, and~ C.~ Nagy, {\em The smooth classification of 4-dimensional complete intersections,} arXiv preprint arXiv:2003.09216 (2020).
		
	
		\bibitem{Fuji}
            M.~ Fuji, {\em $KO$-groups of projective spaces,} Osaka J. Math. \textbf{4}(1967): 141-149.
	
		\bibitem{harper}
		J.~ R.~Harper, {\em Secondary Cohomology Operations,} Graduate studies in mathematics, Vol.\textbf{49}, Amer. Math. Soc.(2002), Providence, RI.
		
	

        \bibitem{hillhopkinsravenel}
         M.~ A.~Hill,~M.~ J.~ Hopkins, and~ D.~ C.~ Ravenel, {\em On the Nonexistence of Elements of Kervaire Invariant One,} Annals of Mathematics \textbf{184}, no. 1 (2016): 1–262. 

  
         \bibitem{hirsch}
          M.~ W.~ Hirsch, and~B. Mazur, {\em Smoothings of piecewise linear manifolds,} Annals of Math. \textbf{80}, Princeton University Press, Princeton, N. J.; University of Tokyo Press, Tokyo (1974).
		
	\bibitem{rhuang}
            R.~Huang, {\em Homotopy of gauge groups over high-dimensional manifolds,} Proceedings of the Royal Society of Edinburgh Section A: Mathematics \textbf{152}, no. 1 (2022): 182-208.
	
		\bibitem{kubo}
		K.~ Kawakubo, {\em On the inertia groups of homology tori,} Journal of the Mathematical Society of Japan \textbf{21}, no. 1 (1969): 37-47.
  
		 \bibitem{Kawakubo69}
            K.~Kawakubo, {\em Smooth structures on $\s^{p}\times \s^{q},$} Proceedings of the Japan Academy \textbf{45} (1969): 215-218.
            
		
		
		\bibitem{kirby}
		R.~C.~Kirby, and~L.~C.~Siebenmann, {\em Foundational essays on topological manifolds, smoothings, and triangulations,} Ann. of Math. Stud. \textbf{88} (1977), Princeton University Press.
        
		\bibitem{kuiperlashof66}
		N.~H.~Kuiper, and~R.~ K.~ Lashof, {\em Microbundles and bundles: I. Elementary theory,} Inventiones Mathematicae \textbf{1}, no. 1 (1966): 1-17.
  
		\bibitem{AKupers23}
           A.~Kupers, {\em The first two k-invariants of Top/O,} arXiv preprint \url{arXiv:2303.12854 (2023).}
           
		\bibitem{lashofrothenberg65}
		R.~Lashof, and~M.~ Rothenberg, {\em Microbundles and smoothing,} Topology \textbf{3}, no. 4 (1965): 357-388.
		
	   \bibitem{LiZhu24}
       P.~Li, and ~Z.~ Zhu, {\em The Homotopy Types of Suspended Simply-connected $6 $-manifolds,} arXiv preprint arXiv:2402.02450 (2024).

       \bibitem{Luck_surgery}
       W.~Lück, and T. Macko, {\em Surgery theory: foundations,} Vol. 362. Springer Nature, 2024.
		
		\bibitem{milgram}
		I.~Madsen, and ~J.~R.~Milgram, {\em The classifying spaces for surgery and cobordism of manifolds,} Ann. of Math. Stud. {\bf 92}, Princeton University Press (1979).

           \bibitem{MilSta}
           J.~W.~Milnor, and~ J.~ D.~ Stasheff, {\em Characteristic classes} No. 76. Princeton University Press, 1974.
		
		\bibitem{tangora}
		R.~E.~Mosher, and~M.~ C.~ Tangora,{\em  Cohomology operations and applications in homotopy theory,} Courier Corporation, 2008.
		
		\bibitem{muk2}
		J.~ Mukai, {\em The $S^ 1$-transfer map and homotopy groups of suspended complex projective spaces,} Mathematical Journal of Okayama University \textbf{24}, no. 2 (1982): 179-200.
         \bibitem{PP99}
         P.~Pavešić, {\em Self-homotopy equivalences of product spaces,} Proceedings of the Royal Society of Edinburgh \textbf{129A}, no. 1 (1999): 181-197.
  
		\bibitem{SO72}
     S.~ Oka, {\em The stable homotopy groups of spheres. II,} Hiroshima Mathematical Journal \textbf{2}, no.1 (1972), 99-161.
     
	
		
		\bibitem{ravenel}
		D.~ C.~ Ravenel, {\em Complex cobordism and stable homotopy groups of spheres,}  AMS Chelsea Publishing, 2003.
  
	
           
	\bibitem{rudyak15}
	Y.~Rudyak, {\em Piecewise linear structures on topological manifolds,} World Scientific, 2016.

       \bibitem{AS24}
       A.~Sarkar, {\em Smooth Structures on the product of $3$-connected $8$-manifolds with spheres,} arXiv preprint arXiv:2502.20736 (2025).
		
		\bibitem{schultz}
		R.~Schultz, {\em On the inertia group of a product of spheres,} Transactions of the American Mathematical Society \textbf{156}, (1971): 137-153.
  
		\bibitem{RS87}
       R.~Schultz, {\em Homology spheres as stationary sets of circle actions,} Michigan Mathematical Journal \textbf{34}, no. 2 (1987): 183-200.
       
	
		\bibitem{JS69}
            J.~L.~Shaneson, {\em Wall's surgery obstruction groups for $G\times Z,$} Annals of Mathematics \textbf{90}, no.2 (1969): 296-334.
            

          \bibitem{Tanaka}
          R.~ Tanaka, {\em On trivialities of Stiefel-Whitney classes of vector bundles over iterated suspension spaces,} Homology, Homotopy and Applications, Homology Homotopy Appl. \textbf{12}(1) (2010): 357-366.
 
		\bibitem{toda}
		H.~Toda, {\em Composition methods in homotopy groups of spheres,} Ann. of Math. Stud. {\bf 49}, Princeton Univ. Press, Princeton, N. J. (1962).
  
		\bibitem{wall}
		C.~ T.~C.~Wall, {\em Classification problems in differential topology. V: On certain 6-manifolds,} Inventiones mathematicae \textbf{1}, no. 4 (1966): 355-374.

           \bibitem{wallsurgerybook}
            C.~ T.~C.~ Wall, {\em Surgery on compact manifolds,} volume 69 of Mathematical Surveys and Monographs. American Mathematical Society, Providence, RI, second edition, 1999. Edited and with a foreword by A. A. Ranicki.
	\end{thebibliography}
\end{document}